\documentclass[11pt]{article}
\usepackage[margin=1in]{geometry}
\usepackage{graphicx,xcolor,cite,mathtools}
\usepackage{amssymb,amsmath,amsthm,mathrsfs,paralist,bm,esint,setspace}
\usepackage{marginnote}
\usepackage{cases}
\RequirePackage[colorlinks,citecolor=blue,urlcolor=blue]{hyperref}

\newcommand{\R}{{\mathbb{R}}}
\newcommand{\E}{\mathrm{E}}
\newcommand{\HH}{\mathcal{H}}
\renewcommand{\P}{\mathrm{P}}
\renewcommand{\d}{\mathrm{d}}
\newcommand{\<}{\langle}
\renewcommand{\>}{\rangle}
\newcommand{\e}{\mathrm{e}}
\newcommand{\Var}{\text{\rm Var}}
\newcommand{\lip}{\text{\rm Lip}}
\DeclareMathOperator{\Cov}{\text{\rm Cov}}
\newcommand{\ep}{\varepsilon}
\newcommand{\cS}{\mathcal{S}}
\newcommand{\Norm}[1]{\left\|  #1   \right\|}

\input cyracc.def

\newtheorem{stat}{Statement}[section]
\newtheorem{proposition}[stat]{Proposition}

\newtheorem{corollary}[stat]{Corollary}
\newtheorem{theorem}[stat]{Theorem}
\newtheorem{lemma}[stat]{Lemma}

\theoremstyle{definition}
\newtheorem{definition}[stat]{Definition}

\newtheorem{remark}[stat]{Remark}

\newtheorem{example}[stat]{Example}

\numberwithin{equation}{section}
\usepackage{xcolor}
\newtheorem{thmprime}{Theorem}

\newtheorem{propprime}{Proposition}

\title{Central limit theorems\\for spatial averages of the stochastic heat
	equation via Malliavin-Stein's method\thanks{%
	Research supported in part by  NSF grants DMS-1811181 (D.N.) and DMS-1855439 (D.K.), and FNR grant APOGee (R-AGR-3585-10) at University of Luxembourg (F.P.).}}
\author{Le Chen\\Emory University\\\texttt{le.chen@emory.edu}\\
	\and
		Davar Khoshnevisan\\University of Utah\\\texttt{davar@math.utah.edu}\\
	\and\and
		David Nualart\\University of Kansas\\\texttt{nualart@ku.edu}\\
	\and
		Fei Pu\\University of Luxembourg\\\texttt{fei.pu@uni.lu}
	}
\date{\today}                                           

\begin{document}

\maketitle

\begin{abstract}
Suppose that  $\{u(t\,, x)\}_{t >0, x \in\R^d}$  is the
solution to a $d$-dimensional stochastic heat equation driven
by a Gaussian noise that is white in time and has a spatially homogeneous
covariance that satisfies Dalang's condition.
 The purpose of this paper is to establish quantitative central limit theorems for
spatial averages of the form
$N^{-d} \int_{[0,N]^d} g(u(t\,,x))\, \d x$, as $N\to\infty$,
where $g$ is a Lipschitz-continuous function or belongs to a class of
locally-Lipschitz functions, using a combination of the Malliavin calculus
and Stein's method for normal approximations.
 Our results include a central limit theorem for the
{\it Hopf-Cole} solution to KPZ equation.
We also establish a functional central limit theorem for these spatial averages.
\end{abstract}

\bigskip\bigskip

\bigskip

\bigskip

\noindent{\it Running head:} CLT for SPDEs.

{\hypersetup{linkcolor=black}
\tableofcontents
}

\section{Introduction}

Consider the following stochastic heat equation on $\R^d$:
\begin{equation}\label{E:SHE}
	 \frac{\partial u}{\partial t}(t\,,x)
	 = \tfrac12\Delta u(t\,,x) + \sigma(u(t\,,x))\eta(t\,,x)
	 \qquad \text{for all }(t\,,x) \in (0\,,\infty)\times \R^d,
\end{equation}
subject to $u(0\,,x)=1$  for all  $x\in \R^d$.
The diffusion coefficient $\sigma:\R\to\R$ is  assumed to be  nonrandom
and Lipschitz continuous.
Moreover, we avoid trivialities by  always assuming that
\begin{equation}\label{E:sigma(1)}
	\sigma(1)\ne 0.\footnote{%
	Indeed, if $\sigma(1)=0$ then it can be checked that the the solution is degenerate: $u(t\,,x)\equiv1$
	for all $(t\,,x)\in\R_+\times\R^d$.}
\end{equation}
The noise term $\eta$ denotes a centered, generalized Gaussian random field such that
\begin{align} \label{E:eta}
	\Cov[\eta(t\,,x)\,,\eta(s\,,y)] = \delta_0(t-s)f(x-y)
	\qquad\text{for all $s,t\ge0$ and $x,y\in\R^d$},
\end{align}
for a nonnegative-definite tempered Borel measure $f$ on $\R^d$ that we fix
throughout.   More formally, this means that the Wiener-integrals
\begin{equation}\label{E:W}
	W_t(\phi):=\int_{(0,t)\times\R^d}
	\phi(x)\,\eta(\d r\,\d x)
	\qquad[t\ge0, \phi\in\mathscr{S}(\R^d)]
\end{equation}
define a centered Gaussian random field with covariance
\[
	\Cov\left[W_s(\phi_1)\,,W_t(\phi_2)\right] = (s\wedge t)
	\< \phi_1\,,\phi_2*f \>_{L^2(\R^d)}
	\qquad\text{for all $s,t\ge0$ and $\phi_1,\phi_2\in\mathscr{S}(\R^d)$}.
\]
Note, among other things,
that $\{ W_t\}_{t\ge 0} $ is an  infinite dimensional Brownian motion.

Recall that the Fourier transform
$\widehat{f}$ of the measure $f$ is a tempered Borel measure on $\R^d$.
We assume here and throughout that $\widehat f$ satisfies the integrability condition\footnote{%
	To be concrete, the Fourier transform is normalized
	so that $\widehat{h}(z) = \int_{\R^d}\e^{ix\cdot z}h(x)\,\d x$ for all $h\in L^1(\R^d)$
	and $z\in\R^d$.} -- {\it Dalang's condition}:
\begin{equation}\label{E:Dalang}
	\int_{\R^d}\frac{\widehat{f}(\d z)}{1+\|z\|^2}<\infty.
\end{equation}
For some results, we will need the following reinforced version of Dalang's condition:
\begin{equation} \label{E:Dalang2}
	\int_{\R^d} \frac{\widehat{f}(\d z)}{(1+\|z\|^2)^{1-\alpha}}<\infty, \qquad
	\text{for some $\alpha \in (0,1]$}.
\end{equation}
Dalang \cite{Dalang1999} has proved that, because of \eqref{E:Dalang}, the stochastic heat equation
\eqref{E:SHE} has a mild solution $u=\{ u(t\,,x)\}_{t\ge 0, x\in \R^d}$
that is the unique predictable random field such that
\begin{equation}\label{E:moments}
	\adjustlimits
	\sup_{t\in [0,T]}\sup_{x\in \R^d} \E\left( | u(t\,,x)|^k \right) <\infty
	\qquad\text{for every $T>0$ and $k\ge2$}.
\end{equation}
Moreover, ``mild'' refers to the fact that $u$ satisfies the evolution equation
\begin{equation}\label{E:mild}
	u(t\,,x)= 1+ \int_{(0,t)\times\R^d} \bm{p}_{t-s}(x-y) \sigma(u(s\,,y))\, \eta(\d s\, \d y)
	\qquad\text{for all $t>0$ and $x\in\R^d$},
\end{equation}
where  $\bm{p}$ denotes the heat kernel on $(0\,,\infty)\times\R^d$; that is,
\[
	\bm{p}_t(x) = \frac{1}{(2\pi t)^{d/2}}\exp\left( - \frac{\|x\|^2}{2t}\right)
	\qquad[t>0, x\in \R^d].
\]
The mapping $(t\,,x) \mapsto u(t\,,x)$   is continuous in   $L^k(\Omega)$  on
$[0\,,\infty)\times\R^d$ for every  $k\ge2$; see Walsh \cite{Walsh}
for the case that $f=\delta_0$, and Dalang \cite{Dalang1999} and its method for
the general case.

In Ref.\ \cite{CKNP} we  observed that
the spatial random field $u(t)=\{u(t\,,x)\}_{x\in\R^d}$ is stationary for every $t\ge0$,
and we proved that $u(t)$ is an ergodic random field for every $t>0$
provided that $\hat{f} \{0\} =0$.
In particular, if the spectral measure $\widehat{f}$ does not have an atom at $0$,
then the ergodic theorem implies that, for all $t>0$ and all measurable functions   $g: \R \rightarrow \R$
such that  $\E[|g(u(t\,,0))|]<\infty$,
\begin{equation}\label{E:ergodic}
    \lim_{N\to\infty} \frac{1}{N^d}\int_{[0,N]^d} g(u(t\,,x))\,\d x= \E[g(u(t\,,0))]
    \qquad\text{a.s.}.
\end{equation}
A natural question is to determine whether
\eqref{E:ergodic} has a matching central limit theorem (CLT).  That is,
we would like to  establish the convergence in distribution of
$N^{d/2} \mathcal{S}_{N,t} (g)$ to a normal law as $N$ tends to infinity,
where
\begin{equation} \label{E:SN}
	\mathcal{S}_{N,t} (g)=
	N^{-d}  \int_{[0,N]^d} g(u(t\,,x))\, \d x - \E [ g(u(t\,,0))]
	\qquad[N,t>0].
\end{equation}

In a recent paper \cite{CKNP2}, we have derived such a CLT under the following assumption:
\begin{equation}\label{E:f_finite}
	0< f(\R^d)<\infty,
\end{equation}
when $g\in \lip$, where ``$\lip$'' denotes the collection of all real-valued
Lipschitz-continuous functions on $\R$.
Our proof rested on  Poincar\'e-type inequalities and Malliavin's calculus,
as well as compactness arguments and L\'evy's characterization of Brownian motion.
The positivity of the total mass of $f$ [in \eqref{E:f_finite}]
merely ensures nontriviality. And the
finite-mass condition on $f$ turns out to be optimal.
Moreover, we  considered, instead of only CLTs for
$N^{-d}\int_{[0,N]^d}g(u(t\,,x))\,\d x$,  functional CLTs for $N^{-d} \int_{\R^d} \psi(x/N) g(u(t\,,x))\,\d x$
where $\psi:\R^d\to\R$ ranges over a large class of nice functions.

In the present paper we appeal to the {\it Malliavin-Stein method} in
order to  establish convergence of
$N^{d/2} \mathcal{S}_{N,t} (g)$ to a normal law in the total variation distance,
which we denote by $d_{\rm TV}$ throughout.
As was initiated by Nourdin and Peccati \cite{NP09,NP}, the combination of Malliavin's calculus
with Stein's method for normal approximations can provide an effective method for deriving quantitative
CLTs for functionals of Gaussian random fields.
Moreover, we will establish such central limit theorems not only for $g\in \lip$, but also for certain locally Lipschitz functions such as
\begin{align}\label{E:Example}
	g(u) = \log(u)\quad\text{and}\quad g(u)=u^\alpha\quad
	\text{for $\alpha\ne 0$ \text{ and $u>0$}.}
\end{align}
In particular, our results will cover the case when $g(u)=\log(u)$ (with $\sigma(u)=u$, $f=\delta_0$, $d=1$),
in which case $g(u)$ is called the {\it Hopf-Cole} solution or {\it height function} of the corresponding {\it
Kardar-Parisi-Zhang} (KPZ) equation \cite{KardarParisiZhang1986}.

Huang et al  \cite{HuangNualartViitasaari2018} were the first to use the
Malliavin-Stein method in order to  study spatial averages of parabolic  SPDEs.
They ({\it ibid}.) studied  the case that $d=1$ and $f=\delta_0$ --- that is, the stochastic heat equation
in $1+1$ dimension, driven by
space-time white noise --- and  proved that, in the case that $g(u)=u$,
\begin{equation}\label{E:HNV18}
	d_{\rm TV}\left( N^{1/2}\mathcal{S}_{N,t} (u) ~,~X\right) = O(1/\sqrt N)
	\qquad\text{as $N\to\infty$},
\end{equation}
for every fixed $t>0$, where $X=X(t)$ has a centered normal distribution.
Huang et al \cite{HuangLeNualart2017,HuangNualartViitasaariZheng2019} study the case that
$d\ge1$ and $f(\d x)= \| x\|^{-\beta}\,\d x$ for some $\beta\in(0\,,d\wedge2)$. This
Riesz-type covariance form does not satisfy condition \eqref{E:f_finite}. Huang et al ({\it ibid.})
establish a CLT though with a nonstandard normalization that depends on the numerical value of
the Riesz kernel index $\beta$.
More recently,
Nualart and Zheng \cite{NZ} use Wiener-chaos expansions to
derive CLTs in the case that $\sigma(u)=u$ and $g(u)=u$.

Let us next describe some of the highlights of this paper. Precise formulations can be found in the next section.

Theorem  \ref{T:TV} states that
if the variance of $N^{d/2} \mathcal{S}_{N,t} (g)$  converges to a strictly positive real number,
then $N^{d/2} \mathcal{S}_{N,t} (g)$ converges in total variation to a normal law.
The proof is based on the Malliavin-Stein approach, as well as on an abstract ergodic theorem
for functionals of the underlying Gaussian noise $\eta$.
Sufficient conditions for the limit variance to be nonzero are given in Proposition \ref{P:B_LB}.

Theorem \ref{T:FCLT}  provides a functional version, in the time variable, of
some of our recent work \cite[Theorem 1.1]{CKNP2}.
This result is a non-trivial extension of a functional CLT of Huang et al. in \cite{HuangNualartViitasaari2018},
valid for SPDEs that are driven by space-time white noise.

We are able to explore the rate of convergence in total variation of $N^{d/2}\mathcal{S}_{N,t}(g)$ to
a normal law in two special cases:
\begin{compactenum}
\item In Theorem  \ref{T:TV_Bnd}  we study the case that $g(v)=v$
	for all $v\in\R$, and find bounds on the total variation distance
	between $\mathcal{S}_{N,t} (g)$, normalized by its standard
	deviation, and  ${\rm N}(0\,,1)$; and
\item In Theorem \ref{T:PAM} we restrict our attention to a general $g$, but study the rate-of-convergence problem  for
	the  so-called {\it parabolic Anderson model}, which is \eqref{E:SHE} in the case that
	$\sigma(z)=z$ for all $z\in\R$.
\end{compactenum}
In Theorems \ref{T:TV}, \ref{T:FCLT}  and  \ref{T:PAM},  $g$ can be either a globally Lipschitz function
or certain locally Lipschitz function that includes examples in \eqref{E:Example}.
The proofs of these theorems can be found respectively in Sections
\ref{S:TV}, \ref{S:FCLT}, \ref{S:Bound} and \ref{S:PAM}.

\bigskip

In this paper and its companion papers \cite{CKNP,CKNP2,CKNP3} we
have presented several new ways of establishing
CLTs for functionals of infinitely-many particle systems based on Malliavin's calculus, compactness,
Poincar\'e inequalities,  the Malliavin-Stein method, and Clark-Ocone formulas.
Previous, more established methods, for deriving
CLTs for particle systems include:
\begin{compactitem}
\item The use of martingale CLTs (see, for example, Deuschel \cite{Deuschel});
\item CLTs for strongly mixing processes (see, for example, the monograph of Bradley \cite{Bradley});
	and
\item Techniques based on association,
	a notion that we recall in the appendix (see, for example, Newman and Wright \cite{NewmanWright} with various
	extensions that can be found in Newman \cite{Newman}, and Rao \cite{Rao12}).
\end{compactitem}
We can compare the techniques of this paper and its companions ({\it ibid.})
with the above methods in the context of SPDEs and systems
of interacting SDEs as follows:
\begin{compactitem}
\item Our methods provide significant improvements over those that appeal to martingale CLTs (see
	\cite{CKNP3}, and  consider adapting the methods of the present paper to the semi-discrete setting
	\cite{Deuschel} for still dramatic improvements such as those hinted at in Corollary
	\ref{C:T_CLT} below);
\item We do not know how to establish strong mixing for SPDEs except possibly when $\sigma$ is constant.
	In that very simple case, the solution to \eqref{E:SHE} is a Gaussian process and one can try to adapt the existing
	ergodic theory (see for example Helson and Sarason \cite{HS}, and Dym and McKean \cite{DymMcKean} for
	an overview) to the present setting.
	We warn however that even that effort will likely require a fair amount of work; and
\item We can try to implement association techniques to derive CLTs for SPDEs (instead of the present methods).
	It is known that the solution to \eqref{E:SHE} is associated in a few  special instances:
	Corwin and Quastel \cite{CorwinQuastel13} (see also Corwin and Ghosal \cite{CorwinGhosal18})
	observed that $u$ is associated when $d=1$, $f=\delta_0$, and $\sigma(u)=u$; and
	a theorem of Pitt \cite{Pitt} implies that $u$ is associated whenever $\sigma$ is a constant.
	The strongest association result that we are able to prove requires a good deal more effort,
	and yet only states  that: \emph{$u$ is associated provided that $\sigma(u)$ does not change sign};
	see Theorem \ref{T:assoc}.
	As we shall see, CLTs for \eqref{E:SHE} do not require that $\sigma(u)$ does not change sign.
	Therefore, we will not pursue association ideas vigorously here.
\end{compactitem}

Throughout, $\mathcal{F}=\{\mathcal{F}_t\}_{t\ge0}$ denotes the Brownian filtration
generated by the infinite dimensional Brownian motion $\{W_t\}_{t\ge0}$, defined in \eqref{E:W}.
As is customary, we assume that $\mathcal{F}$ is augmented in the usual way.
For every $Z\in L^k(\Omega)$, we write $\|Z\|_k$ instead of the more cumbersome
$\|Z\|_{L^k(\Omega)}=\{\E(|Z|^k)\}^{1/k}$.
We use $\mathrm{N}(0\,,1)$ to denote the standard normal distribution.
For every $g:\R\mapsto \R$ we write
\[
	\lip(g) = \sup_{-\infty<a<b<\infty}
	\frac{|g(b)-g(a)|}{|b-a|}.
\]
Thus, $g\in\lip$ if and only if $\lip(g)<\infty$.

\section{Main results} \label{S:Main}

Let $\mathcal{S}_{N,t} (g)$ denote the random variable defined in (\ref{E:SN}) for every $t\ge0$, $N>0$,  and any measurable function $g: \R \rightarrow \R$.
Throughout, we will write, for $t_1,t_2 \ge 0$,
\begin{align} \label{E:Var_N}
	\mathbf{B}_{N,t_1,t_2}(g) := {\rm Cov} \left( N^{d/2} \mathcal{S}_{N,t_1} (g)
	\,, N^{d/2} \mathcal{S}_{N,t_2} (g) \right)
\end{align}
and  for  $t_1=t_2 =t$ we put $\mathbf{B}_{N,t}(g)=\mathbf{B}_{N,t,t}(g)$.
We have proven in \cite[Proposition 5.2]{CKNP2}\footnote{
	This result is proved in  \cite{CKNP2} in the case $t_1=t_2$ and the proof in the case $t_1  \not =t_2$ is analogous.}
that,  when $g\in \lip$,
\begin{align}\label{E:Var_Lim}
	\lim_{N\to\infty}\mathbf{B}_{N,t_1,t_2} (g)
	= \mathbf{B}_{t_1, t_2}(g)
	:= \int_{\R^d} {\rm Cov} \left[ g(u(t_1\,,x))~,~ g(u(t_2\,,0)) \right] \d x,
\end{align}
and
\begin{equation} \label{E:Var_finite}
 	\int_{\R^d}  \left|{\rm Cov} [ g(u(t_1\,,x))~,~ g(u(t_2\,,0))] \right| \d x <\infty;
\end{equation}
see \cite[Lemma 5.1]{CKNP2}. Denote $\mathbf{B}_{t}(g) =\mathbf{B}_{t, t}(g) $.

We are also interested in the cases when the above Lipschitz functions $g$ are replaced by certain locally Lipschitz ones.
In particular, assuming $\sigma(0)=0$, the solution $u(t,x)$ is strictly positive almost surely and it turns out that the following set of conditions will be sufficient for this purpose:
\begin{subnumcases}{ \label{E:cond_g} }
	g\in C^2\left(0,\infty\right),  \label{E:cond_g1}\\
	\text{$g'$ is either strictly positive or strictly negative,} \label{E:cond_g2} \\
	\text{$g''$ is monotone over $(0,\infty)$,} \label{E:cond_g3} \\
	\E\left( \left|\frac{\d^i g}{\d x^i}\,(u(t\,,0))\right|^k\right)<\infty, \quad
	\text{for all $t>0$, $i\in \{0,1,2\}$ and $k\in \mathbb{Z}$.} \label{E:cond_g4}
\end{subnumcases}
It is clear that conditions \eqref{E:cond_g1} -- \eqref{E:cond_g3} are satisfied by examples in \eqref{E:Example}.
As for condition \eqref{E:cond_g4}, thanks to the  comparison  principle established by Chen and Huang \cite{CH19}  and nonnegative moments proved in Chen and Huang \cite[Theorem 1.8]{CH19Density}
(see also Theorem 5.1 of Conus et al \cite{CJK12} for the case when $d=1$ and $f=\delta_0$),
by assuming both condition $\sigma(0)=0$ and the reinforced Dalang's condition \eqref{E:Dalang2},
all examples in \eqref{E:Example} satisfy condition \eqref{E:cond_g4} (even for all $i\ge 0$).
One can see that conditions regarding locally Lipschitz $g$ in Theorems \ref{T:TV}, \ref{T:FCLT}, and \ref{T:PAM} are all special cases of \eqref{E:cond_g}.

Now we are ready to state our results.
The following is the first result of this paper, and  shows that   the weak convergence
theorem of \cite[Theorem 1.1]{CKNP2} holds in total variation when the limit variance is not zero.
This is a strict improvement because manifestly $\lim_{N\to\infty}
N^{d/2} \mathcal{S}_{N,t} (g)=0$ in $L^2(\Omega)$ when the limiting variance is zero.

\begin{theorem} \label{T:TV}
	Let $u$ denote the solution to \eqref{E:SHE} with $u(0)\equiv1$ and $\sigma(1)\ne 0$
	and suppose that \eqref{E:f_finite} holds.
	Suppose that one of the three conditions holds:

	\noindent
	(i)  $g\in \lip$,

	\noindent
	(ii) $g\in C^1(\R)$  and  for some $k>4$
	\begin{align} \label{E:MmCond_TV}
	    \E\left( \left| g^{(i)} (u(s\,,0))\right|^k\right)<\infty, \qquad
	    \text{for all $s>0$, $i\in \{0,1\}$}.
	\end{align}

	\noindent (iii) $\sigma(0)=0$,  $f$ satisfies the reinforced Dalang's condition \eqref{E:Dalang2},
	$g\in C^1(0,\infty)$ and   \eqref{E:MmCond_TV} holds.

	Then, properties \eqref{E:Var_Lim} and \eqref{E:Var_finite} are satisfied and, for any $t>0$, provided that $\mathbf{B}_t(g)>0$, we have that
	\begin{align} \label{E:TV}
	    \lim _{N \rightarrow \infty} d_{\rm TV}
	    \left( N^{d/2} \mathcal{S}_{N,t} (g)~,~\sqrt{  \mathbf{B}_t(g) }\,\mathrm{N}(0,1) \right)=0.
	\end{align}
	Moreover,  in Case (iii), we have the  following sufficient condition for $\mathbf{B}_t(g)\in (0,\infty)$:
	\begin{align} \label{E:gMonotone}
	    \text{$\sigma(x)\ne 0$ for all $x>0$ and $g'$ is either strictly positive or
	    strictly negative for all $x>0$,}
	\end{align}
 \end{theorem}

It is clear that the parabolic Anderson model ($d=1, \sigma(u)=u, f=\delta_0$) and the function $g(u)=\log u$ satisfy the conditions (iii) and \eqref{E:gMonotone} of Theorem \ref{T:TV}; see \cite[Theorem 5.1]{CJK12} for the moment condition    \eqref{E:MmCond_TV}. Therefore, Theorem \ref{T:TV} implies the following central limit theorem for the
Hopf-Cole solution to the KPZ equation.
\begin{corollary}\label{KPZCLT}
	Assume $d=1$, $f=\delta_0$ and $\sigma(u)=u$ for all  $u\in \R$.
	Then, for all $t>0$,
	\begin{align} \label{CLTKPZ}
		\frac{1}{\sqrt N} \int_0^N \left\{ \log u(t\,, x) -\E[\log u(t\,, 0)] \right\} \d x
		\to{\rm N} (0\,,\sigma_t^2)
		\quad\text{in distribution as $N\to\infty$,}
	\end{align}
	where $\sigma_t^2:= \int_{-\infty}^\infty
	\Cov[\log u(t\,, x)\,, \log u(t\,,0)]\,\d x \in (0\,, \infty)$.
\end{corollary}

Our next result provides a functional version in time of the weak convergence proved in \cite[Theorem 1.1]{CKNP2}.

\begin{theorem} \label{T:FCLT}
	Let $u$ denote the solution to \eqref{E:SHE} with $u(0)\equiv1$ and $\sigma(1)\ne 0$
	and suppose that $f$ satisfies both condition \eqref{E:f_finite} and the reinforced Dalang's condition \eqref{E:Dalang2} for some $\alpha\in (0,1]$.
	Then we have the following two cases: \\
	(i) For any   $g\in  C^1(\R)$ such that $g'$ is H\"older continuous of order $\delta \in (0,1]$, it holds that
	\begin{align} \label{E:FCLT}
		\left\{ N^{d/2} \mathcal{S}_{N,t} (g) \right\}_{t\in [0,T]}
		\xrightarrow{C[0,T]}\left\{ \mathcal{G}_t\right\} _{t\in [0,T]}
		\quad\text{as $N\to\infty$, for any $T>0$},
	\end{align}
	where $\{\mathcal{G}_t\}_{t\ge0}$ is a centered Gaussian process with covariance
	$\E [ \mathcal{G}_{t_1}  \mathcal{G}_{t_2} ] = \mathbf{B}_{t_1,t_2}(g)$
	$[$see \eqref{E:Var_Lim}$]$.\\
	(ii) If $\sigma(0)=0$, then \eqref{E:FCLT} holds for any $g\in C^2 (0\,,\infty)$ such that $g''$ is monotone over $(0\,,\infty)$ and
	\begin{align} \label{E:MmCond_FCLT}
		\E\left(\left|\frac{\d^i g}{\d x^i}\,(u(t\,,0))\right|^k\right)  <
		\infty\quad\text{for all $t\ge0$, $k\ge 2$, and $i\in \{0,1,2\}$}.
	\end{align}
 \end{theorem}

In the special case that $g(v)=v$ for all $v\in\R$,  we can use  Malliavin-Stein's method to show that
the total variation distance between $\mathcal{S}_{N,t}(g)$
and a suitable normal distribution is $O(N^{-d/2})$. The next theorem contains the exact form
of this statement, and generalizes the recent work of Huang et al \cite{HuangNualartViitasaari2018}.

\begin{theorem}\label{T:TV_Bnd}
	If $g(v)=v$ for all $v\in\R$, then
	for two real numbers $\lambda,L>0$ --- depending only on $(f\,,\sigma)$ --- it holds that
	\begin{equation}\label{E:TV_Bnd}
		d_{\rm TV} \left(   \frac{\mathcal{S}_{N,t}(g)}{\sqrt{\Var(\mathcal{S}_{N,t}(g))}}
		\,,  {\rm N}(0\,,1) \right) \le  \frac{L\e^{\lambda t}}{N^{d/2}\mathbf{B}_{N,t}(g)},
	\end{equation}
	uniformly for all $t>0$ such that
	$\mathbf{B}_t(g)>0$
	and all $N>0$ large enough to ensure that $\mathbf{B}_{N,t}(g)>0$.
\end{theorem}

We are able to study the parabolic Anderson model as well.
That is the when $\sigma(z)=z$ for all $z\in\R$.
Though we hasten to add that the proof of the following
is different from that of Theorem \ref{T:TV_Bnd}.

At this point, we are not able to analyze the case where both $\sigma$ and $g$ are
fairly general [nice] functions.
\begin{theorem}\label{T:PAM}
    Let $u$ denote the solution to \eqref{E:SHE} with $\sigma(z)=z$ and $u(0)\equiv1$
    and assume that $f$ satisfies both \eqref{E:f_finite} and the reinforced Dalang's condition \eqref{E:Dalang2} for some $\alpha\in (0,1]$.
    Suppose that  $g\in C^2(0\,,\infty)$ and for some $k>4$,
    \begin{align} \label{E:MmCond_PAM}
	\E\left(\left|\frac{\d^i g}{\d x^i}\,(u(t\,,0))\right|^k\right)  <
	\infty\quad\text{for all $t\ge0$ and $i\in \{0,1,2\}$}.
    \end{align}
    Then for two real numbers $\lambda,L>0$ ---
    depending only on $(f\,,\sigma)$ --- it holds that
    \begin{equation}\label{E:PAM_Bnd}
	d_{\rm TV} \left(   \frac{\mathcal{S}_{N,t}(g)}{\sqrt{\Var(\mathcal{S}_{N,t}(g))}}
	\,,  {\rm N}(0\,,1) \right) \le  \frac{L\: \Theta_t\:  \e^{\lambda t}}{N^{d/2}\mathbf{B}_{N,t}(g)},
    \end{equation}
    uniformly for all $t>0$ such that $\mathbf{B}_{N}(g)>0$ and all $N>0$ large enough to ensure that  $\mathbf{B}_{N,t}(g)>0$, where $\Theta_t:= \Norm{g'(u(t\,,0))}_k
    \max(\Norm{g'(u(t\,,0))}_k,\Norm{g''(u(t\,,0))}_{k})$.
\end{theorem}

In Theorem \ref{T:PAM},  the reinforced Dalang's condition \eqref{E:Dalang2} is used to ensure that  almost surely $u(t,x)>0$ for all  $t\ge 0$ and $x\in \R^d$. We do not need this condition  if we assume  $g\in C^2(\R)$ instead of  $g\in C^2(0\,,\infty)$.

\medskip
 Note that both \eqref{E:TV_Bnd} and \eqref{E:PAM_Bnd} are trivial when $\mathbf{B}_{N,t}(g)=0$.
\bigskip

Among other things, Theorems \ref{T:TV_Bnd} and/or \ref{T:PAM}
and Eq.'s \eqref{E:Var_N} and \eqref{E:Var_Lim} together imply the
following extension of \eqref{E:HNV18} from the case that
$d=1$ and $f=\delta_0$ to the  more general settings of Theorems \ref{T:TV_Bnd}
and/or \ref{T:PAM}: If either $g(v)=v$ for all $v\in\R^d$ or
$\sigma(z)=z$ for all $z\in\R$, and if $\mathbf{B}_t(g)>0$ then
\[
	d_{\rm TV} \left(   \frac{\mathcal{S}_{N,t}(g)}{\sqrt{\Var(\mathcal{S}_{N,t}(g))}}
	\,,  {\rm N}(0\,,1) \right) = O(N^{-d/2})\qquad\text{as $N\to\infty$}.
\]
Moreover, because Theorem \ref{T:TV_Bnd} and \ref{T:PAM} involve
quantitative bounds on the total variation distance, we can also sometimes use
them to prove that $\mathcal{S}_{N,t}(g)$ is asymptotically normal as $N,t\to\infty$
simultaneously. The following provides the requisite technical result that allows for
this sort of undertaking. Before we state the result,
let us note that the parabolic Anderson model [that is, $\sigma(z)=z$
for all $z\in\R$] satisfies all three conditions of the following.

\begin{proposition}\label{P:B_LB}
	Suppose $g(v)=v$ for all $v\in\R$,
	and let $\{t_N\}_{N>0}$ be a net of strictly positive numbers such that
	$t_N=o(N^2)$ as $N\to\infty$. Recall
	condition \eqref{E:sigma(1)} and suppose additionally that either
	$Q=\sigma$ or $Q=-\sigma$
	satisfy any one of the following conditions:
	\begin{compactenum}
		\item $Q(x)\ge0$ for all $x>0$;
		\item There exists $c>0$ such that $Q(w)\ge c$ for all $w>0$;
		\item $Q(0)=0$ and there exists $c>0$ such that
			$Q(w)\ge cw$ for all $w>0$.
	\end{compactenum}
	Then,  $\liminf_{N\to\infty}t_N^{-1}\mathbf{B}_{N,t_N}(g)>0$
	under conditions 1 and 2, and $\liminf_{N\to\infty}\mathbf{B}_{N,t_N}(g)>0$
	under condition 3.
\end{proposition}

Indeed, Theorems \ref{T:TV_Bnd} and \ref{T:PAM}, together with
Proposition \ref{P:B_LB},  yield the following result immediately, with no
need for additional justification:

\begin{corollary}\label{C:T_CLT}
	Suppose that $g(v)=v$ for all $v\in\R$, $\sigma$ satisfies one of
	the three conditions of Proposition \ref{P:B_LB},
	and $\{t_N\}_{N>0}$ is a net of strictly positive numbers. Then,
	\[
		\lim_{N\to\infty}
		d_{\rm TV}\left( \frac{\mathcal{S}_{N,t_N}(g)}{\sqrt{\Var(\mathcal{S}_{N,t_N}(g))}}
		~,~ {\rm N}(0\,,1)\right)
		=0\quad\text{provided that}\quad t_N =o(\log N)\quad\text{as $N\to\infty$}.
	\]
\end{corollary}
Such time-dependent CLTs have been anticipated by Deuschel \cite{Deuschel}.

In light of our earlier work \cite{CKNP2}, the proof of Proposition \ref{P:B_LB} is
short enough, and simple enough, that it can be included here.

\begin{proof}[Proof of Proposition \ref{P:B_LB}]
	According to the proof of Proposition 5.3 of Chen et al \cite{CKNP2},
	there exists a real number $C>0$ such that,
	under either condition 1 or condition 2,
	\[
		\Cov[u(t\,,x)\,,u(t\,,y)]\ge C\int_0^t \left(\bm{p}_{2s}*f\right)(x-y)\,\d s
		\qquad\text{for all $t>0$ and $x,y\in\R^d$}.
	\]
	Since $g(v)=v$ for all $v\in\R$, it follows that
	\begin{align*}
		\mathbf{B}_{N,t}(g) &\ge\frac{C}{N^d}\int_{[0,N]^d}\d x\int_{[0,N]^d}\d y\
			\int_0^t\d s \left( \bm{p}_{2s}*f\right)(x-y)\\
		&= \frac{C}{N^d}\int_{[0,N]^d}\d x\int_{x-[0,N]^d}\d z\
			\int_0^t\d s \left( \bm{p}_{2s}*f\right)(z)\\
			& \ge \frac{C}{N^d}\int_{[N/4, 3N/4]^d}\d x\int_{[-N/4,N/4]^d}\d z\
			\int_0^t\d s \left( \bm{p}_{2s}*f\right)(z) \\
			&=\frac{C}{2^d}\int_{[-N/4,N/4]^d}\d z\int_0^t\d s \left( \bm{p}_{2s}*f\right)(z)\\
		&= \frac{C}{2^d}tf(\R^d)
		 - \frac{C}{2^d}\int_{\R^d\setminus[-N/4, N/4]^d}\d z\int_0^t\d s \left( \bm{p}_{2s}*f\right)(z).
	\end{align*}
	If $z\not\in[-N/4\,,N/4]^d$ and $y\in[-N/8\,,N/8]^d$, then
	$z-y\not\in[-N/8\,,N/8]^d$, and hence
	\begin{align*}
		&\int_{\R^d\setminus[-N/4,N/4]^d}\d z\int_0^t\d s \left( \bm{p}_{2s}*f\right)(z)\\
			&\quad \le f([-N/8\,,N/8]^d)\int_{\R^d\setminus[-N/8,N/8]^d}\d x\int_0^t\d s \
			\bm{p}_{2s}(x) + tf(\R^d\setminus[-N/8,N/8]^d)\\
		&\quad\le tf(\R^d)\int_{\R^d
			\setminus[-N/(8\sqrt{t})],N/(8\sqrt{t})]^d}
			\bm{p}_2(x)\d x+ tf(\R^d\setminus[-N/8,N/8]^d).
	\end{align*}
	Replace $t$ by $t_N$ and use the fact that $t_N=o(N^2)$ as $N\to\infty$
	in order to complete the proof of the proposition
	under conditions 1 and/or 2.

	Next suppose condition 3 of the proposition holds. In this case,
	the proof of Proposition 5.3 of Chen et al \cite{CKNP2} yields strictly positive numbers
	$\delta$ and $R$ such that
	\[
		\Cov[u(t\,,x)\,,u(t\,,y)]\ge
		\frac{[\sigma(1)]^2}{2}\int_0^\delta\d s\int_{[-R,R]^d} f(\d w)\
		\bm{p}_{2(t-s)}(x-y+w)
		\qquad\text{for all $t>0$ and  $x,y\in\R^d$}.
	\]
	Therefore,
	\begin{align*}
		\mathbf{B}_{N,t}(g) &\ge\frac{[\sigma(1)]^2}{2N^d}
			\int_0^\delta\d s\int_{[-R,R]^d} f(\d w)\int_{[0,N]^d}\d x\int_{[0,N]^d}\d y\
			\bm{p}_{2(t-s)}(x-y+w)\\
		&=\frac{[\sigma(1)]^2}{2N^d}
			\int_{t-\delta}^t \d s\int_{[-R,R]^d} f(\d w)\int_{[0,N]^d}\d x \int_{x-[0, N]^d}\d z\
			\bm{p}_{2s}(z+w)\\
		&\ge\frac{[\sigma(1)]^2}{2^{d+1}}
			\int_{t-\delta}^t \d s\int_{[-R,R]^d} f(\d w)\int_{[-N/4,N/4]^d}\d z\
			\bm{p}_{2s}(z+w)\\
		&=\frac{[\sigma(1)]^2\delta}{2^{d+1}}f\left([-R\,,R]^d\right)
			-\frac{[\sigma(1)]^2}{2^{d+1}}
			\int_{t-\delta}^t \d s\int_{[-R,R]^d} f(\d w)\int_{\R^d\setminus
			[-N/4,N/4]^d}\d z\
			\bm{p}_{2s}(z+w).
	\end{align*}
	The preceding triple integral can be bounded from above as follows:
	\begin{align*}
		\int_{t-\delta}^t \d s\int_{[-R,R]^d} f(\d w)\int_{\R^d\setminus
			[-N/4,N/4]^d}\d z\
			\bm{p}_{2s}(z+w)
			&\le f(\R^d)\int_{t-\delta}^t \d s\int_{\R^d\setminus
			[R-N/4,-R+N/4]^d}\d x\
			\bm{p}_{2s}(x)\\
		&\le f(\R^d)\delta\int_{\R^d\setminus
			[(R-N/4)/\sqrt t,(N/4-R)/\sqrt t]^d}\bm{p}_2(x)\,\d x.
	\end{align*}
	Replace $t$ by $t_N$ and observe that the integral tends to zero as $N\to\infty$
	because $t_N=o(N^2)$ as $N\to\infty$. This
	completes the proof because $f([-R\,,R]^d)>0$;
	see Chen et al \cite[Proposition 5.3]{CKNP2}.
\end{proof}

%
%
%

\section{Preliminaries}

In this section we collect some basic facts about the Malliavin-Stein
method. We also derive  estimates on the $L^k(\Omega)$-norm of the Malliavin
derivatives of the solution $u=\{u(t\,,x)\}_{t\ge0,\,x\in\R^d}$ to the stochastic PDE  \eqref{E:SHE}.
These $L^k(\Omega)$ estimates will play a pivotal role in the remainder of the paper.
We also  provide sufficient conditions, based on our earlier work \cite{CKNP2},
for the nondegeneracy of the limiting variance of $N^{d/2}\mathcal{S}_{N,t}(g)$.

\subsection{Malliavin's calculus and Stein's method}

Let $\mathcal{H}_0$ denote the Hilbert space that is obtained by completing
$\mathscr{S}(\R^d)$ under the pre-Hilbertian norm that is defined by scalar product
$\langle \phi_1\,, \phi_2\rangle_{\mathcal{H}_0} = \langle \phi_1\,,  \phi_2 *f \rangle_{L^2(\R^d)}$,
and let $\mathcal{H}= L^2(\R_+;\mathcal{H}_0)$.
The Gaussian random field $\{ \eta(h)\}_{\phi \in \mathcal{H}}$, formed by the Wiener integrals
\begin{equation}\label{E:iso}
	\eta(h) := \int_{\R_+\times \R^d}  h(s\,,x) \,\eta(\d s\,\d x)\qquad[h\in\mathcal{H}],
\end{equation}
is called the {\it isonormal Gaussian process} on the Hilbert space $\mathcal{H}$.
Thus, we can develop the Malliavin calculus in this framework; see, for instance,  Nualart \cite{Nualart}.

We denote by $D$ the Malliavin derivative operator and by $\delta$ the corresponding divergence
operator that is defined by the following adjoint relation:
\begin{equation}\label{D:delta}
	\E\left(\langle DF\,, v \rangle_{\mathcal{H}}\right) = \E[F \delta(v)],
\end{equation}
valid for every random variable $F$ in the Gaussian Sobolev space $\mathbb{D}^{1,2}$
and every $v$ in the domain in $L^2(\Omega)$ of  $\delta$.
An important property of the divergence operator is that
\[
	\delta(F) = \int_{\R_+\times\R^d} F(s\,,x)\,\eta(\d s\,\d x)
	\qquad[\text{the Walsh integral}],
\]
when $F$ is a  predictable and square-integrable random field.

We make extensive use of the following
form of the Clark-Ocone formula  (see Chen et al \cite[Proposition 6.3]{CKNP}):
\begin{equation} \label{E:Clark-Ocone}
	F= \E F  + \int_{\R_+\times\R^d}
	\E\left(D_{s,z} F \mid \mathcal{F}_s\right) \eta(\d s \, \d z)
	\qquad\text{a.s.\ for every $F\in \mathbb{D}^{1,2}$}.
\end{equation}
 Among other things, the Clark--Ocone formula
readily yields the Poincar\'e inequality,
\begin{equation} \label{E:Poincare}
	\Var(F) \le \E( \|DF\|_{\HH}^2)
	\qquad\text{for all $F\in\mathbb{D}^{1,2}.$}
\end{equation}

Recall that the total variation distance between two  probability measures
$\mu$ and $\nu$ on $\R$ is defined by
\[
	d_{\rm TV}  (\mu ,\nu)= \sup_{B\in \mathcal{B}(\R)} | \mu(B)- \nu(B)|,
\]
where  $\mathcal{B}(\R)$ denotes the family of all Borel subsets on $\R$. As is customary,
we  let $d_{\rm TV}(F\,,G)$ denote the total variation distance between the laws of $F$ and $G$
whenever $F$ and $G$ are random variables. And $d_{\rm TV}(F\,,{\rm N}(a\,,b))$ is
written interchangeably for $d_{\rm TV}(F\,,G)$ where $G$ has a normal distribution
with mean $a$ and variance $b$.

The combination of Stein's method for normal approximations with Malliavin calculus leads to the following
bound on the total variation distance. See Nualart and Nualart \cite[Theorem 8.2.1]{NN} for details and proof.

\begin{proposition}\label{P:TV}
	Suppose that $F\in \mathbb{D}^{1,2}$ satisfies $F=\delta(v)$ for some  $v$ in the domain in
	$L^2(\Omega)$  of the divergence operator $\delta$, and suppose that $\tau^2 := \E(F^2) >0$. Then
	\begin{equation} \label{E:SM1}
		d_{\rm TV} (F\,,  {\rm N}(0,\tau^2)) \le \frac 2 {\tau^2}
		\E \left( | \tau^2 - \langle DF\,, v \rangle_{\mathcal{H}} | \right).
	\end{equation}
\end{proposition}

Thanks to the duality relationship \eqref{D:delta}, the pair $(F\,,v)$ of Proposition \ref{P:TV}
satisfies
\begin{equation} \label{E:SM2}
	\E(\langle DF\,, v \rangle_{\mathcal{H}}) = \tau^2 = \Var(F).
\end{equation}
Therefore, Proposition \ref{P:TV} is bounding the total variation distance between $F$
and ${\rm N}(0\,,\Var(F))$ by the centralized $L^1$-norm of $\<DF\,,v\>_{\mathcal{H}}$.
Note also that, as a consequence of \eqref{E:SM1}, \eqref{E:SM2},
and Jensen's inequality,
\begin{equation}\label{E:SM3}
	d_{\rm TV}(F\,, {\rm N}(0\,,1))\le 2\sqrt{\Var\left( \< DF\,,v\>_{\mathcal{H}}\right)}
	\quad\text{when $\Var(F)=1$}.
\end{equation}

\subsection{Malliavin derivatives of the solution}

In this part, we will establish the strict positivity of the Malliavin derivatives to the stochastic heat equation in Theorem \ref{T:Positive_D} and
derive its moment estimates in Lemma \ref{L:Du} and Proposition \ref{P:secondD}.
We first establish the positivity of the Malliavin derivative and it will be a key ingredient
in the proof of the positivity of the limiting variance in our CLT.

\begin{theorem} \label{T:Positive_D}
	Let $u$ denote the solution to \eqref{E:SHE} subject to $u(0)\equiv1$.
	Suppose that $\sigma(0)=0$ and $\sigma(x)\ne 0$ for all $x>0$.
	Suppose that  $f$ satisfies the reinforced Dalang's condition \eqref{E:Dalang2} for some $\alpha\in(0\,,1]$.
	Then, there is a version of the Malliavin derivative of $u(t\,,x)$ that satisfies
	\begin{align} \label{E:DuPos}
		D_{s,z}u(t\,,x) >0 \quad \text{for all $t>s$ and $x\in\R^d$\quad a.s.}
	\end{align}
	for all $s>0$ and $z\in\R^d$.
\end{theorem}

\begin{proof}
	Choose and fix an  arbitrary $s\ge 0$ and $z\in\R^d$.
	Proposition 3.2 of Chen and Huang \cite{CH19Density} implies that,
	for all $(s\,,z)$  fixed,
	the random field
	$\{D_{s,z}u(t\,,x)\}_{(t,x)\in (s,\infty)\times\R^d}$ satisfies the following stochastic integral equation:
	\begin{align*}
		D_{s,z}u(t\,,x) = \bm{p}_{t-s}(x-z) \sigma(u(s\,,z)) +
		\int_{\left(s,t\right)\times\R^d} \bm{p}_{t-r}(x-y)
		\sigma'(u(r\,,y)) D_{s,z} u(r\,,y)\,\eta(\d r\,\d y).
	\end{align*}
	Because $u(s\,,z)>0$, by the assumptions on $\sigma$, we see that
	$C_{s,z}:=\sigma(u(s\,,z))$ belongs to $\mathcal{F}_s$ and is strictly positive a.s.
	For $(s\,,z)$ fixed, the field $(t\,,x)\mapsto V_{s,z}(t\,,x):=  C_{s,z}^{-1}\:   D_{s,z}u(t\,,x)$ satisfies
	\begin{align*}
		V_{s,z}(t\,,x) =  p_{t-s}(x-z)  +
		 \int_{(s,t)\times\R^d}
		\bm{p}_{t-r}(x-y) \sigma'(u(r\,,y)) V_{s,z} (r\,,y)\,\eta(\d r\,\d y),
	\end{align*}
	 where $(t,x) \in (s, \infty) \times \R^d$.
	Note that the above integral equation is
	nothing but the mild solution to the following variant of the stochastic heat equation:
	\begin{align*}
	\begin{cases}
		\displaystyle \frac{\partial }{\partial t} V_{s,z}(t\,,x) =
			\tfrac{1}{2}\Delta V_{s,z}(t\,,x) + H_{s,z}(t\,,x\,,\omega) V_{s,z}(t\,,x) \eta(t\,,x)
			&[t>s,\: x\in\R^d],\\
		\displaystyle V_{s,z}(s\,,x) = \delta_0 ,
	\end{cases}
	\end{align*}
	where $H_{s,z}(t\,,x\,,\omega):= \sigma'(u(t\,,x) (\omega))$.
	Notice that $|H_{s,z}(t\,,x\,,\omega)|\le \lip(\sigma)$ uniformly in $(s,z,t,x,\omega)$; see Remark \ref{R:Lip}.
	This is exactly the same setup as Theorem 1.8 of Chen and Huang \cite{CH19Density} or Theorem 1.5 of
	Chen et al \cite{CHN} for the case when $d=1$ and $f=\delta_0$ (see also Theorem 1.5 of Chen and Huang \cite{CH19}).
	Therefore, one can apply these references to conclude that
	\begin{align*}
	    \P\left\{ V_{s,z}(t\,,x)>0,\: t>s,x\in\R^d \bigg| C_{s,z}>0\right\} =1,
	\end{align*}
	or in other words,
	\begin{equation*}
		\P\left\{ D_{s,z} u(t,x)>0,t>s,x\in\R^d\right\}
		= \P\left\{ V_{s,z}(t\,,x)>0,t>s,x\in\R^d \bigg| C_{s,z}>0\right\}
			\P\left\{ C_{s,z}>0\right\} =1,
	\end{equation*}
	for all $s\ge 0$ and $,z\in\R^d$,
	which is nothing but \eqref{E:DuPos}. This completes the proof of the theorem.
\end{proof}
\bigskip

In order to derive the moment estimates of the Malliavin derivatives of the solution to \eqref{E:SHE}, let us introduce some notation.
Denote
\[
	\Upsilon(\lambda) := \frac{2}{(2\pi)^d}\int_{\R^d} \frac{\hat{f}(\d z)}{2\lambda+\|z\|^2}
	\qquad\text{for  all $\lambda>0$}.
\]
Dalang's condition (\ref{E:Dalang}) ensures that $\Upsilon(\lambda)<\infty$
for all $\lambda>0$. In that case, $\Upsilon$ decreases strictly as $\lambda$ increases.
Therefore, $\Upsilon$ has an inverse which we denote by
\[
	\Lambda= \Upsilon^{-1}.
\]
The following variant of Chen et al \cite[Theorem 6.4]{CKNP}
shows a way in which $\Lambda$ can
be used to control the size of the moments of the Malliavin derivative of the solution.

\begin{lemma}\label{L:Du}
	For all real numbers $\varepsilon\in(0\,,1)$,
	$t>0$, and $k\ge2$,
	and for every $x\in\R^d$,
	\begin{equation}  \label{E:D<p2}
		\left\| D_{s,z}u(t\,,x) \right\|_k \le
		C_{t,k,\varepsilon,\sigma}
		\bm{p}_{t-s}(x-z),
	\end{equation}
	valid for a.e.\ $(s\,,z)\in(0\,,t)\times\R^d$, where
	\begin{equation} \label{E:ec3}
		C_{t,k,\varepsilon,\sigma}:=\frac{8\left( |\sigma(0)|\vee\lip(\sigma)\right)
		\e^{2t\Lambda(\mathsf{a}(\varepsilon)/k)}}    {\varepsilon^{3/2}}
		\quad\text{and}\quad
		\mathsf{a}(\varepsilon)
		:= \frac{(1-\varepsilon)^2}{2^{(d+6)/2}[|\sigma(0)|\vee\lip(\sigma)]^2},
	\end{equation}
	where $1\div 0:=\infty$.
\end{lemma}

\begin{proof}
	Choose and fix some $T>0$. We proved in Chen et al \cite[Theorem 6.4]{CKNP}
	(see also \cite[Lemma 4.2]{CKNP2}) that,
	for each $t\in(0\,,T]$ and $x\in\R^d$, the random variable
	$u(t\,,x)$ is in the Gaussian Sobolev space $\mathbb{D}^{1,k}$  (see Nualart
	\cite[Section 1.5]{Nualart}) for every
	$k\ge2$, and that
	\[
		\left\| D_{s,z}u(t\,,x) \right\|_k \le
		C_{T,k,\varepsilon,\sigma}
		\bm{p}_{t-s}(x-z),
	\]
	for a.e.\ $s\in(0\,,t)$ and $x,z\in\R^d$. Set $t=T$ and relabel
	$[T\leftrightarrow t]$ to derive the result.
\end{proof}

Lemma \ref{L:Du} has nontrivial content if and only if $|\sigma(0)|\vee\lip(\sigma)>0$.
This is precisely when $\sigma\not\equiv0$, which is equivalent to the statement that
the stochastic PDE \eqref{E:SHE} is not identically the same as the nonrandom heat equation
$\partial_t u = \frac12\Delta u$.

In the particular case of the parabolic Anderson model, that is,
when $\sigma(z)=z$ for all $z\in\R$, we are also able to estimate the moments of second-order
Malliavin derivatives.  Notice that when $\sigma(z)=z$
for all $z\in\R$, the constant defined in (\ref{E:ec3}) has the form
\begin{equation} \label{E:ec5}
	C_{t,k,\varepsilon} := 8\varepsilon^{-3/2}
	\exp\left\{ 2t\Lambda\left( \frac{(1-\varepsilon)^2 }{2^{(d+6)/2}\,k}\right)\right\}.
\end{equation}
Later on, for the proof of Theorem \ref{T:PAM},
we will need the following moment bound of second-order derivatives.

\begin{proposition} \label{P:secondD}
	Suppose that $\sigma(z)=z$ for all $z\in\R$.
	Then, for all real numbers $\varepsilon\in(0\,,1)$,
	$t>0$, and $k\ge2$,
	and for every $x\in\R^d$,
	\[
		\left\| D_{r,z}D_{s,y}  u(t\,,x) \right\|_k \le
		C^*_{t,k,\varepsilon}  \left[ \bm{p}_{t-s} (x-y)  \bm{p}_{s-r} (y-z)
		\mathbf{1}_{(r,t)}(s)+  \bm{p}_{t-r} (x-z)
		\bm{p}_{r-s} (z-y) \mathbf{1}_{(0,r)}(s) \right],
	\]
	valid for a.e.\ $(s\,,y),(r\,,z) \in(0\,,t)\times\R^d$,
	where
	\begin{align}\label{E:Cstar}
		C^*_{t,k,\varepsilon}=16 \varepsilon^{-2}
		\exp\left\{ 3t\Lambda\left(
		\frac{(1-\varepsilon)^2 }{2^{(d+6)/2}\,k}\right)\right\}.
	\end{align}
\end{proposition}

\begin{proof}
	The proof is fairly involved, though it mostly
	follows arguments that have been used earlier
	in simpler forms. Specifically, see
	the proof of Theorem 6.4 in \cite{CKNP};
	see also Lemma 2.2 of \cite{CH19}).

	Consider the Picard iterations defined by $u_0(t\,,x):=1$ and
	\[
		u_{n+1}(t\,,x):= 1+ \int_{(0,t)\times\R^d} \bm{p}_{t-\tau}(x-y)
		u_n(\tau\,,\xi)\, \eta ( \d \tau\, \d \xi)
	\]
	for all $n\in\mathbb{Z}_+$,  $t\ge 0$, and $x\in \R^d$.
	The elementary properties of the Malliavin derivative and
	induction together show that $u_n(t\,,x)\in \cap_{k\ge2}\mathbb{D}^{2,k}$
	for every $(t\,,x)\in \mathbb{R}_+ \times \mathbb{R}^d$ and $n\in\mathbb{Z}_+$.
	Moreover,
	\[
		D_{s,y} u_{n+1}(t\,,x)= \bm{p}_{t-s} (x-y) u_n(s\,,y) +
		\int_{(s,t)\times\R^d} \bm{p}_{t-\tau} (x-\xi)
		D_{s,y} u_n(\tau\,,\xi )\,\eta(\d \tau \,\d \xi),
	\]
	for almost every $(s\,,y) \in (0\,,t)\times  \R^d$. Since $D_{s,y}u_0(t\,,x)=0$ a.s.\
	for every $s,t\ge0$ and $x,y\in\R^d$,
	we can, and in fact will,
	change the value of  $D_{s,y}u_{n+1}(t\,,x)$ on a Lebesgue-null set of
	values of $(s\,,y)$ and \emph{define} $D_{s,y} u_{n+1}(t\,,x)$ for every $(s\,,y)\in(0\,,t)\times\R^d$.

	A second round of differentiation yields the following, after one more round of modifications
	on null sets:
	\begin{align*}
		D_{r,z} D_{s,y} u_{n+1}(t\,,x) & =
			\bm{p}_{t-s} (x-y)  D_{r,z} u_n(s,y) +  \bm{p}_{t-r} (x-z)   D_{s,y} u_n(r\,,z)   \\
		& \hskip1.8in + \int_{(s\vee r,t)\times\R^d} \bm{p}_{t-\tau} (x-\xi)
			D_{r,z} D_{s,y} u_n(\tau\,,\xi )\,\eta(\d \tau \,\d \xi),
	\end{align*}
	for  every $(s\,,y), (r\,,z) \in (0\,,t)\times  \R^d$.

	According to the proof of Theorem 6.4 in Chen et al \cite{CKNP}, and thanks to the same argument
	that was used in Lemma \ref{E:D<p2},
	the estimate  (\ref{E:D<p2}) holds also --- with the same constant $C_{t,k,\varepsilon,\sigma}$ ---
	when we replace $D_{s,z}u(t\,,x)$ by $D_{s,z}u_{n+1}(t\,,x)$.
	Indeed, Chen et al ({\it ibid.})  first bound the moments
	of $D_{s,z}u_n(t\,,x)$ and then obtain \eqref{E:D<p2} by
	passing to the limit, as $n\to\infty$. In this way, we find that
	\begin{align*}
		\| D_{r,z} D_{s,y} u_{n+1}(t\,,x)\|_k  & \le  C_{t,k,\varepsilon}
		 	\left[ \bm{p}_{t-s} (x-y)  \bm{p}_{s-r} (y-z)  \mathbf{1}_{(r,t)}(s)
			+  \bm{p}_{t-r} (x-z)
			\bm{p}_{r-s} (z-y) \mathbf{1}_{(0,r)}(s) \right]   \\
		& \hskip1.3in + \left\|  \int_{(s\vee r,t)\times\R^d} \bm{p}_{t-\tau} (x-\xi)
			D_{r,z} D_{s,y} u_n(\tau\,,\xi ) \,\eta(\d \tau \,\d \xi) \right\|_k,
	\end{align*}
	for the same constant $ C_{t,k,\varepsilon}$ that was defined earlier in (\ref{E:ec5}).
	We can now apply the Burkholder--Davis--Gundy inequality \cite{BDG},
	using with the Carlen and Kre\'e \cite{CarlenKree1991}
	bounds on the optimal constants, in order to find that
	\begin{align*}
		& \E \left( \left| \int_{(s\vee r,t)\times\R^d} \bm{p}_{t-\tau} (x-\xi)
		 	   D_{r,z} D_{s,y} u_{n+1}(\tau\,,\xi ) \,\eta(\d \tau \,\d \xi)  \right |^k \right)  \\
		& \le \left(2\sqrt{k}\right)^k  \E \left( \left|
			\int_{r\vee s} ^t  \d\tau \int_{\R^d} \d\xi \int_{\R^d} f(\d\xi')\
			\bm{p}_{t-\tau} (x-\xi+\xi') \bm{p}_{t-\tau} (x-\xi)
			\mathcal{T}_n(\tau\,,\xi-\xi')\mathcal{T}_n(\tau\,,\xi')
			\right |^{k/2} \right),
	\end{align*}
	where
	\[
		\mathcal{T}_n(\tau\,,a) := D_{r,z} D_{s,y} u_n(\tau\,,a )
		\qquad\text{for every $\tau>0$ and $a\in\R^d$}.
	\]
	We combine the above with Minkowski's inequality and H\"older's inequality,
	and deduce that
	\begin{align*}
		& \left\| \int_{(s\vee r,t)\times\R^d} \bm{p}_{t-\tau} (x-\xi)
		 	   D_{r,z} D_{s,y} u_{n+1}(\tau\,,\xi ) \,\eta(\d \tau \,\d \xi)  \right\|_k \\
		& \le 2\sqrt{k} \left[
			\int_{r\vee s} ^t  \d \tau \int_{\R^d} \d \xi
			\int_{\R^d} f(\d \xi') \
			\bm{p}_{t-\tau} (x-\xi+\xi')\bm{p}_{t-\tau} (x-\xi)
			\| \mathcal{T}_n(\tau\,,\xi-\xi')  \|_k
			\|\mathcal{T}_n(\tau\,,\xi')  \|_k \right]^{1/2}.
	\end{align*}
	The preceding computations yield the following inequality on the
	$L^k(\Omega)$-norm of the second Malliavin derivative of $u_{n+1}(t\,,x)$:
	\begin{gather*}
		\Norm{ D_{r,z} D_{s,y} u_{n+1}(t\,,x)}_k
		\le  C_{t,k,\varepsilon}
		[\mathcal{P}_{s,r,t}(y\,,z\,;x) \mathbf{1}_{(r,t)}(s)
	  		+\mathcal{P}_{r,s,t}(z\,,y\,;x) \mathbf{1}_{(0,r)}(s) ]\\
		\qquad+ 2\sqrt{k}   \left[
			\int_{r\vee s} ^t  \d \tau \int_{\R^d} \d \xi \int_{\R^d} f(\d \xi') \
			\bm{p}_{t-\tau} (x-\xi+\xi')\bm{p}_{t-\tau} (x-\xi)
			\Norm{ \mathcal{T}_n(\tau\,,\xi -\xi')}_{k}
			\Norm{\mathcal{T}_n(\tau\,,\xi')}_{k} \right]^{1/2},
	\end{gather*}
	where
	\[
		\mathcal{P}_{s,r,t}(y\,,z\,;x)=  \bm{p}_{t-s} (x-y)  \bm{p}_{s-r} (y-z).
	\]

	Set $\gamma :=2^{d/2}16k$ and  $h_0(t) := 1$ for all $t>0$, and define
	\[
		h_n(t):= \int_0^t h_{n-1}(s) \left(\bm{p}_{2(t-s)} *f \right)(0) \,\d s
		\qquad\text{for all $t>0$ and $n\in\mathbb{N}$.}
	\]
	We claim that, for every $n\in\mathbb{Z}_+$,
	\begin{align}  \label{Claim1}
		 &\Norm{ D_{r,z} D_{s,y} u_n(t\,,x)}_k   \\\nonumber
		 &\le \sqrt{2} C_{t,k,\varepsilon}
			 \left[  \mathcal{P}_{s,r, t}(y\,,z\,;x)
			 \sqrt{ \sum_{i=0}^n \gamma^i h_i(t-s)}\, \mathbf{1}_{(r,t)}(s)
			 +  \mathcal{P}_{r,s, t}(z\,,y\,;x)
			 \sqrt{\sum_{i=0}^n \gamma^i h_i(t-r) }\, \mathbf{1}_{(0,r)}(s) \right].
	\end{align}
	We prove  (\ref{Claim1}) using induction on $n$.

	Eq.\ (\ref{Claim1}) holds for $n=0$ since $u_0\equiv1$ whence
	$D_{r,z}D_{s,y} u_0(t\,,x)=0$ a.s. We now suppose that \eqref{Claim1} with $n$
	replaced by any arbitrary $m=0, \dots, n-1$ and seek to verify that
	\eqref{Claim1} holds for $n\ge1$.

	Indeed, by the induction hypothesis,
	\begin{align*}
		& \int_{r\vee s} ^t  \d \tau \int_{\R^d} \d \xi \int_{\R^d}f(\d\xi')\
			\bm{p}_{t-\tau} (x-\xi+\xi')\bm{p}_{t-\tau} (x-\xi)
			\| \mathcal{T}_{n-1}(\tau\,,\xi -\xi')  \|_{k}
			\| \mathcal{T}_{n-1}(\tau\,,\xi )  \|_{k}\\
		&\le  2\mathbf{1}_{(r,t)}(s) C^2_{t,k,\varepsilon}
			\sum_{i=0}^{n-1} \gamma^i \int_s ^t  \d \tau \int_{\R^d} \d \xi \int_{\R^d}f(\d\xi')\
			\bm{p}_{t-\tau} (x-\xi+\xi')\bm{p}_{t-\tau} (x-\xi) \\
		&\hskip3in \times  \mathcal{P}_{s,r,\tau}(y\,,z\,;\xi-\xi')  \mathcal{P}_{s,r,\tau}(y\,,z\,;\xi)h_i(\tau-s)  \\
		&\quad+   2\mathbf{1}_{(0,r)}(s) C^2_{t,k,\varepsilon}
			\sum_{i=0}^{n-1} \gamma^i    \int_r ^t  \d \tau \int_{\R^d} \d \xi
			\int_{\R^d} f(\d\xi')\ \bm{p}_{t-\tau} (x-\xi+\xi')\bm{p}_{t-\tau} (x-\xi)\\
		&\hskip3in \times
			\mathcal{P}_{r,s,\tau}(z\,,y\,;\xi-\xi')  \mathcal{P}_{r,s,\tau}(z\,,y\,;\xi) h_i(\tau-r)\\
		&=:  2\mathbf{1}_{(r,t)}(s)  \Phi(s\,,r\,;y\,,z)+
			 2\mathbf{1}_{(0,r)}(s)  \Phi(r\,,s\,;z\,,y).
  	\end{align*}
	Because
	\[
		\bm{p}_{t-s}(x-y)\bm{p}_s(y-z)= \bm{p}_t(x-z) \bm{p}_{s(t-s)/t}\left(y-z-\frac st(x-z)\right),
	\]
	it follows that
	\begin{align*}
		\Phi(s\,,r\,;y\,,z)&=C^2_{t,k,\varepsilon} \bm{p}^2_{s-r} (y-z)
			\sum_{i=0}^{n-1} \gamma^i \int_s ^t  \d \tau \int_{\R^d} \d \xi
			\int_{\R^d} f(\d \xi')\ \bm{p}_{t-\tau} (x-\xi+\xi')\bm{p}_{t-\tau} (x-\xi) \\
		&\hskip.3in \times    \bm{p}_{\tau-s} (\xi-\xi'-y)   \bm{p}_{\tau-s} (\xi-y) h_i(\tau-s)   \\
		&=C^2_{t,k,\varepsilon} \bm{p}^2_{s-r} (y-z) \bm{p}^2_{t-s}(x-y)\\
		&\hskip.3in \times  \sum_{i=0}^{n-1} \gamma^i \int_s ^t  \d \tau
			\int_{\R^d} \d \xi \int_{\R^d}f(\d\xi')\
			\bm{p}_{ (t-\tau)(\tau-s)/(t-s)} \left(\xi-\xi'-y -\frac {\tau-s}{t-s} (x-y) \right)\\
		&\hskip.3in \times  \bm{p}_{ (t-\tau)(\tau-s)/(t-s)}
 	   \left(\xi-y -\frac {\tau-s}{t-s} (x-y) \right) h_i(\tau-s).
	\end{align*}
	Therefore, the semigroup property of the heat kernel, together with its symmetry, yield the following:
	\[
		\Phi(s\,,r\,;y\,,z)
		= C^2_{t,k,\varepsilon}    \bm{p}^2_{s-r} (y-z) \bm{p}^2_{t-s}(x-y)
		\sum_{i=0}^{n-1} \gamma^i \int_s ^t  \d \tau
		\int_{\R^d}f(\d \xi')\  \bm{p}_{ 2(t-\tau)(\tau-s)/(t-s)} ( \xi' )
		h_i(\tau-s).
	\]
	According to Lemma 6.6 of Chen et al \cite{CKNP},
	\[
		\int_s^t g(\tau-s) p_{2(\tau-s)(t-\tau)/(t-s)}(y)\,\d \tau \le
		2^{(d+2)/2}\int_s^t g(\tau-s) p_{2(t-\tau)}(y)\,\d \tau,
	\]
	for every nondecreasing function $g:(0\,,t-s)\to\R$.
	In accord with Lemma 2.6 of Chen and Kim \cite{CK19},
	every $h_i$ is nondecreasing (this fact follows from induction on $i$).
	Therefore, it follows from the above [with $g\leftrightarrow h_i$]
	and the definition of the $h_i$'s that
	\begin{align*}
		\Phi(s\,,r\,;y\,,z) &\le C^2_{t,k,\varepsilon}
			2^{(d+2)/2}  \bm{p}^2_{s-r} (y-z) \bm{p}^2_{t-s}(x-y)
			\sum_{i=0}^{n-1} \gamma^i  h_{i+1}(t-s)\\
		&=\frac{C^2_{t,k,\varepsilon} }{8k}  \bm{p}^2_{s-r} (y-z) \bm{p}^2_{t-s}(x-y)
			\sum_{i=1}^n \gamma^i  h_i(t-s),
	\end{align*}
	the last line holding since we originally made the special choice,
	$\gamma=2^{d/2}  16 k$. Consequently,
	\begin{align*}
		&\| D_{r,z} D_{s,y} u_n(t\,,x)\|_k\\
		& \le  C_{t,k,\varepsilon}  [\mathcal{P}_{s,r,t}(y\,,z\,;x) \mathbf{1}_{(r,t)}(s)
			+\mathcal{P}_{r,s,t}(z\,,y\,;x) \mathbf{1}_{(0,r)}(s) ]\\
		&\hskip1in+ \sqrt{2}  C_{t,k,\varepsilon}  \Bigg[\mathbf{1}_{(r,t)}(s)
			\bm{p}_{s-r} (y-z) \bm{p}_{t-s}(x-y)
			\left( \sum_{i=1}^n \gamma^{i}  h_{i}(t-s) \right)^{1/2}  \\
		&\hskip2in + \mathbf{1}_{(0,r)}(s) \bm{p}_{r-s} (z-y) \bm{p}_{t-r}(x-y)
			\left( \sum_{i=1}^n \gamma^{i}  h_{i}(t-r) \right)^{1/2} \Bigg],
	\end{align*}
	which proves (\ref{Claim1}).

	We appeal to the nondecreasing property
	of the $h_n$'s once again in order to deduce from \eqref{Claim1} that
	\[
	 	\| D_{r,z} D_{s,y} u_{n+1}(t\,,x)\|_k   \le   \sqrt{2}
		C_{t,k,\varepsilon} \bm{P}(s\,,r\,,t\,;y\,,z\,,x)
		\sqrt{\sum_{i=1}^n \gamma^{i}  h_{i}(t) },
	\]
	where
	\[
		\bm{P}(s\,,r\,,t\,;y\,,z\,,x)  :=\mathbf{1}_{(r,t)}(s)
		\bm{p}_{s-r} (y-z) \bm{p}_{t-s}(x-y)+
		\mathbf{1}_{(0,r)}(s) \bm{p}_{r-s} (z-y) \bm{p}_{t-r}(x-y).
	\]
	 We now let $n\to\infty$ in order to obtain
	 \[
		\| D_{r,z} D_{s,y} u(t\,,x)\|_k
		 \le   \sqrt{2}  C_{t,k,\varepsilon}
		 \bm{P}(s\,,r\,,t\,;y\,,z\,,x)  \sqrt{
		 \sum_{i=1}^\infty \gamma^{i}  h_{i}(t)}.
	\]
	The details are the same as its counterpart in the proof of Theorem 6.4 of \cite{CKNP}.
	Therefore, we skip those details. Because the Plancherel theorem implies
	that $\int_0^\infty\exp(-\lambda t)(\bm{p}_t*f)(0)\,\d t=\Upsilon(\lambda)$,
	Lemma 6.7 of Chen et al \cite{CKNP} implies that
	\[
		\sum_{i=1}^\infty \gamma^{i}  h_{i}(t)  \le
		\frac{\exp(2\lambda t)}{1-\tfrac 12 \gamma \Upsilon(\lambda)},
	\]
	for all large enough $\lambda>0$ that satisfy
	$2>\gamma\Upsilon(\lambda)$.  We now follow the proof of  Lemma 4.2 in \cite{CKNP2} and choose
	the particular value,
	\[
		\lambda= \Lambda \left( \frac {(1-\varepsilon)^2}{k2^{(d+6)/2}} \right).
	\]
	This choice yields
	\[
		\| D_{r,z} D_{s,y} u(t\,,x)\|_k   \le 2  \varepsilon^{-1/2}  C_{t,k,\varepsilon}
		\exp\left\{
		t\Lambda\left( \frac{(1-\varepsilon)^2}{k2^{(d+6)/2}}\right)\right\}
		\bm{P}(s\,,r\,,t\,;y\,,z\,,x),
	\]
	and completes the proof of Proposition \ref{P:secondD}.
	\end{proof}

\subsection{Malliavin derivatives and shifts of the noise}

In the general setup of Malliavin calculus, the Hilbert space $L^2(\Omega)$
includes square-integrable random variables that are functions of the
isonormal process $h\mapsto\eta(h)$; see \eqref{E:iso}. Thus, we identify every
random variable $F\in L^2(\Omega)$ with a function $F(\eta)$ of the underlying noise
$\eta$. Note that this function need not be local. That is, the instance $F(\omega)$
of the random variable $F$ [$\omega\in\Omega$] need not be a function of
the instance $\eta(\omega)$ of the noise. To see how this can happen, choose and
fix some $h\in\mathcal{H}$ and consider the random variable $F_0=\eta(h)$. Because Wiener
integrals cannot be defined pathwise (unless $h$ is sufficiently smooth), this shows that
the random variable $F_0$ is not a local function of $\eta$, though it is of course a function of
$\eta$.

Let us follow the proof of  Lemma 7.1 in \cite{CKNP} and define, for every $y\in\R^d$,
the shifted Gaussian noise $\eta_y$, and the corresponding Gaussian
family of random variables in the following manner:
For all  $\varphi\in \mathcal{H}$ and $y\in\R^d$ we  set
\[
	\eta_y(\varphi) := \eta(\varphi_y),
\]
where $\varphi_y(s\,,x):= \varphi(s\,,x-y)$.
We remark that the elements of the Hilbert space $\mathcal{H}_0$ are generalized functions,
so the shifted element $\varphi_y$ should be defined  as in distribution
theory: $(\varphi_y\,,\psi):=(\varphi\,,\psi_{-y})$ for every smooth function $\psi$ of rapid decrease.
By the Wiener isometry for Wiener stochastic integrals, $\eta_y$ has the same law as $\eta$
for every $y\in\R^d$. As in Chen et al \cite[Lemma 7.1]{CKNP},
we define a family of shift operators $\{\theta_y\}_{y\in\R^d}$ that act on
every random variable $F\in L^2(\Omega)$ to create new random variables
$\{F\circ\theta_y\}_{y\in\R^d}$ in $L^2(\Omega)$ as follows:
\[
	(F\circ\theta_y)(\eta):=F(\eta_y)
	\qquad\text{for every $y\in\R^d$}.
\]
\begin{lemma}\label{L:D:shift}
	Choose and fix some $y\in\R^d$ and $F\in\mathbb{D}^{1,2}$.
	Then, a.s., $D (F \circ\theta_y ) = (D F) _y\circ\theta_y$.
\end{lemma}

\begin{proof}
	We first examine the case that $F(\eta)=\eta(\varphi)$ for some $\varphi\in\mathcal{H}$.
	On one hand, $(D F)_y=\varphi_y$.
	Since $\varphi$ is non random, this proves that $(DF)_y\circ\theta_a=
	\varphi_y$ a.s.  for all $a\in\R^d$. This holds in particular when $a=y$. On the other hand,
	$F\circ\theta_y=\eta_y(\varphi)= \eta(\varphi_y)$
	and hence $D(F\circ\theta_y)=\varphi_y$. This proves the lemma in the case
	that $F=\eta(\varphi)$.

	The preceding special case and the chain rule of Malliavin calculus (see
	Nualart \cite[Proposition 1.2.3]{Nualart}) together imply that the lemma holds also when
	$F (\eta) = \prod_{j=1}^k \Phi_j(\eta(\varphi_j))$ when
	$\Phi_1,\ldots,\Phi_k:\R\to\R$ are smooth and grow at most polynomially,
	and $\varphi_1,\ldots,\varphi_k\in\mathcal{H}$. The general result follows from this case
	and density; see Nualart \cite[Proposition 1.1.1]{Nualart}.
\end{proof}

\subsection{An abstract ergodicity result}

In \cite{CKNP} we proved that the infinite dimensional
random variable $u(t)$ is [spatially] ergodic for every $t>0$.
The following is a non-trivial variation on that and uses similar ideas.
\[
	\text{Throughout this section, we suspend the finiteness portion of assumption \eqref{E:f_finite}.}
\]
Because of the above, the Fourier transform $\hat{f}$ of $f$ is a positive-definite
tempered measure that need not be bounded.

\begin{theorem}\label{T:ergodic}
	To every random variable $G\in L^2(\Omega)$ we associate a random field ---
	also denoted by $G$ --- via
	$G(x) := G\circ\theta_x$ for all $x\in\R^d$. Then,
	$\{G(x)\}_{x\in\R^d}$ is stationary. Moreover, the following are equivalent:
	\begin{compactenum}
		\item  $\{G(x\}_{x\in\R^d}$ is ergodic for every $G\in L^2(\Omega)$
			such that $x\mapsto G(x)$ is continuous in probability;
		\item $\hat{f}\{0\}=0$.
	\end{compactenum}
\end{theorem}

The  spectral condition 2, $\hat{f}\{0\}=0$, is equivalent to either one of the following:
\begin{compactenum}
	\item $\hat{f}$ has no atoms;
	\item $f\{x\in\R^d:\, \|x\|<r\}=o(r^d)$ as $r\to\infty$.
\end{compactenum}
See Chen et al \cite{CKNP} for details.

Before we prove Theorem \ref{T:ergodic}, we point out a further generalization which is
noteworthy (though we will not need it here).

\begin{remark}
	One can make only small adjustments to the proof of Theorem \ref{T:ergodic}
	in order to see that if $\bm{G}=(G_1\,,\ldots,G_m)$ is an $m$-dimensional
	random vector with $G_i\in L^2(\Omega)$ for every $i=1,\ldots,m$, then
	$\{\bm{G}\circ\theta_x\}_{x\in\R^d}$ is stationary and ergodic, where
	$(\bm{G}\circ\theta_x)(\eta) := \bm{G}(\eta_x)$, the same as in the case $m=1$. This
	particular phrasing of Theorem \ref{T:ergodic} extends the ergodicity result of
	\cite{CKNP} since it was shown in the latter reference that $u(t\,,x+y)=u(t\,,x)\circ\theta_y$
	a.s.\ for all $t\ge0$ and $x,y\in\R^d$. In fact, this phrasing can be viewed to be an infinite-dimensional
	extension of a classical result of Maruyama \cite{Maruyama}; see also Dym and McKean \cite{DymMcKean}.
\end{remark}

\begin{proof}[Proof of Theorem \ref{T:ergodic}]
	Stationarity  is immediate;  we  prove only ergodicity.

	First, suppose that $\{G(x)\}_{x\in\R^d}$
	is ergodic for every $G\in L^2(\Omega)$ such that $x\mapsto G(x)$
	is continuous in probability.\footnote{Continuity
		in probability is here only to ensure that $x\mapsto G(x)$ has a Lebesgue-measurable version.}
	Consider random variables of the form $G=\eta(\bm{1}_{[0,1]}\otimes\varphi)$
	where $\varphi\in\mathscr{S}(\R^d)$ is a probability density
	function. In this case, $G(x)=\int_{(0,1)\times\R^d}\varphi(y-x)\,\eta(\d s\,\d y)$
	is continuous in $L^2(\Omega)$ and hence in probability. Indeed,
	\[
		\Cov[G(x)\,,G(0)] = \int_{\R^d} \varphi(y-x)(\varphi*f)(y)\,\d y
		\qquad\text{for all $x\in\R^d$}.
	\]
	Consequently,
	$\| G(x)-G(x')\|_2^2 = 2\int_{\R^d} [\varphi(y)-\varphi(y-x+x')
	](\varphi*f)(y)\,\d y\to 0$ as $x'\to x$. Also, the assumed ergodicity implies
	among other things that
	\[
		\left\| N^{-d}\int_{[0,N]^d} G(x)\,\d x\right\|_2^2
		=\left( I_N*\tilde{I}_N*\tilde{\varphi}*\varphi*f\right)(0)\to0
		\qquad\text{as $N\to\infty$},
	\]
	where $I_N := N^{-d}\bm{1}_{[0,N]^d}$ and $\tilde{Q}(z):=Q(-z)$
	for all $z\in\R^d$ and all $Q:\R^d\to\R$. Because
	\[
		\left( I_N*\tilde{I}_N*\tilde{\varphi}*\varphi*f\right)(0)
		= \frac{1}{(2\pi)^d}\int_{\R^d} |\hat{\varphi}(z)|^2\prod_{j=1}^d
		\left(\frac{1-\cos(N z_j)}{(Nz_j)^2}\right)
		\hat{f}(\d z),
	\]
	and $\hat{\varphi}(0)=1$, the dominated convergence theorem implies that
	$\hat{f}\{0\}=0$.

	For the more interesting half of the proof we assume $\hat{f}\{0\}=0$,
	and extend the ideas of the proof of Lemma 7.2 of Chen et al \cite{CKNP}.

	Define $\mathcal{G}(x) := \prod_{j=1}^k g_j ( G(x+\zeta^j))$, where
	$k\ge1$,   $\zeta^1,\ldots,\zeta^k\in\R^d$, and
	$g_1,\ldots,g_k:\R\to\R$ are Lipschitz-continuous and bounded
	functions that satisfy  $g_j(0)=0$ and
	$ \lip(g_j)=1$ for every $j=1,\ldots,k$. Since $\mathcal{G}$ is
	a stationary random field, the ergodic theorem reduces the problem to proving the following:
	\begin{equation}\label{E:cond_var}
		V_N(G) := \Var\left( N^{-d}  \int_{[0,N]^d}   \mathcal{G}(x) \,\d x\right)
		\to0\quad\text{as $N\to\infty$}.
	\end{equation}
	See \cite[Lemma 7.2 ]{CKNP} for the details of this argument.

	The bulk of the proof is concerned with the proposition in the
	special case that the random variable $G$ has the form,
	\begin{equation}\label{E:G}
		G  =  h(   \eta(\psi_1)\,,\ldots\,,\eta(\psi_m) ),
	\end{equation}
	where $h\in C_c^\infty(\R^m)$ and  $\psi_1,\ldots,\psi_m \in C_c(\R_+\times \R^d)$.
	Let us first understand why it is enough to study $G$'s of the form \eqref{E:G}.

	Given an arbitrary random variable $G\in L^2(\Omega)$ and a number $\varepsilon>0$
	we can find a random variable $\bar{G}$ that can be written as the right-hand side of \eqref{E:G}
	and satisfies $\|G-\bar{G}\|_2\le\varepsilon$; see
	Nualart \cite[Proposition 1.1.1]{Nualart}. The argument
	that we used at the very beginning portion of the proof can be recycled to
	show that $x\mapsto \bar{G}(x):=\bar{G}\circ\theta_x$  is continuous
	in $L^2(\Omega)$ and hence in probability. Let $\bar{\mathcal{G}}$ denote the analogue
	of $\mathcal{G}$ but with $G$ replaced by $\bar{G}$ everywhere.
	By the triangle inequality,
	\begin{align*}
		\left| \sqrt{V_N(G)} - \sqrt{V_N(\bar{G})}\right| &\le
			\left\| \left( N^{-d}\int_{[0,N]^d}\mathcal{G}(x)\,\d x - \E \mathcal{G}(0)\right)
			- \left( N^{-d}\int_{[0,N]^d}\bar{\mathcal{G}}(x)\,\d x - \E\bar{\mathcal{G}}(0)\right)
			\right\|_2\\
		&\le 2\varepsilon.
	\end{align*}
	If we could prove \eqref{E:cond_var} for all random variables of the form \eqref{E:G},
	then $\lim_{N\to\infty}V_N(\bar{G})=0$ whence
	$\limsup_{N\to\infty}V_N(G)\le4\varepsilon^2$
	by the above. Since $\varepsilon>0$ is arbitrary this shows that it is enough to prove
	\eqref{E:cond_var} for random variables of the form \eqref{E:G}.

	Suppose now that $G$ has the form \eqref{E:G}.
	By the Poincar\'e inequality  (\ref{E:Poincare}),
	\[
		V_N(G) \le  \E \left( \left\| N^{-d}
		\int_{[0,N]^d}  D \mathcal{G}(x) \,\d x\right\|_{\mathcal{H}}^2 \right)\\
		= N^{-2d} \iint_{[0,N]^{2d}}\d x\,\d y\
		\E \left[  \langle D \mathcal{G}(x), D\mathcal{G}(y)  \rangle_{\mathcal{H}} \right].
	\]
	Before we study the expectation inside the double integral, note using the chain rule
	of the Malliavin calculus (see Nualart \cite[Proposition 1.2.3]{Nualart}) that
	\[
	 	D_{s,z} \mathcal{G}(x)
		=  \sum_{j_0=1}^k  \left(\prod_{j=1, j \not= j_0}^{k}
		g_j ( G(x+\zeta^j))  \right) D_{s,z} G( x+ \zeta^{j_0}),
	\]
	and observe that
	$D_{s,z} G( x+ \zeta^{j_0})
	=   \sum_{i=1}^m
	(\partial_i h) (Y)  \psi_i(s,z-x-\zeta^{j_0})$,
	where $Y$ denotes the $m$-dimensional Gaussian random vector
	$( \eta_{x+\zeta^{j_1}}( \psi_1  ) , \dots,
	   \eta_{x+\zeta^{j_m}}( \psi_m )).$
	Thus,
	\begin{align*}
		\left\| D_{s,z}\mathcal{G}(x) \right\|_2 &\le \lip(h)
			\sum_{j_0=1}^k \left\|\prod_{j=1, j \not= j_0}^{k}
			g_j ( G(x+\zeta^j))  \right\|_2 \sum_{i=1}^m
			\left| \psi_i(s\,,z-x-\zeta^{j_0})\right|\\
		&\le\lip(h)\sup_{\substack{1\le j\le k\\
			w\in\R}} |g_j(w)|^{k-1}\sum_{\substack{1\le j_0\le k\\
			1\le i\le m}}
			\left| \psi_i(s\,,z-x-\zeta^{j_0})\right|
			:= L\sum_{\substack{1\le j_0\le k\\
			1\le i\le m}}\left| \psi_i(s\,,z-x-\zeta^{j_0})\right|.
	\end{align*}
	Consequently, the Walsh--isometry for stochastic integrals yields
	\begin{align*}
		&\E \left[  \langle D \mathcal{G}(x), D\mathcal{G}(y)  \rangle_{\mathcal{H}} \right]
			=  \int_0^t\d s\int_{\R^d}\d a\int_{\R^d}
			f(\d b)\ \E\left[ D_{s,a}\mathcal{G}(x) D_{s,a-b}\mathcal{G}(y)\right]\\
		&\le L^2\sum_{\substack{1\le j_0,j_1\le k\\1\le i_0,i_1\le m}}
			\int_0^t\d s\int_{\R^d}\d a\int_{\R^d}f(\d b)\
			\left| \psi_{i_0}(s\,,a-x-\zeta^{j_0})
			\psi_{i_1}(s\,,a-b-y-\zeta^{j_1})\right|.
	\end{align*}
	Recall that $I_N := N^{-d} \bm{1}_{[0,N]^d}$, and integrate the preceding displayed expression
	$[I_N(x)\,\d x I_N(y)\,\d y]$ in order to deduce from Fubini's theorem that
	\[
		V_N(G) \le
		L^2\sum_{\substack{1\le j_0,j_1\le k\\1\le i_0,i_1\le m}}
		\int_0^t\d s
		\left( I_N* \tilde{I}_N * f * |\psi_{i_0}(s)| *
		|\tilde{\psi}_{i_1}(s)| \right)\left( \zeta^{j_0}-\zeta^{j_1}\right),
	\]
	where $Q(s)=Q(s\,,\bullet)$ for all $s\ge0$
	and  space-time functions $Q:\R_+\times\R^d\to\R$, and
	$\tilde{q}(x)=q(-x)$ for all $x\in\R^d$ and $q:\R^d\to\R$. Since $I_N*\tilde{I}_N*f$ is a
	continuous, positive definite function, it is maximized at the origin. This is a consequence of
	the  Bochner--Hergloz theorem of classical Fourier analysis, and found in
	introductory probability textbooks.
	This yields
	\[
		V_N(G) \le  (kL)^2\left( I_N*\tilde{I}_N*f \right)(0) \int_0^t
		\left( \sum_{i=1}^m\|\tilde{\psi}_i(s)\|_{L^1(\R^d)}\right)^2
		\,\d s.
	\]
	According to Chen et al \cite[Proof of Proposition 3.7]{CKNP},
	$(I_N*\tilde{I}_N*f)(0)\to0$  as $N\to\infty$ because
	$\hat{f}\{0\}=0$. This establishes
	(\ref{E:cond_var}) when $G$ has the form \eqref{E:G},
	and completes the proof of the proposition.
\end{proof}

\section{CLT: Proof of Theorem \ref{T:TV}} \label{S:TV}

Before we begin the proof properly,
let us pause to make two elementary remarks, one from real analysis and the other
from information theory.

\begin{remark}\label{R:Lip}
	If $g\in\lip$ then according to Rademacher's theorem,
	the weak derivative $g'$ of $g$ is a.e.-defined
	and satisfies $|g'|\le\lip(g)$ a.e.;
	see for example Federer \cite[Theorem 3.1.6]{Federer}.
\end{remark}

\begin{remark}\label{R:IT}
	If $0<c_2<c_1$, then
	\[
		d_{\rm TV}[{\rm N}(0\,,c_1)\,,{\rm N}(0\,,c_2)]
		\le\frac12\sqrt{\frac{c_1-c_2}{c_2}}.
	\]
	One can prove this quickly using Pinsker's inequality (see,
	for example, Cover and Thomas \cite{CoverThomas}).
	Indeed, Pinsker's inequality tells us that
	$2[d_{\rm TV}(\sqrt{c_1}Z\,,\sqrt{c_2}Z)]^2$ is at most
	the relative entropy of ${\rm N}(0\,,c_1)$ with respect to
	${\rm N}(0\,,c_2)$, which is
	$-\vert\log(c_1/c_2)\vert + (c_1-c_2)/(2c_2) \le (c_1-c_2)/(2c_2).$
\end{remark}

According to Chen et al \cite[Lemma 4.3]{CKNP2} or \eqref{E:D<p2}, it is easy to see that in case of Lipschitz $g$,
\begin{align}\label{E:SNallN}
	\mathcal{S}_{N,t}(g)\in \cap_{k\ge2}
	\mathbb{D}^{1,k}\qquad\text{for all $N>0$,}
\end{align}
and
\begin{equation} \label{E:ec2}
	D_{s,z} \mathcal{S}_{N,t}(g) = N^{-d}  \int_{[0,N]^d} g'(u(t\,,x)) D_{s,z}u(t\,,x) \,\d x.
\end{equation}
However, when $g$ goes beyond the globally Lipschitz functions, we need to identify those $g$'s
that satisfy $g(u(t\,,x))\in \mathbb{D}^{1,k}$ for all $k\ge 2$, $t>0$, and $x\in\R^d$, in order to apply Clark-Ocone formula \eqref{E:Clark-Ocone}.
The following lemma serves this purpose.
Indeed, an application of part (i) of Lemma \ref{L:Sobolev} below shows that under conditions of part (ii) of Theorem \ref{T:TV},
$\mathcal{S}_{N,t}(g)\in \mathbb{D}^{1,4} \subseteq  \mathbb{D}^{1,2}$.

\begin{lemma} \label{L:Sobolev}
Let $I=\R$ or $I=(0,\infty)$.
	Suppose that $F$ is a   random variable taking values in $I$ that belongs to
 $\cap_{k\ge 1} \mathbb{D}^{1,k}$.
	Let $g: I\to \R$ be a measurable function such that
	$ \| g(F) \| _k <\infty$ for some $k \ge 2$. Then:
	\begin{itemize}
		\item[(i)] If $g\in C^1(I)$ satisfies  $\|  g'(F) \|_{p} <\infty$ for some $p>k$,
			then  $g(F) \in \mathbb{D}^{1,k}$ and
			\begin{equation} \label{E:ChainRule}
				D[g(F)]= g'(F) DF.
			\end{equation}
		\item[(ii)]  If $g\in C^2(I)$ satisfies  $\|  g'(F) \|_{p} +\|   g''(F) \|_{p} <\infty$ for some $p>k$,
			then  $g(F) \in \mathbb{D}^{2,k}$ and
			\[
				D^2[g(F)]= g''(F) DF \otimes DF + g'(F) D^2 F.
			\]
	\end{itemize}
\end{lemma}

\begin{proof}
	 We  prove {\it(i)}; {\it (ii)} is proved
	similarly and so we skip the proof of that part.
	For every $M\in\mathbb{N}$ and $x>0$,
	we set $g_M(x)=   ( g(x) \wedge M) \vee (-M)$. Because $g_M$ is Lipschitz,
	by the chain rule for the Malliavin calculus $g_M(F) \in \mathbb{D}^{1,k}$ and
	\begin{equation} \label{E_:ChainRule}
		D[g_M(F)]= g'(F) \mathbf{1}_{\{ |F| \le M\}} DF \qquad\text{a.s.}
	\end{equation}
	Because $ \| g(F) \| _k <\infty$, 	$g_M(F)$ converges to $g(F)$
	as $M\to \infty$ in $L^k(\Omega)$. Moreover, by H\"older's inequality
	and \eqref{E_:ChainRule},  $D[g_M(F)]$ converges in $L^k( \Omega\,; \mathcal{H})$
	to the right-hand side of \eqref{E:ChainRule} as $M\to\infty$.
	This completes the proof of {\it (i)}.
\end{proof}

Once the property $\mathcal{S}_{N,t}(g)\in \mathbb{D}^{1,2}$ is established, one can apply the Clark-Ocone formula  (\ref{E:Clark-Ocone})
to see that a.s., $\mathcal{S}_{N,t} (g)=
\int_{(0,t) \times \R^d} \E[ D_{s,z} \mathcal{S}_{N,t}(g) \mid \mathcal{F}_s]\,
\eta(\d s\, \d z).$
Since the Walsh stochastic integral is an extension of the divergence operator $\delta$,
it follows that (a.s.)
\begin{equation}\label{E:HatV}
	\mathcal{S}_{N,t}(g)=\delta( v_{N,t}(g))
	\quad\text{with}\quad
	v_{N,t}(g) (s\,,z)
	:= \E[D_{s,z}\mathcal{S}_{N,t}(g) \mid \mathcal{F}_s] \mathbf{1}_{[0,t]}(s).
\end{equation}
Throughout the paper,we denote by
\begin{align} \label{E:KN}
	K_{N,t_1,t_2}:= N^d \left<D\mathcal{S}_{N,t_1}(g) ~,~ v_{N,t_2}(g) \right>_{\mathcal{H}}
	\quad\text{and}\quad
	K_{N} := K_{N,t,t}.
\end{align}

The next lemma shows that properties \eqref{E:Var_Lim} and \eqref{E:Var_finite} hold under assumptions (ii)  or (iii) of Theorem  \ref{T:TV}.
\begin{lemma}  \label{lem4.4}
Suppose that either $g\in C^1(\R)$ or $\sigma(0)=0$, $g\in C^1(0,\infty)$and $f$ satisfies the reinforced Dalang's condition \eqref{E:Dalang2} . Suppose, in addition that \eqref{E:MmCond_TV} holds
	for some number $k>2$.
	Then,  \eqref{E:Var_finite} and \eqref{E:Var_Lim} hold  for all $t_1, t_2\ge0$.
\end{lemma}

\begin{proof}
Let us first show \eqref{E:Var_finite}.
First of all, note that
	$g(u(t\,,x))\in\mathbb{D}^{1,2}$
	thanks to Lemma \ref{L:Sobolev} and the fact that \eqref{E:MmCond_TV} holds for some $k>2$.
	This proves that the Clark-Ocone formula \eqref{E:Clark-Ocone} is valid.
	Applying Clark-Ocone formila yields
	\begin{align*}
		\Cov\left[g(u(t_1,x)),g(u(t_2,0))\right] &=
		\E \Bigg( \int_0^\infty \d s \int_{\R^d} \d y \int_{\R^d} f(\d z) \E \left[ g'(u(t_1,x)) D_{s,y} u(t_1,x) \mathbf{1}_{[0,t_1]}(s) | \mathcal{F}_s \right]\\
		 & \qquad  \times \E \left[ g'(u(t_2,0)) D_{s,y} u(t_2,0) \mathbf{1}_{[0,t_2]}(s) | \mathcal{F}_s \right] \Bigg).
 	\end{align*}
	 Finally, H\"older's inequality and the estimate \eqref{E:D<p2} together yield
	 \begin{align*}
		\int_{\R^d} | \Cov\left[g(u(t_1,x)),g(u(t_2,0))\right]  | \d x
		& \le C  \int_0^{t_1\wedge t_2} \d s \int_{\R^{3d}} \d x \d y f(\d z) \:
			\bm{p}_{t_1-s}(x-y) \bm{p}_{t_2-s}(x-y-z)\\
		& = C f(\R^d) (t_1\wedge t_2)<\infty.
	\end{align*}
	The equality in \eqref{E:Var_Lim} follows from the following claim:
	\begin{align} \label{E2:Var_Lim}
		\lim_{N\to\infty} \Cov \left( N^{d/2} \mathcal{S}_{N,t_1} (g),  N^{d/2} \mathcal{S}_{N,t_2} (g) \right)
		= \int_{\R^d} \Cov \left[g(u(t_1,x)), g(u(t_2,0))\right]\d x.
	\end{align}
	Indeed, by the stationarity of the solution in its spatial variable, we see that
	\begin{align*}
		\Cov \left( N^{d/2}\mathcal{S}_{N,t_1}(g),N^{d/2} \mathcal{S}_{N,t_2}(g) \right)
		& = \frac{1}{N^{d}} \int_{[0,N]^d}\d x\int_{[0,N]^d}\d y\: \Cov\left(g(u(t_1,x)), g(u(t_2,y))\right)\\
		& = \int_{[0,N]^d}\Cov\left(g(u(t_1,z)), g(u(t_2,0))\right) \,\d z.
	\end{align*}
	Now let $N\to\infty$ to deduce \eqref{E2:Var_Lim} from \eqref{E:Var_finite}
	and the dominated convergence theorem.
\end{proof}

The following proposition is a basic ingredient in the proof of both Theorem \ref{T:TV} and \ref{T:FCLT}.

\begin{proposition} \label{P:L1Conv}
	Suppose that $f$ satisfies \eqref{E:f_finite}. Under each of the three cases (i), (ii) and (iii) of Theorem \ref{T:TV},
	the following convergence holds true
	\begin{equation} \label{E:L1Conv}
		\lim_{N\rightarrow \infty}   {\rm Var} \left(N^d \left<
		D\mathcal{S}_{N,t_1}(g) ~,~ v_{N,t_2}(g) \right>_{\mathcal{H}}  \right)=0,
		\qquad\text{for any $t_1,t_2\ge 0$.}
	\end{equation}
\end{proposition}

\begin{proof}[Proof of   Proposition \ref{P:L1Conv} in Case (i)]
	We may use \eqref{E:HatV} together with Fubini's theorem
	and the definition of $\mathcal{H}$ in order
	to decompose the quantity of interest $K_{N,t_1,t_2}$ defined in \eqref{E:KN}  as follows:
	\begin{align*}
		K_{N,t_1,t_2} &= N^d  \int_0^{t_1\wedge t_2} \d s  \int_{\R^d} \d z\
			D_{s, z}\mathcal{S}_{N,t_1}(g)
			\left(\E[D_{s,  \bullet }\mathcal{S}_{N,t_2}(g)\mid \mathcal{F}_s] *f \right)(z)\\
		&= N^{d} \int_{[0,N]^{d}}\d x \int_{[0,N]^{d}}\d y
			\int_0^{t_1 \wedge t_2}\d s  \int_{\R^d} \d z\
			D_{s, z}g(u(t_1\,,x))
			\left(  \E[D_{s, \bullet }g(u(t_2\,,y) \mid \mathcal{F}_s] *f \right)(z)\\
		&= I_{1,N} - I_{2,N},
	\end{align*}
	where
	\begin{align*}
		I_{1,N} &:= N^{d} \int_{[0,N]^{d}}\d x \int_{\R^d}\d y
			\int_0^{t_1\wedge t_2}\d s  \int_{\R^d} \d z\
			D_{s, z}g(u(t_1\,,x))
			\left(  \E[D_{s, \bullet }g(u(t_2\,,y) \mid \mathcal{F}_s] *f \right)(z),\\
		I_{2,N} &:= N^{d} \int_{[0,N]^{d}}\d x \int_{\R^d\setminus[0,N]^{d}}\d y
			\int_0^{t_1\wedge t_2}\d s  \int_{\R^d} \d z\
			D_{s, z}g(u(t_1\,,x))
			\left(  \E[D_{s, \bullet }g(u(t_2\,,y) \mid \mathcal{F}_s] *f \right)(z).
	\end{align*}

	The following shows that the asymptotic behavior of $K_{N,t_1,t_2}$ is the same as that of $I_{1,N}$.
	\bigskip

	\noindent{\bf Claim 1.} {\it
		$I_{2,N}\to 0$ in $L^2(\Omega)$ as $N\to\infty$.
		Thus, \eqref{E:L1Conv} holds if $I_{1,N} -\E I_{1,N}
		\to0$ in $L^2(\Omega)$ as $N\to\infty$.}\\

	\noindent{\bf Proof of Claim 1.} The second assertion of Claim 1 is a ready consequence of
	the first assertion and Jensen's inequality. It suffices to prove the first assertion then.
	We may appeal to \eqref{E:D<p2} with $c=C_{t_1,k,\varepsilon,\sigma}C_{t_2,k,\varepsilon,\sigma}$, $k=4$,
	and   $\varepsilon=1/2$ (say) as follows:
	\begin{align*}
		\| I_{2,N} \|_2& \le
			N^{-d}  \int_{[0,N]^d}   \d x \int_{\R^d\setminus[0,N]^d}\d y
			\int_0^{t_1\wedge t_2}\d s  \int_{\R^d} \d z\
			\|D_{s, z}g(u(t_1\,,x))\|_4
			\left(  \|D_{s, \bullet }g(u(t_2\,,y)\|_4 *f\right)(z)\\
 		 & \le  \frac{c[{\rm Lip } (g)]^2}{N^d}
			\int_{[0,N]^d}\d x \int_{\R^d\setminus[0,N]^d}\d y
			\int_0^{t_1\wedge t_2}\d s  \int_{\R^d} \d z\
			\bm{p}_{t_1-s}(x-z) (   \bm{p}_{t_2-s}*f)(z-y)\\
		&= \frac{c[{\rm Lip } (g)]^2}{N^d} \int_{[0,N]^d}\d x \int_{\R^d}\d y
			\int_0^{t_1\wedge t_2}\d s\  \bm{1}_{\R^d\setminus[0,N]^d}(y)
			\left(   \bm{p}_{t_1+t_2-2s}*f \right)(x-y)\\
		&= c[{\rm Lip } (g)]^2\int_0^{t_1\wedge t_2}\d s\int_{\R^d}\d z
			\left(   \bm{p}_{t_1+t_2-2s}*f \right)(z) \int_{\R^d}
			\d x\  \bm{1}_{[0,1]^d \cap (\R^d\setminus[0,1]^d+ \frac{z}{N})}(x),
	\end{align*}
	owing to the semigroup property of the heat kernel.
	Now $f$ is a finite  measure [see \eqref{E:f_finite}] and $ \int_{\R^d}
	\bm{1}_{[0,1]^d \cap (\R^d\setminus[0,1]^d+ \frac{z}{N})}(x)\,\d x
	\to0$ as $N\to\infty$ for every $z\in\R^d$. Therefore,
	the dominated convergence theorem ensures that the preceding
	displayed expression converges to zero as $N\to\infty$.
	This proves Claim 1.\qed\\

	In order to investigate the asymptotic behavior of
	$I_{1,N}$, we first define
	\[
		H_s(z) := \int_{\R^d} \E\left[D_{s, z}g(u(t_2\,,\alpha))
		\mid \mathcal{F}_s\right]\d\alpha
		\qquad\text{for all $s\in[0\,,t_2]$ and $z\in\R^d$}.
	\]
	Since $D_{s,z} g(u(t_2\,,y))=g'(u(t_2\,,y)) D_{s,z} u(t_2\,,y)$ for a.e.\
	$(s\,,z)\in(0\,,t_2)\times\R^d$, and since $y\mapsto D_{s,z}u(t_2\,,y)$ is continuous in $L^2(\Omega)$
	for a.e.\ $(s\,,z)\in(0\,,t_2)\times\R^d$ (this  can be proved by the same arguments as in \cite[Lemma 19]{Dalang1999}), Lemma \ref{L:Du} (applied
	with $\varepsilon=1/2$, say) implies
	that $H := \{ H_{s}(z) \} _{(s,z) \in [0,t] \times \R^d}$ is  an adapted   random field
	that satisfies
	\begin{equation}\label{E:HLip}
		\|    H_{s,z} \|_k \le
		C_{t_2,k,1/2,\sigma } {\rm Lip } (g)  \int_{\R^d}  \bm{p}_{t_2-s} (y-z)\,\d y
		= C_{t_2,k,1/2,\sigma} {\rm Lip } (g),
	\end{equation}
	uniformly for all $k\ge2$, $s\in(0\,,t_2)$, and $z\in\R^d$.
 	Our interest in the random field $H$ stems from the fact that $I_{1,N}$ can be written
	in terms of $H$ as follows:
	\begin{equation}\label{E:I_Psi}
		I_{1,N} = N^{-d}  \int_{[0,N]^d}\d x\int_0^{t_1 \wedge t_2}\d s
		\int_{\R^d}\d z\  D_{s, z}g(u(t_1\,,x)) ( H_{s} *f )(z)
 		= N^{-d}  \int_{[0,N]^d}\Psi(x)\,\d x,
	\end{equation}
	where
	\begin{equation} \label{E:Psi}
		\Psi(x) := \int_0^{t_1 \wedge t_2} \d s  \int_{\R^d} \d z\  D_{s, z}g(u(t_1\,,x))( H_{s} *f )(z).
	\end{equation}
	Similar estimates to those used in the proof of  \cite[Theorem 6.4]{CKNP}
	 and  \cite[Lemma 4.2]{CKNP2} show that $(s\,,z)\mapsto H_s(z)$
	has a version that is
	continuous in probability on $[t_1 \wedge t_2] \times \R^d$. In particular,  there also exists a measurable version
	of the random field $\Psi$, and it can be written as $ \Psi(x)= g'(u(t_1\,,x)) \Psi_1(x) $,
	where
	$x\mapsto\Psi_1(x) := \int_0^{t_1}\d s  \int_{\R^d}\d z\  D_{s, z}u(t_1\,,x)( H_{s} *f )(z)$
	is continuous in probability. This and a simple extension of
	Doob's separability theorem \cite{Doob}
	together imply that $H$, $\Psi_1$, and $\Psi$ all have Lebesgue measurable versions,
	which we feel free to use.\\

	\noindent\textbf{Claim 2.} {\it For every $x,y\in\R^d$,
		$\Psi(x+y)=\Psi(x)\circ\theta_y$ a.s.}\\

	\noindent\textbf{Proof of Claim 2.}
	We may write
	\begin{align*}
	 	\Psi(x+y) &=\int_0^{t_1\wedge t_2}\d s\int_{\R^d}\d \alpha\int_{\R^d} \d z\
			 D_{s, z}g(u(t_1\,,x+y))  \left(  \E[D_{s, \bullet }g(u(t_2\,,\alpha))
			\mid \mathcal{F}_s] *f \right)(z)\\
		&= \int_0^{t_1\wedge t_2}\d s\int_{\R^d}\d \alpha\int_{\R^d} \d z\
			 D_{s, z}g(u(t_1\,,x+y))  \left(  \E[D_{s, \bullet }g(u(t_2\,,\alpha+y))
			\mid \mathcal{F}_s] *f \right)(z).
	\end{align*}
	We have proven in \cite[Lemma 7.1]{CKNP} that  for all $t\ge 0$, $u(t\,,x+y)=u(t\,,x)\circ\theta_y$ a.s.
	Therefore, by Lemma \ref{L:D:shift},
	\begin{align*}
	 	\Psi(x+y)& =\int_0^{t_1 \wedge t_2}\d s\int_{\R^d}\d \alpha\int_{\R^d} \d z\
		 	\left(D_{s, z-y} g(u(t_1\,,x))\right)\circ\theta_y \\
			 & \quad \times  \left(
			\E\left[\left\{D_{s, \bullet -y}g(u(t_2\,,\alpha))\right\}\circ\theta_y
			\mid \mathcal{F}_s\right] *f \right)(z)\\
		& = \int_0^{t_1 \wedge t_2}\d s\int_{\R^d}\d \alpha\int_{\R^d} \d z\
		 	\left(D_{s, z} g(u(t_1\,,x))\right)\circ\theta_y \left(
			\left\{\E[D_{s, \bullet }g(u(t_2\,,\alpha))\right\}\circ\theta_y
			\mid \mathcal{F}_s] *f \right)(z),
	\end{align*}
	after a change of variables $[z-y\leftrightarrow z]$. Now, $\mathcal{F}_s$ is generated
	by all random variables
	of the form $\int_{(0,t)\times\R^d} \psi(s\,,x)\,\eta_y(\d s\,\d x)$ as
	$\psi$ ranges over $\mathcal{H}$.\footnote{Point being that $\mathcal{F}_s$ is
		generated equally well by the infinite-dimensional Brownian motion
		$\{W_t\circ\theta_y\}_{t\ge0}$ for any $y\in\R^d$; compare with \eqref{E:W}.}
	Because of this observation, it
	is now a simple exercise in measure theory to verify that
	$\{\E[D_{s, \bullet }g(u(t_2\,,\alpha))\}\circ\theta_y\mid\mathcal{F}_s]=
	\E[D_{s, \bullet }g(u(t_2\,,\alpha))\mid\mathcal{F}_s]\circ\theta_y$ a.s.,
	which in turn implies that
	\[
		\Psi(x+y) = \left(\int_0^{t_1\wedge t_2}\d s\int_{\R^d}\d \alpha\int_{\R^d} \d z\
	 	D_{s, z} g(u(t_1\,,x))
		\left(\E\left[ D_{s, \bullet }g(u(t_2\,,\alpha))
		\mid \mathcal{F}_s \right] *f \right)(z)\right)\circ\theta_y.
	\]
	almost surely. This proves Claim 2.\qed\\

	\noindent\textbf{Claim 3.} {\it $\Psi(0)\in L^2(\Omega)$}.\\

	\noindent\textbf{Proof of Claim 3.}
	We proceed as we did for Claim 1. Thanks to
	the definition \eqref{E:Psi} of $\Psi$ and Lemma \ref{L:Du},
	\begin{align*}
		\|\Psi(0)\|_2 &\le C_{t_1,4,1/2,\sigma}\lip(g)f(\R^d)\int_0^{t_1 \wedge t_2}\d s  \int_{\R^d} \d z
			\left\| D_{s, z}g(u(t_2\,,0))\right\|_4\\
		&\le C_{t_1,4,1/2,\sigma}[\lip(g)]^2f(\R^d)\int_0^{t_1\wedge t_2}\d s  \int_{\R^d} \d z
		\left\| D_{s, z}u(t_2\,,0)\right\|_4\\
		&\le  C_{t_1,4,1/2,\sigma} C_{t_2,4,1/2,\sigma} \left[ \lip(g)\right]^2f(\R^d)(t_1 \wedge t_2),
	\end{align*}
	where we used \eqref{E:HLip} in the first line (with $k=4$),
	the chain rule of Malliavin calculus (see Nualart \cite[Proposition 1.2.3]{Nualart}) in
	the second line, and Lemma \ref{L:Du} in the last line. This verifies Claim 3.\qed\\

	We can now complete the proof of the proposition in Case (i) as follows.
	Thanks to Claims 2 and 3, we may deduce from Theorem \ref{T:ergodic} that
	the $L^2(\Omega)$-random field $\{\Psi(x)\}_{x\in\R^d}$ is stationary and ergodic. Thus,
	\eqref{E:I_Psi} and the ergodic theorem together ensure that
	$I_{1,N} \to \E\Psi(0)$ as $N\to\infty$, a.s.\ and in $L^2(\Omega)$.
	This implies, among other things, that
	$I_{1,N}-\E I_{1,N} \to 0$ in $L^2(\Omega)$ as $N\to\infty$. Now apply Claim 1 to
	deduce \eqref{E:L1Conv} and hence Case (i) of Proposition \ref{P:L1Conv}.
\end{proof}

\begin{proof}[Proof of  Proposition \ref{P:L1Conv} in Cases (ii) and (iii)]
Cases (ii) and (iii) can be proved similarly.
	For {\bf Claim 1}, notice that for $i=1,2$,
	\[
		\Norm{D_{s,z}g(u(t_i\,,x))}_4
		\le
		\Norm{g'(u(t_i\,,x))}_{p}
		\Norm{D_{s,z}u(t_i\,,x)}_{p'}
		=
		\Norm{g'(u(t_i\,,0))}_{p}
		\Norm{D_{s,z}u(t_i\,,x)}_{p'},
	\]
	where  $p^{-1}+(p')^{-1} = 1/4$,
	and  the last equality is due to the stationarity of $x \mapsto u(t_i\,,x)$. Hence, the moment bound for
	$I_{2,N}$ is still valid with the constant
	$c[\lip(g)]^2 = C_{t_1,4,1/2,\sigma}C_{t_2,4,1/2,\sigma} [\lip(g)]^2$
	replaced by
	\[
		C_{t_1,p',1/2,\sigma}C_{t_2,p',1/2,\sigma} \Norm{g'(u(t_1,0))}_{p}\Norm{g'(u(t_2,0))}_{p}.
	\]
	The above constant is finite thanks to \eqref{E:MmCond_TV}. Hence, {\bf Claim 1} is true.
	{\bf Claim 2} still holds without any change. As for {\bf Claim 3}, the moment bounds for $\Psi(0)$
	are replaced with the following, using similar arguments as above:
	\[
		\Norm{\Psi(0)}_2 \le
		C_{t_1,p',1/2,\sigma}
		C_{t_2,p',1/2,\sigma}
		\Norm{g'(u(t_1\,,0))}_{p}
		\Norm{g'(u(t_2\,,0))}_{p}
		f(\R^d)
		\left(t_1\wedge t_2\right).
	\]
	With these changes, the proof of Cases (ii) and (iii)  follows the same arguments as Case (i).
\end{proof}

We now proceed with the following.

\begin{proof}[Proof of Theorem \ref{T:TV}]
	We first consider Case (i).
	Throughout, we choose and fix some $t>0$ and recall that $\lim_{N\to\infty}
	\mathbf{B}_{N,t}(g)=\mathbf{B}_t(g)>0$. In particular, we observe that
	$\mathbf{B}_{N,t}(g)>0$ for all sufficiently large $N$.
	Also, Remark \ref{R:IT} shows that $\sqrt{\mathbf{B}_{N,t}(g)}Z$
	converges in total variation to $\sqrt{\mathbf{B}_t(g)} Z$ as $N\to\infty$.
	Therefore, it suffices to prove that
	\begin{equation}\label{E:goal}
		\mathscr{E}_N:= d_{\rm TV}\left( N^{d/2}\mathcal{S}_{N,t}(g)
		~,~ \sqrt{\mathbf{B}_{N,t}(g)} Z\right) \to 0
		\qquad\text{as $N\to\infty$}.
	\end{equation}
	In view of  the representation of  $\mathcal{S}_{N,t}(g)$ given in \eqref{E:HatV},
	we may therefore apply \eqref{E:SM1} to see that
	\begin{equation}\label{E:K}
		\mathscr{E}_N\le \frac{2}{\mathbf{B}_{N,t}(g)}
		\E\left( \left| \mathbf{B}_{N,t}(g) - N^d \left<
		D\mathcal{S}_{N,t}(g) ~,~ v_{N,t}(g) \right>_{\mathcal{H} }  \right| \right),
	\end{equation}
	valid for all $N$ large enough to ensure that $\mathbf{B}_{N,t}(g)>0$.
	According to \eqref{E:SM2},
	$\E(\< D\mathcal{S}_{N,t}(g)\,,v_{N,t}(g) \>_{\mathcal{H}}) =\mathbf{B}_{t,N}(g)$. Therefore,
	in light of \eqref{E:goal} and \eqref{E:K}, and owing to Jensen's inequality,
	it suffices to prove that the variance of $ N^d\left< D\mathcal{S}_{N,t}(g) ~,~ v_{N,t}(g) \right>_{\mathcal{H}}$
	converges to zero as $N\to \infty$, which has been shown in Proposition \ref{P:L1Conv}.

	For Cases (ii) adn (iii), the limit \eqref{E:TV} can be proved in the same way as above and   \eqref{E:Var_finite} and \eqref{E:Var_Lim} hold due to Lemma \ref{lem4.4}. Finally,
Property $\mathbf{B}_t(g)>0$ in Case (iii) is a
	direct consequence of  Corollary \ref{L:CovFinite} below and \eqref{E:Var_finite}
\end{proof}

Proposition \ref{P:associate}  below is devoted to establish   the nondegeneracy of $\mathrm{B}_t(g)$ in Case (iii) of Theorem \ref{T:TV}.
Our next result is an application of Lemma  \ref{T:Positive_D} and the Clark-Ocone formula \eqref{E:Clark-Ocone}.
In some sense, the following says that the solution $u$ to \eqref{E:SHE} has a strict association property; see Appendix for more exploration in this direction.

\begin{proposition} \label{P:associate}
	Let  $u$ denote the solution to \eqref{E:SHE}  subject to
	$u(0)\equiv1$, and
	suppose that $\sigma(0)=0$, $\sigma(x)\ne 0$ for all $x>0$,
	and $f$ satisfies  the reinforced Dalang's condition  \eqref{E:Dalang2} for some $\alpha\in(0\,,1]$.
	Choose and fix $n$  space-time points
	$(t_1,x_1),\dots, (t_n,x_n)\in(0\,,\infty)\times\R^d$, and
	real-valued functions $g_1, g_2 \in C^1((0\,,\infty)^n)$ such that
	for every $i=1,\ldots,n$ and $j=1,2$:
	\begin{compactenum}
		\item $\partial_i g_j>0$; and
		\item $g_j(u(t_1,x_1),\dots,u(t_n,x_n)) \in L^2(\Omega)$
		and $\partial_i g_j(u(t_1,x_1),\ldots,u(t_n,x_n))\in \cup_{p>2} L^p(\Omega)$.
	\end{compactenum}
	Then, we have the strict positivity result,
	\begin{align} \label{E:CovAss}
		\Cov\left[g_1(u(t_1,x_1),\dots,u(t_n,x_n))
		~,~g_2(u(t_1,x_1),\dots,u(t_n,x_n))\right]
		>0.
	\end{align}
\end{proposition}
\begin{proof}
	Denote the covariance in \eqref{E:CovAss} by $I$.  According to Lemma \ref{L:Sobolev}, Assumption 2
	ensures that $g_j(u(t_1,x_1),\dots,u(t_n,x_n))\in\mathbb{D}^{1,2}$ for $j=1,2$.
	Therefore, we may appeal to the Clark-Ocone formula \eqref{E:Clark-Ocone} in order to obtain
	\begin{align}\begin{aligned} \label{E:I}
		I = \E\Bigg( & \int_0^{\infty} \d s \int_{\R^d}\d y\int_{\R^d}  f(\d z) \\
			     & \times \E \left[\sum_{i=1}^{n} \partial_i g_1
			     (u(t_1,x_1),\dots,u(t_n,x_n)) D_{s,y} u(t_i,x_i) \mathbf{1}_{[0,t_i]}(s)
			     \ \bigg|\ \mathcal{F}_s \right]\\
		             & \times \E \left[\sum_{j=1}^{n}
		             \partial_j g_2(u(t_1,x_1),\dots,u(t_n,x_n)) D_{s,y+z} u(t_j,x_j)
		             \mathbf{1}_{[0,t_j]}(s) \ \bigg|\ \mathcal{F}_s \right] \Bigg).
	\end{aligned}\end{align}
	The integral $I$ is well defined and finite.
	Indeed, Assumption 2, H\"older's inequality and the estimate \eqref{E:D<p2}
	together yield
	\begin{align} \notag
		|I| \le & C \int_0^{\infty}\d s \int_{\R^d}\d y \int_{\R^d} f(\d z)
			\left(\sum_{i=1}^n \bm{p}_{t_i-s}(x_i-y)\mathbf{1}_{[0,t_i]}(s)\right)
			\left(\sum_{j=1}^n \bm{p}_{t_j-s}(x_j-y-z)\mathbf{1}_{[0,t_j]}(s)\right)\\  \label{E:EST1}
		\le &C \sum_{i,j=1}^n \int_0^{t_i\wedge t_j} \d s \int_{\R^d}\d y \int_{\R^d} f(\d z)\
			 \bm{p}_{t_i-s}(x_i-y)
		 	\bm{p}_{t_j-s}(x_j-y-z)< \infty.
	\end{align}

	Since $\sigma\left(0\right)=0$,  $u(0)\equiv1$, and $\sigma(1)\ne 0$, and because
	$f$ satisfies the reinforced Dalang's condition \eqref{E:Dalang2} for some $\alpha\in(0\,,1]$,
	one can apply Theorem 1.6 of  Chen and Huang \cite{CH19}
	in order to see  that $u(t\,,x)>0$ almost surely for all $t>0$ and $x\in\R^d$.
	Hence, by Assumption 1 and Theorem \ref{T:Positive_D}, the two conditional expectations
	in \eqref{E:I} are strictly positive a.s. This proves that
	$I>0$, which is another way to state the proposition.
\end{proof}

Notice that Proposition  \ref{P:associate} remains valid
if in its assumption 1 we require that all partial derivatives are strictly negative,
instead of strictly positive. In this way, we can include examples
such as $g(x_1,\dots,x_n)=\prod_{i=1}^{n}  x_i^{\alpha_i}$
where $\alpha_1,\ldots,\alpha_n<0$.

The following corollary is an immediate consequence of Proposition \ref{P:associate}.
\begin{corollary} \label{L:CovFinite}
	Let  $u$ denote the solution to \eqref{E:SHE} subject to $u(0)\equiv1$ and $\sigma(0)=0$, and
	suppose that $f$ satisfies both condition \eqref{E:f_finite} and the reinforced Dalang's condition \eqref{E:Dalang2} for some $\alpha\in(0\,,1]$.
	Choose and fix some $g\in C^1(0\,,\infty)$ that satisfies \eqref{E:MmCond_TV}
	for some number $k>2$.
	 Suppose, in addition:
	\begin{compactenum}
		\item $\sigma(x)\ne 0$ for all $x>0$; and
		\item $g'$ is either strictly positive or strictly negative on $(0\,,\infty)$;
	\end{compactenum}
	then $\mathbf{B}_{t_1,t_2}\left(g\right) >0$ for all $t_1, t_2>0$.
\end{corollary}

Note that all examples in \eqref{E:Example}
satisfy the conditions of the above corollary.

\section{Functional CLT: Proof of Theorem \ref{T:FCLT}}  \label{S:FCLT}

In order to show Theorem \ref{T:FCLT} we will establish the weak convergence of the finite-dimensional distributions, as well as tightness.
We start with a couple of technical lemmas.

\begin{lemma} \label{L:h1}
	Fix $0\le a< b \le \tau$, $k\ge 2$. Let  $\{\Phi(s\,,y)\}_{ (s,y) \in [0,T]\times \R^d}$
	be an adapted and measurable random field such that,
	for some constants $M\ge 1$ and $K>0$, and for some function $\varphi: [0\,,\tau] \rightarrow \R_+$,
	\begin{equation} \label{E:h1}
		\| \Phi(s,y) \| _k \le K N^{-d}  \varphi(s)\int_{[0,N]^d} \bm{p}_{M(\tau-s)} (x-y)\, \d x,
	\end{equation}
	for all
	$(s\,,y) \in [0\,,\tau] \times \R^d$. Then,
	\[
		\left\| \int_{[a,b] \times \R^d} \Phi(s\,,y) \,\eta( \d s\, \d y) \right\|_k
		\le 2  \sqrt{ k f(\R^d)}  K N^{-d/2} \left( \int_a^b \varphi^2(s)\, \d s\right)^{1/2}.
	\]
\end{lemma}

\begin{proof}
	By the BDG inequality, the estimate \eqref{E:h1}
	and the semigroup property of the heat kernel together yield
	\begin{align*}
		&\left\| \int_{[a,b] \times \R^d} \Phi(s\,,y)\, \eta( \d s\, \d y) \right\|_k^2 \le 4k
			\int_a ^b \d s   \int_{\R^{2d}}  \d y\, f(\d z) \| \Phi(s\,,y) \| _k \| \Phi(s\,,y-z) \| _k  \\
		&  \qquad \le  4k K^2   N^{-2d} \int_{[0,N]^{2d}} \d x_1 \d x_2 \int_a ^b \d s\
			\varphi^2(s)  \int_{\R^{2d}}  \d y\ f(\d z)\
			\bm{p}_{M(\tau-s)} (x_1-y) \bm{p}_{M(\tau-s)} (x_2-y+z) \\
		& \qquad =4k K^2   N^{-2d} \int_{[0,N]^{2d}} \d x_1 \d x_2 \int_a ^b  \d s\
			\varphi^2(s)   \int_{\R^{d}}  f(\d z)\ \bm{p}_{2M(\tau-s)} (x_1-x_2 -z) \\
		&\qquad \le   4k K^2 N^{-d} f(\R^d) \int_a^b \varphi^2(s)\, \d s.
	\end{align*}
	The proof is complete.
\end{proof}

\begin{lemma} \label{L:4Holder}
	For every $\rho \in [0\,,1)$ there exists a constant $C_\rho$ such that
	\[
		\int_0^1 (1-x) ^{-\rho} \e^{-\lambda x (1-x)}\, \d x  \le C_\rho ( 1\wedge \lambda^{\rho-1} )
		\qquad\text{ for every $\lambda>0$}.
	\]
\end{lemma}

\begin{proof}
	 Without loss of generality, we may -- and will -- assume
	that $\lambda\ge1$. Now simple computations allow us to write
	\begin{align*}
		\int_0^1 (1-x) ^{-\rho} \e^{-\lambda x (1-x)}\, \d x &=
			\int_0^{1/2} (1-x) ^{-\rho} \e^{-\lambda x (1-x)} \,\d x
			+ \int_{1/2}^1 (1-x) ^{-\rho} \e^{-\lambda x (1-x)}\, \d x \\
		 & \le  2^\rho  \int_0^{1/2} \e^{ -\lambda x/2}\, \d x +
		 	\int_{1/2}^1 (1-x) ^{-\rho} \e^{-\lambda  (1-x)/2}\, \d x \\
		 & \le  (2^\rho+1)
		 	\int_0^{ \infty}  x^{-\rho} \e^{ -\lambda x/2}\, \d x
			 < 4\Gamma(1-\rho) \lambda^{\rho-1}.
	\end{align*}
	 This provides the desired estimate.
\end{proof}

According to Sanz-Sol\'e and Sarr\`a  \cite{SS},
the reinforced Dalang's condition \eqref{E:Dalang2}
implies  the following H\"older continuity of the solution:
\begin{equation} \label{E:uHolder}
	\sup_{x\in \R^d} \|  u(t_2\,,x) - u(t_1\,,x) \| _p \le
	C_{T,p,\gamma}   |t_2-t_1|^\gamma,
\end{equation}
 valid for all $t_1,t_2 \in [0\,,T]$, $\gamma \in (0\,, \alpha/2)$, and $p\ge 2$.
The next proposition  provides  moment estimates which are well tailored to establish tightness.

\begin{proposition}  \label{P:Holder}
	Assume that $f$ satisfies both \eqref{E:f_finite} and the reinforced Dalang's condition \eqref{E:Dalang2} for some $\alpha\in (0,1]$.
	Then we have the following two cases: \\
	(i) If $g\in  C^1(\R)$ is such that $g'$ is H\"older continuous of order $\delta \in (0\,,1]$,
	then for every $T>0$,   $\gamma  \in (0\,,\alpha/2)$,  and $k\ge 2\vee \delta^{-1}$,
	there exists a number $L>0$ depending on $T, k,\sigma, f,  \gamma$ and $g$ such that
	\[
		\sup_{0\le t_1, t_2 \le T}
		\E \left( | \mathcal{S}_{N,t_2}(g) -\cS_{N,t_1}(g) |^k \right)
		\le  L| t_2-t_1| ^{k \gamma\delta} N^{-kd/2}.
	\]
	(ii) If $\sigma(0)=0$ and if $g\in C^2 (0\,,\infty)$ is such that $g''$ is monotone over $(0\,,\infty)$ and
	condition \eqref{E:MmCond_FCLT} is satisfied, then
	for every $T>0$, $\gamma \in (0\,,\alpha/2)$ and $k\ge 2$
	there exists a number $L>0$ depending on $T, k,\sigma, f,
	\gamma$ and $g$ such that
	\[
		\sup_{0\le t_1, t_2 \le T} \E \left( |
		\mathcal{S}_{N,t_2}(g) -\cS_{N,t_1}(g) |^k \right) \le  L| t_2-t_1| ^{k  \gamma} N^{-kd/2}.
	\]
\end{proposition}

\begin{proof}[Proof of Case (i) of Proposition \ref{P:Holder}]
	Fix $\rho  \in (0\,,1/2)$. For all  $N>0$ and $ 0\le  t_1 \le t_2  \le T $
	we can write using Clark-Ocone formula \eqref{E:Clark-Ocone} and stochastic Fubini's theorem
	\begin{align*}
		\cS_{N,t_2} (g)- \cS_{N,t_1} (g)
			&=  N^{-d} \int _{[0,N]^d} [ g(u(t_2\,,x)) - \E(g(u(t_2\,,x)))
			- g(u(t_1\,,x))+\E(g(u(t_1\,,x)))  ] \,\d x \\
		&= \int_{[0,t_1] \times \R^d} \mathcal{A}(s\,,y)\, \eta (\d s \, \d y)
			+ \int_{[t_1,t_2] \times \R^d} \mathcal{B}(s\,,y)\, \eta (\d s \, \d y),
	\end{align*}
	where
	\begin{align*}
		\mathcal{A}(s\,,y) & := N^{-d}  \int_{[0,N]^d}
			\E \left[  D_{s,y}  (g(u(t_2,x)) -g(u(t_1,x)) ) \mid \mathcal{F}_s \right] \d x
			&  [0\le s\le t_1],\\
		\mathcal{B}(s\,,y) & := N^{-d}  \int_{[0,N]^d}
			\E \left[  D_{s,y}  g(u(t_2,x))  \mid \mathcal{F}_s \right] \d x
			& [t_1\le s\le t_2].
	\end{align*}
	We can make the additional decomposition,
	$\mathcal{A}(s\,,y)=: \mathcal{A}_1(s\,,y) + \mathcal{A}_2(s\,,y)$, where
	\begin{align*}
		\mathcal{A}_1(s\,,y) & := N^{-d}  \int_{[0,N]^d}
			\E \left[  (g' (u(t_2,x)) -g' (u(t_1,x)))  D_{s,y} u(t_2,x) \mid \mathcal{F}_s \right] \d x, \\
	 	\mathcal{A}_2(s\,,y) & := N^{-d}  \int_{[0,N]^d}
			\E \left[   g' (u(t_2,x)) [ D_{s,y} u(t_1,x)-D_{s,y} u(t_2,x)] \mid \mathcal{F}_s \right] \d x.
	\end{align*}
	We first estimate the $k$-norm of the random variables
	$\mathcal{A}_1(s\,,y)$,  $\mathcal{A}_2(s\,,y)$, and   $\mathcal{B}(s\,,y)$
	in order to apply Lemma \ref{L:h1}.
	The H\"older continuity of $g'$,  the estimate  \eqref{E:D<p2}, and the H\"older continuity of
	$u(t\,,x)$  (see \eqref{E:uHolder})  together imply
	\begin{align}  \notag
		\| \mathcal{A}_1(s\,,y) \|_k & \le   L_\delta (g')  \sup_{x\in \R^d}
			\| u(t_2,x)- u(t_1,x) \|_{2 k\delta} ^{\delta}  N^{-d}  \int_{[0,N]^d}
			\|  D_{s,y} u(t_2,x)  \| _{2k}\,\d x \\
	& \le   \label{E:A1}
			C_{T, 2k, \ep, \sigma} L_\delta(g') C^\delta_{T, 2k\delta,\gamma}
			|t_2-t_1 |^{\gamma \delta}  N^{-d}  \int_{[0,N]^d}  \bm{p}_{t_2-s} (x-y)\,\d x,
	\end{align}
	where $  L_\delta (g')  = \sup_{-\infty <a<b<\infty} | g'(b)- g'(a)|/|b-a|^\delta$.

	The estimation of the  $k$-norm of   $\mathcal{A}_2(s\,,y)$ is more involved.
	We first write for every $s\in [0\,,t_1]$ and $y\in \R^d$,
	\begin{align*}
		D_{s,y} u(t_1,x)-D_{s,y} u(t_2,x) &= \sigma( u(s\,,y))
			[\bm{p}_{t_1-s} (x-y) - \bm{p}_{t_2-s} (x-y)] \\
		& \quad + \int_s^{t_1}   \int_{\R^d} [ \bm{p}_{t_1-r} (x-z )- \bm{p}_{t_2-r} (x-z )]
			\sigma'(u(r\,,z)) D_{s,y} u(r\,,z)\, \eta (\d  r\, \d z)  \\
		& \quad  + \int_{t_1}^{t_2}   \int_{\R^d} \bm{p}_{t_2-r} (x-z ) \sigma'(u(r\,,z))
			D_{s,y} u(r\,,z) \,\eta (\d  r\, \d z),
	\end{align*}
	where $\sigma'$  denotes the weak derivative of $\sigma$; see Remark \ref{R:Lip}.
	As a consequence,
	\[
		\Norm{ \mathcal{A}_2(s\,,y) }_k  \le   N^{-d} \int_{[0,N]^d} \| g'(u(t_2,x))\|_{2k} [ \Phi_1(x)+ \Phi_2(x)+ \Phi_3(x)] \d x,
	\]
	where
	\begin{align*}
		\Phi_1 & := \Norm{ \sigma( u(s\,,y))} _{2k}  |\bm{p}_{t_1-s} (x-y) - \bm{p}_{t_2-s} (x-y)|, \\
		\Phi_2 & := \left\|  \int_{(s,t_1)\times\R^d} [ \bm{p}_{t_1-r} (x-z )- \bm{p}_{t_2-r} (x-z )]
			\sigma'(u(r\,,z)) D_{s,y} u(r\,,z)\, \eta (\d  r\, \d z) \right\|_{2k}, \\
		\Phi_3 & := \left\|  \int_{(t_1,t_2)\times\R^d}   \bm{p}_{t_2-r} (x-z )
			\sigma'(u(r\,,z)) D_{s,y} u(r\,,z) \,\eta (\d  r\, \d z) \right\|_{2k} .
	\end{align*}
	 First we estimate $\Phi_1 $.
	Let $C_{T,2k,\sigma}=\sup_{t\in [0,T]}
	\Norm{\sigma( u(s\,,0))} _{2k} $ and use the following estimate  from Chen and Huang \cite[Lemma 4.1]{CH19}:
	\begin{equation}   \label{E:time}
 		|\bm{p}_{t_1-s} (x-y) - \bm{p}_{t_2-s} (x-y)|
		\le C \left(\frac{t_2-t_1}{t_1-s}\right)^{\rho/2}
		\bm{p}_{4(t_2-s)} (x-y),
	\end{equation}
	 valid with $\rho=2\gamma \delta <\alpha$, where $\alpha\in (0\,,1]$ is the  exponent
	that appears in the reinforced Dalang's condition \eqref{E:Dalang2}. Thus, we obtain
	\begin{equation} \label{E:Phi1}
		\Phi_1 \le  C_{T,2k,\sigma}   C %
		\left(\frac{t_2-t_1}{t_1-s}\right)^{\gamma\delta}
		\bm{p}_{4(t_2-s)} (x-y).
	\end{equation}
	For $\Phi_2$,  we apply the BDG inequality,  \eqref{E:time},
	and \eqref{E:D<p2}, and write
	\begin{align*}
		\Phi^2_2 & \le 8k  [{\rm Lip}(\sigma) C_{T,2k, \ep, \sigma}]^2
			\int_s^{t_1}   \d r  \int_{\R^d} \d z\int_{\R^d} f(\d z')\
			| \bm{p}_{t_1-r} (x-z )- \bm{p}_{t_2-r} (x-z ) |\\
		& \quad \times |  \bm{p}_{t_1-r} (x-z+z' )- \bm{p}_{t_2-r} (x-z+ z' ) |
			\bm{p}_{r-s} (y-z) \bm{p}_{r-s} (y-z+z') \\
		& \le 8k  [{\rm Lip}(\sigma) C C_{T,2k, \ep, \sigma}]^2 (t_2-t_1)^{\rho}
			\int_s^{t_1}   \frac{\d r }{(t_1-r)^\rho} \\
		& \quad \times  \int_{\R^d}\d z \int_{\R^d}f(\d z')\
			\bm{p}_{4(t_2-r)} (x-z)  \bm{p}_{4(t_2-r)} (x-z+z')
			\bm{p}_{r-s} (y-z) \bm{p}_{r-s} (y-z+z')\\
		& \le 2^{2d+2}k    [{\rm Lip}(\sigma) C C_{T,2k, \ep, \sigma}]^2(t_2-t_1)^{\rho}
			\int_s^{t_1}   \frac{\d r }{(t_1-r)^\rho} \\
		& \quad \times  \int_{\R^d}\d z \int_{\R^d}f(\d z')\
			\bm{p}_{4(t_2-r)} (x-z)  \bm{p}_{4(t_2-r)} (x-z+z')   \bm{p}_{4(r-s)} (y-z) \bm{p}_{4(r-s)} (y-z+z').
	\end{align*}
	The elementary identity,
	\begin{equation} \label{E:equa1}
		\bm{p} _\sigma(x)    \bm{p} _\sigma(y)  =2^d \bm{p}_{2\sigma} (x-y) \bm{p}_{2\sigma} (x+y),
	\end{equation}
	valid for every $\sigma>0$ and $x,y\in\R^d$, yields the
	following with $C^{(1)} := 2^{2d+2}k  [{\rm Lip}(\sigma)CC_{T,2k, \ep, \sigma}]^2  $:
	\begin{align}     \notag
		\Phi^2_2  & \le    C^{(1)}   (t_2-t_1)^{\rho}
			\int_s^{t_1} \frac{\d r}{(t_1-r)^\rho}  \\ \notag
		& \quad \times \int_{\R^d} \d z \int_{\R^d} f(\d z')\
			\bm{p}_{8(t_2-r)} (2x-2z+z')  \bm{p}_{8(t_2-r)} (z')
			\bm{p}_{8(r-s)} (2y-2z+z') \bm{p}_{8(r-s)} (z') \\ \notag
		&=    2^{-d}C^{(1)}  (t_2-t_1)^{\rho}     \bm{p}_{8(t_2 -s)} (2(x-y))
			\int_s^{t_1} \frac{\d r}{(t_1-r)^\rho}
			\int_{\R^{d}}  f(\d z')  \
			\bm{p}_{8(t_2-r)} (z')   \bm{p}_{8(r-s)} (z')\\
		&=    2^{-2d}C^{(1)}  (t_2-t_1)^{\rho}     \bm{p}_{4(t_2 -s)} ^2(x-y)
			\int_s^{t_1}   \frac{\d r}{(t_1-r)^\rho}      \label{E:Phi2a}
			\int_{\R^{d}}  f(\d z') \
			\bm{p}_{8(t_2-r)(r-s)/(t_2-s)} (z').
	\end{align}
	In the last equality we used  the formulas
	\begin{equation} \label{E:equa2}
		\bm{p} _\sigma(x)     \bm{p} _\tau (x)   = (2\pi)^{-d/2}
		(\sigma+\tau)^{-d/2} \bm{p}_{\sigma \tau/(\sigma+\tau)}(x)
		\qquad [\sigma, \tau >0, x\in \R^d],
	\end{equation}
	and
	\begin{equation} \label{E:equa3}
		\bm{p} _\sigma(2x)   = 2^{-d} (2\pi \sigma)^{d/2}
		\bm{p}^2 _{\sigma/2}(x)  \qquad  [\sigma>0, x\in \R^d].
	\end{equation}
	Plancherel's identity allows us to write
	\begin{equation} \label{E:EQ2}
		\int_{\R^d}  p_{8(t_2-r)(r-s)/(t_2-s)} (z) \,f(\d z)
		=\left(2\pi\right)^{-d} \int_{\R^d} \exp\left(-\frac{4(t_2-r)(r-s)}{t_2-s}\Norm{\xi}^2\right) \hat{f}(\xi)\d\xi.
	\end{equation}
	Moreover,
	\begin{align}   \notag
		\int_s^{t_1}
			\exp\left(-\frac{4(t_2-r)(r-s)}{t_2-s}\Norm{\xi}^2\right) \frac{\d r}{(t_1-r)^\rho}
			& \le \int_s^{t_1} \exp\left(-\frac{4(t_1-r)(r-s)}{t_1-s}\Norm{\xi}^2\right)
			 \frac{\d r } {(t_1-r)^\rho}\\ \notag
		&\hskip-.3in= (t_1-s)^{1-\rho}  \int_0^1
			\exp\left(- 4 x(1-x) (t_1-s)\Norm{\xi}^2\right)
			 \frac{\d x}{(1-x)^\rho} \\ \notag
		&\hskip-.3in\le   C_\rho (t_1-s)^{1-\rho}
			\left(1\wedge  (4(t_1-s) \| \xi \|^2)^{\rho-1} \right) \\  \label{E:EQ3}
		&\hskip-.3in\le C_\rho  4^{\rho-1}   ( (4T)^{1-\rho} \wedge  \| \xi \|^{2(\rho-1)} ),
	\end{align}
	where we made the change of variable $(r-s)/(t_1-s) =x$ in the equality and
	applied  Lemma \ref{L:4Holder} in the second inequality.
	We substitute  \eqref{E:EQ2}  and \eqref{E:EQ3} into   \eqref{E:Phi2a},
	taking into account that
	$\int_{\R^d}  ( (4T)^{1-\alpha} \wedge  \| \xi \|^{2(\rho-1)} ) \hat{f} (\xi)\,\d \xi <\infty$
	due to \eqref{E:Dalang2} and $\rho\le \alpha$,    in order to find that
	\begin{equation} \label{E:Phi2}
		\Phi_2 \le  L  (t_2-t_1)^{\rho /2}      \bm{p}_{4(t_2-s)} (x-y),
	\end{equation}
	where here and for the rest of the proof $L$ will denote a constant that   depends on
	$(T,k, \sigma, \gamma, \rho,f)$, and can vary from line to line.
	We use the same arguments as before and
	employ the formulas \eqref{E:equa1}, \eqref{E:equa2} and \eqref{E:equa3},
	in order to write
	\begin{align}
		\Phi^2_3  &    \notag
			\le 8k  [{\rm Lip}(\sigma)  C_{T,2k, \ep, \sigma}]^2  \\ \notag
		& \quad \times  \int_{t_1}^{t_2}   \d r
			 \int_{\R^d}\d z  \int_{\R^d}
			f(\d z')\ \bm{p}_{t_2-r} (x-z )  \bm{p}_{t_2-r} (x-z+ z' )
			\bm{p}_{r-s} (y-z) \bm{p}_{r-s} (y-z+z') \\ \notag
		& \le   2^{2d+3} k [{\rm Lip}(\sigma)  C_{T,2k, \ep, \sigma}]^2 \\ \notag
		& \quad \times   \int_{t_1}^{t_2} \d r
			 \int_{\R^d}\d z  \int_{\R^d}
			f(\d z')\
			\bm{p}_{2(t_2-r)} (2x-2z+z')  \bm{p}_{2(t_2-r)} (z')
			\bm{p}_{2(r-s)} (2y-2z+z') \bm{p}_{2(r-s)} (z')\\ \label{E:Holder7}
		& \le  C^{(2)} \bm{p}^2_{  t_2-s} (x-y)\int_{t_1}^{t_2} \d r
			 \int_{\R^d}\d z  \int_{\R^d}
			f(\d z')\
			\bm{p}_{2(t_2-r)(r-s)/(t_2-s)} (z'),
	\end{align}
	with $C^{(2)} := 8 k [{\rm Lip}(\sigma)  C_{T,2k, \ep, \sigma}]^2$.
	Again,  Plancherel's identity allows us to write
	\begin{equation} \label{E:Holder5}
		\int_{\R^d} p_{2(t_2-r)(r-s)/(t_2-s)} (z)\, f(\d z)
			=\left(2\pi\right)^{-d} \int_{\R^d}
			\exp\left(-\frac{(t_2-r)(r-s)}{t_2-s}\Norm{\xi}^2\right) \hat{f}(\xi)\,\d\xi,
	\end{equation}
	and we have
	\begin{align}  \notag
		&\int_{t_1}^{t_2} \exp\left(-\frac{(t_2-r)(r-s)}{t_2-s}\Norm{\xi}^2\right)
		 \d r
			\le \int_{t_1}^{t_2}
			\exp\left(-\frac{(t_2-r)(r-t_1)}{t_2-t_1}\Norm{\xi}^2\right)
			 \d r\\ \notag
		& \qquad = (t_2-t_1) \int_0^1  \exp \left(  -x(1-x) (t_2-t_1) \| \xi \|^2 \right)
			 \d x \\ \notag
		& \qquad    \le  C_0 (t_2-t_1)  (1\wedge ( (t_2-t_1) \| \xi \|^2)^{-1})
			\le  C_0 (t_2-t_1)  (1\wedge ( (t_2-t_1) \| \xi \|^2)^{\alpha-1}) \\ \label{E:Holder6}
		& \qquad    \le  C_0 (t_2-t_1)^\alpha  (T^{1-\alpha}\wedge  \| \xi \|^{ 2\alpha-1}),
	\end{align}
	where we used Lemma \ref{L:4Holder}
	with $\rho=0$ in the equality. From \eqref{E:Holder7}, \eqref{E:Holder5}, and \eqref{E:Holder6} we obtain
	\begin{equation} \label{E:Phi3}
		\Phi_3 \le L (t_2-t_1)^{\alpha/2}   \bm{p}_{  t_2-s} (x-y).
	\end{equation}
	From \eqref{E:Phi1}, \eqref{E:Phi2} and \eqref{E:Phi3} we deduce
	\begin{equation} \label{E:A2}
		\|  \mathcal{A}_2(s\,,y) \|_k \le  \Norm{g'(u(t_2\,,0))}_{2k}
		     L(t_2-t_1)^{\gamma\delta  }  (t_1-s)^{-\gamma\delta} N^{-d}
		\int_{[0,N]^d} \bm{p}_{  4(t_2-s)} (x-y)\,\d x.
	\end{equation}
	The H\"older continuity of $g'$ implies that   $ \sup_{r\in [0,T]} \Norm{g'(u(r\,,0))}_{2k}   <\infty$
	and  an application of Lemma \ref{L:h1}
	with $[a\,,b] = [0\,,t_1]$, $\tau=t_2$,  $M=4$,
	$\varphi(s)= (t_1-s)^{-\gamma\delta} \mathbf{1}_{[0,t_1)}(s)$
	yields the desired bound for
	$ \|\int_{[0,t_1] \times \R^d} \mathcal{A}(s\,,y)\, \eta (\d s \, \d y)\|_k$.

	Finally,
	\begin{equation} \label{E:A3}
		\| \mathcal{B}(s\,,y)\|_k  \le    \sup_{r\in[0,T]}\Norm{g'(u(r\,,0))}_{2k}   C_{T,2k, \sigma, \ep}
		N^{-d}   \int_{[0,N]^d} \bm{p}_{t_2-s} (x-y)\,\d x,
	\end{equation}
	and a further application of  Lemma \ref{L:h1} with $[a\,,b] = [t_1,t_2]$,
	$\tau=t_2$,  $M=1$, and  $\varphi(s)= 1$
	allows us to bound $ \|\int_{[t_1,t_2] \times \R^d} \mathcal{B}(s\,,y)\, \eta (\d s \, \d y)\|_k$
	by a constant times $N^{-d/2} (t_2-t_1) ^{1/2}$.
	The proof of Case (i) of Proposition \ref{P:Holder} is now complete.
\end{proof}

\begin{proof}[Proof of Case (ii) of Proposition \ref{P:Holder}]
	We follow the proof of Case (i) of Proposition \ref{P:Holder},
	making necessary changes to adapt the arguments to the present setting.
	The main difference comes from the bound for $\mathcal{A}_1$ under the new
	assumptions on $g$. In particular, notice that
	\begin{align*}
		\Norm{g'\left(u(t_1,x)\right)-g'\left(u(t_2,x)\right)}_{2k}
			&\le \left(\Norm{g''\left(u(t_1,x)\right)}_{4k}+
			\Norm{g''\left(u(t_2,x)\right)}_{4k}\right)
			\sup_{x \in \R^d}\Norm{u(t_1,x)-u(t_2,x)}_{4k}\\
		&\le C_{T,4k,\gamma}\left(\Norm{g''\left(u(t_1,0)\right)}_{4k}+
			\Norm{g''\left(u(t_2,0)\right)}_{4k}\right) | t_2-t_1|^{\gamma}\\
		&\le C_{T,4k,\gamma}\sup_{t\in[0,T]}\Norm{g''\left(u(t\,,0)\right)}_{4k} | t_2-t_1|^{\gamma}\\
		&=: C_1  | t_2-t_1|^\gamma,
	\end{align*}
	where the first inequality is due to the monotonicity of $g''$, the second  is due to \eqref{E:uHolder}
	and the stationarity of the solution, and the constant
	$C_{T,4k,\gamma}$ comes from \eqref{E:uHolder}.
	Note that the continuity of $t\in [0\,,T] \mapsto \Norm{g''(u(t\,,0))}_{4k}$
	ensures that the above constant $C_1$ is finite.
	Hence,
	\begin{gather} \label{E:A1'}
		\Norm{\mathcal{A}_1(s\,,y)}_k \le C_1|t_1-t_2|^{\gamma} N^{-d}
		\int_{[0,N]^d}\bm{p}_{t_2-s}(x-y)\,\d x.
	\end{gather}
	The  remainder
	of the proof is similar, with the remark that  $ \sup_{r\in [0,T]} \Norm{g'(u(r\,,0))}_{2k}   <\infty$ due to condition
	\eqref{E:MmCond_FCLT}.
\end{proof}

Equipped with Proposition  \ref{P:Holder}, we can now proceed with the proof of Theorem \ref{T:FCLT}.

\begin{proof}[Proof of Theorem \ref{T:FCLT}]
	We first prove Case (i).
	Choose and fix some $T>0$.
	By   Proposition \ref{P:Holder},
	a standard application of Kolmogorov's continuity theorem and
	the Arzel\`{a}-Ascoli theorem  ensures that
	$\{  N^{d/2}\,\mathcal{S}_{N,\bullet}(g) \}_{N\ge 1}$ is a tight net of processes
	on $C([0\,,T])$. Therefore, it remains to prove that the finite-dimensional
	distributions of the process $t\mapsto N^{d/2}\mathcal{S}_{N,t}(g)$
	converge to those of $\{\mathcal{G}_t\}_{t\in[0,T]}$.

	Let us choose and fix  some $T>0$ and $m\ge 1$ points $t_1,\ldots,t_m\in(0\,,T)$.
	 Consider the  random vector $F_N= (F_N^{(1)}, \dots, F_N^{(m)})$ defined by
	\[
		F_N^{(i)}  := N^{d/2}  \mathcal{S}_{N,t_i} (g),
		\qquad \text{for }i=1,\ldots,m,
	\]
	  let $G=(\mathcal{G}_{t_1} \,,\ldots,\mathcal{G}_{t_m})$ denote a centered Gaussian random vector
	with covariance matrix $(\mathbf{B}_{t_i,t_j}(g))_{1\le i,j\le m}$.
	Recall from \eqref{E:HatV} the random fields $v_{N,t_1}(g),\ldots,v_{N,t_m}(g)$ which are such that
	$F_N^{(i)} =\delta(N ^{d/2} v_{N,t_i}(g) )$
	for all $i=1,\ldots,m$.  Set $V_N^{(i)} =N ^{d/2} v_{N,t_i}(g)$ for $i=1,\dots,m$ and
	let $V_N=(V_N^{(1)}, \dots, V_N^{(m)})$.
	Since $\E[F_N^{(i)}]=0$ and $F_N^{(i)}\in\mathbb{D}^{1,4}$ (see \eqref{E:SNallN}), we can write,
	by a  generalization of  a result of Nourdin and Peccatti \cite[Theorem 6.1.2]{NP},
	\begin{equation} \label{E:equa7}
		| \E( h(F_N)) -\E (h(G)) | \le \tfrac 12 \|h ''\|_\infty
		\sqrt{   \sum_{i,j=1}^m   \E \left( \left|
		 \mathbf{B}_{i,j} (g)- \langle DF_N^{(i)} \,, V_N^{(j)} \rangle_{\HH} \right|^2
		\right)}
	\end{equation}
	for every $h\in C^2(\R^m)$ that has bounded second partial derivatives, and with the notation
	 \[
		\|h'' \| _\infty= \adjustlimits\max_{1\le i,j \le m}
		\sup_{x\in\R^m}  \left| \frac { \partial ^2h (x) } {\partial x_i \partial x_j} \right|.
	\]
	Thus,  in view of \eqref{E:equa7}  in order to show the  convergence in law of $F_N$ to $G$, it suffices to show that for any $i,j =1,\dots, m$, we have
	\begin{equation} \label{E:h6}
		\lim_{N\rightarrow \infty} \E \left( \left|
		 \mathbf{B}_{i,j} (g)- \langle DF_N^{(i)} \,, V_N^{(j)} \rangle_{\HH}\right|^2
		\right)=0.
	\end{equation}
	Notice that, by the duality relationship  \eqref{D:delta} and the convergence \eqref{E:Var_Lim} and   we have, as $N \rightarrow \infty$,
	\begin{align} \notag
		\E \left( \langle DF_N^{(i)} \,,V_N^{(j)}  \rangle_{\HH} \right)
		& = N^{d}\E\left(\langle  D \mathcal{S}_{N,t_i} (g) , v_{N,t_j}(g) \rangle_{\HH}  \right) \\\notag
		& = N^{d}\E\left(\mathcal{S}_{N,t_i} (g) , \delta (v_{N,t_j}(g) ) \right)\\
		& = N^{d}\E\left(\mathcal{S}_{N,t_i} (g) ,  \mathcal{S}_{N,t_j} (g)  \right) \rightarrow  \mathbf{B}_{t_i,t_j}(g). \label{E:h7}
	\end{align}
	Therefore, the convergence \eqref{E:h6} follows immediately from  \eqref{E:h7} and  \eqref{E:L1Conv}.
		It follows from this fact that the finite-dimensional distributions of
	$t\mapsto N^{d/2} \,\mathcal{S}_{N,t}(g)$ converge to those of
	$\mathcal{G}$ as $N\to\infty$. The proof of Case (i) of Theorem \ref{T:FCLT} is now complete.
	\bigskip

	The proof of Case (ii) follows the same lines as that of Case (i) with the following changes:
	Tightness  is a consequence of part (ii) of Proposition \ref{P:Holder}.
	In order to apply Theorem 6.1.2 of Nourdin and Peccati \cite{NP},
 	we need to verify that $\mathcal{S}_{n,t}(g)\in \mathbb{D}^{1,4}$ for all $t>0$.
	This follows from  the fact that  $g(u(t\,,x)) \in  \mathbb{D}^{1,4}$ for all   $t>0$
	and for all  $x\in \R^d$ due to Lemma \ref{L:Sobolev} and assumption \eqref{E:MmCond_FCLT}
	and the fact that the integral
	$\int_{[0,N]^d} g(u(t\,,x) )\,\d x$ can be approximated by Riemann sums in the norm
	of $\mathbb{D}^{1,4}$.
	Finally, the convergence in \eqref{E:h7} is due to \eqref{E2:Var_Lim}.
	This completes the proof of Theorem \ref{T:FCLT}.
\end{proof}

\section{Total Variation Distance for linear \texorpdfstring{$g$}{g}: Proof of Theorem \ref{T:TV_Bnd}} \label{S:Bound}

Throughout this section, we choose and fix a number $t>0$.
The strategy of the proof of Theorem \ref{T:TV_Bnd} is as follows:
We proceed as we did with Theorem \ref{T:TV}, and begin by appealing to
\eqref{E:goal}, \eqref{E:HatV}, and \eqref{E:K}, but use the linearity of $g$
in order to make  precise estimates of the centralized moments of $K_N$
(rather than appeal to  ergodic-theoretic arguments).
The details are lengthy and technical, and will follow in due time. For now, we suffice it to say that
one can indeed anticipate simplifications in the present setting
in which $g$ is linear. For example,
because $g(v)=v$ for all $v\in\R^d$, \eqref{E:ec2} simplifies to the a.e.\
identity (in $(s\,,z)\in(0\,,t)\times\R^d$):
\begin{equation}\label{E:g(z)=z}
	D_{s,z}\mathcal{S}_{N,t}(g) = N^{-d}\int_{[0,N]^d} D_{s,z}u(t\,,x)\,\d x.
\end{equation}
This is one of the quantities that appears in the $\mathcal{H}$-inner product
in the definition \eqref{E:KN} of $K_N$. We will be able to make use of the above
because of the well-known fact that
the integrand $D_{s,z}u(t\,,x)$ solves the following stochastic
integral equation for all $t>0$ and $x\in\R^d$ and a.e.\
$(s\,,y)\in(0\,,t)\times\R^d$: a.s.,
\begin{equation}\label{E:Du}
	D_{s,y} u(t\,,x) = \bm{p}_{t-s}(x-y) \sigma(u(s\,,y)) +
	\int_{(s,t)\times\R^d} \bm{p}_{t-s}(x-z) \sigma'(u(r\,,z))
	D_{s,y} u(r\,,z)\,\eta(\d r\,\d z).
\end{equation}
The above integral eqution also allows us to better understand the second quantity
$v_{N,t}(g)$ in the $\mathcal{H}$-inner product of the definition \eqref{E:KN} of $K_N$.
Indeed,
because of the martingale property of Walsh integral,
the above eadily yields the following a.s.-identity that is valid with the same null-set
restrictions as does the preceding identity:
\[
	\E\left[  D_{s,y}u(t\,,x)\mid\mathcal{F}_s \right]
	= \bm{p}_{t-s}(x-y) \sigma(u(s\,,y)).
\]
Integrate this over $x\in[0\,,N]^d$, apply Fubini's theorem, and divide by $N^d$
to find that for all $t>0$ and a.e.\ $(s\,,y)\in(0\,,t)\times\R^d$,
\begin{equation}\label{E:Pi}\begin{split}
	&\E\left[ D_{s,y}\mathcal{S}_{N,t}(g) \mid \mathcal{F}_s\right]
		=\Pi^{(N)}_{s,y}\sigma(u(s\,,y))
		\quad\text{a.s.,\ where}\\
	&\Pi^{(N)}_{s,y}=\Pi^{(N)}_{s,y}(t) := N^{-d}\int_{[0,N]^d}\bm{p}_{t-s}(x-y)\,\d x.
\end{split}\end{equation}
In light of \eqref{E:HatV}, we find the following formula for the
random field $v_{N,t}(g)$, valid since $g(v)=v$ for all
$v\in\R$: For all $t>0$ and a.e.\ $(s\,,z)\in(0\,,t)\times\R^d$,
\[
	v_{N,t}(g) (s\,,z) = \Pi^{(N)}_{s,z}\sigma(u(s\,,z))\bm{1}_{[0,t]}(s).
\]
Finally, let us integrate both sides of \eqref{E:Du} over $x\in[0\,,N]^d$,
and then divide by $N^d$, in order to deduce
from standard Fubini-type arguments and \eqref{E:g(z)=z} that a.s.\
for a.e.\ $(s\,,y)\in(0\,,t)\times\R^d$,
\begin{equation}\label{E:DvT}\begin{split}
	&D_{s,y}\mathcal{S}_{N,t}(g) = v_{N,t}(g)(s\,,y) + \mathcal{T}(s\,,y),
		\quad\text{where}\\
	&\mathcal{T}(s\,,y) = \mathcal{T}_{g,t,N}(s\,,y) := \int_{(s,t)\times\R^d}
		\Pi^{(N)}_{s,z} \sigma'(u(r\,,z))
		D_{s,y} u(r\,,z)\,\eta(\d r\,\d z),
\end{split}\end{equation}
and $\sigma'$ is any [measurable]
version of the weak derivative of $\sigma$ (see Remark \ref{R:Lip}).
And $D_{s,y}\mathcal{S}_{N,t}=0$ when $s\ge t$.
The preceding yields, among other things, that
\begin{equation}\label{E:K:g(z)=z}
	\frac{K_N}{N^d} = \left\| v_{N,t}(g)\right\|_{\mathcal{H}}^2 +
	\left\< \mathcal{T}\,,v_{N,t}(g)\right\>_{\mathcal{H}};
\end{equation}
see \eqref{E:KN}.
We estimate the preceding two quantities separately and in turn by way of two lemmas.

\begin{lemma}\label{L:||v||:g(z)=z}
	There exists a real number $L_1=L_1(f(\R^d)\,,\lip(\sigma))$ such that for every $\varepsilon\in(0\,,1)$,
	$N,t>0$,
	\[
		\Var\left( \left\|v_{N,t}(g)\right\|_{\mathcal{H}}^2\right) \le
		 L_1t^3C_{t,4,\varepsilon,\sigma}^2 A_t^2 N^{-3d},
	\]
	where $C_{t,4,\varepsilon,\sigma}$ was defined in Lemma \ref{L:Du}
	and $A_t := \sup_{s\in(0,t)}\|\sigma(u(s\,,0))\|_4$.
\end{lemma}

\begin{proof}
	Choose and fix $N,t>0$ and $\varepsilon\in(0\,,1)$ throughout.
	The definition of the Hilbert space $\mathcal{H}$ implies that
	\begin{align*}
		\left\| v_{N,t}(g) \right\|_{\mathcal{H}}^2 &= \int_0^t\d s
			\int_{\R^d}\d y\ v_{N,t}(g)(s\,,y) \left( v_{N,t}(g)(s\,, \bullet )*f\right)(y)\\
		&= \int_0^t\d s \int_{\R^d}\d y\int_{\R^d}f(\d z)\
			\Pi^{(N)}_{s,y}\Pi^{(N)}_{s,y-z}
			\sigma(u(s\,,y))\sigma(u(s\,,y-z)).
	\end{align*}
	Therefore, we can apply Minkowski's inequality for integrals\footnote{That is,
		$\| \int_0^t [X_s-\E(X_s)]\,\d s\|_2 \le
			\int_0^t\|X_s - \E(X_s)\|_2\,\d s,$
		whenever this makes sense, for any stochastic process $X=\{X_s\}_{0\le s\le t}$.
		}
	in order to see that
	\begin{align}\nonumber
		\Var\left(\left\| v_{N,t}(g)\right\|_{\mathcal{H}}^2\right)
			&\le \left[\int_0^t\left\{
			\Var\left(\int_{\R^d}\d y\int_{\R^d}f(\d z)\
			\Pi^{(N)}_{s,y}\Pi^{(N)}_{s,y-z}\sigma(u(s\,,y))\sigma(u(s\,,y-z))\right)
			\right\}^{1/2}\,\d s\right]^2\\
		&=: \left[\int_0^t\left\{
			\Var(F_s) \right\}^{1/2}\,\d s\right]^2,
	\label{E:Var(|v|)}\end{align}
	the definition of the process $\{F_s\}_{s\in(0,t)}$ being obvious.
	By the Poincar\'e inequality \eqref{E:Poincare},
	$\Var(F_s)\le\E(\|DF_s\|_{\mathcal{H}}^2)$. Therefore, we need to compute
	$DF_s$. But that is, by the chain rule \cite[Proposition 1.2.3]{Nualart}, given by
	\begin{align*}
		D_{r,w}F_s &= \int_{\R^d}\d y\int_{\R^d}f(\d z)\
			\Pi^{(N)}_{s,y}\Pi^{(N)}_{s,y-z} D_{r,w}\left[\sigma(u(s\,,y))\sigma(u(s\,,y-z))\right]\\
		& = 2\int_{\R^d}\d y\int_{\R^d}f(\d z)\
			\Pi^{(N)}_{s,y}\Pi^{(N)}_{s,y-z} \sigma'(u(s\,,y)) D_{r,w}u(s\,,y)\sigma(u(s\,,y-z)),
	\end{align*}
	where the prefactor $2$ comes from symmetry. In this way we find that
	\begin{align*}
		\E\left( \|DF_s\|_{\mathcal{H}}^2\right) &= \E\left[\int_0^s\d r\int_{\R^d}\d w\int_{\R^d}f(\d v)\
			D_{r,w}(F)D_{r,w-v}(F)\right]\\
		&=4\int_0^s\d r\int_{\R^d}\d w\int_{\R^d}f(\d v)\int_{\R^d}\d y\int_{\R^d}f(\d z)
			\int_{\R^d}\d y'\int_{\R^d}f(\d z')\
			\Pi^{(N)}_{s,y}\Pi^{(N)}_{s,y-z}\Pi^{(N)}_{s,y'}\Pi^{(N)}_{s,y'-z'}\\
		&\quad\times\E\left[\sigma'(u(s\,,y)) D_{r,w}u(s\,,y)\sigma(u(s\,,y-z))
			\sigma'(u(s\,,y')) D_{r,w-v}u(s\,,y')\sigma(u(s\,,y'-z'))\right].
	\end{align*}
	H\"older's inequality ensures that
	the final expectation in the above display is bounded above in modulus by
	\begin{align*}
		&[\lip(\sigma)]^2 \left\| D_{r,w}u(s\,,y)\right\|_4 \left\| \sigma(u(s\,,y-z))\right\|_4
			\left\| D_{r,w-v}u(s\,,y')\right\|_4 \left\| \sigma(u(s\,,y'-z'))\right\|_4\\
		&=[\lip(\sigma)]^2\left\| \sigma(u(s\,,0))\right\|_4^2
			\left\| D_{r,w}u(s\,,y)\right\|_4 \left\| D_{r,w-v}u(s\,,y')\right\|_4
			&\text{[stationarity]}\\
		&\le \left[C_{t,4,\varepsilon,\sigma}A_t\lip(\sigma)\right]^2
			\bm{p}_{s-r}(y-w)\bm{p}_{s-r}(y'-w+v),
	\end{align*}
	uniformly for all $s\in(0\,,t)$, where $\varepsilon\in(0\,,1)$ is arbitrary;
	see Lemma \ref{L:Du} for the last line. These computations and the Poincar\'e inequality
	together yield the following:
	\begin{align*}
		\Var(F_s) &\le 4\left[C_{t,4,\varepsilon,\sigma}A_t\lip(\sigma)\right]^2
			\int_0^s\d r\int_{\R^d}\d w\int_{\R^d}f(\d v)\int_{\R^d}\d y\int_{\R^d}f(\d z)
			\int_{\R^d}\d y'\int_{\R^d}f(\d z')\\
		&\quad\times
			\Pi^{(N)}_{s,y}\Pi^{(N)}_{s,y-z}\Pi^{(N)}_{s,y'}\Pi^{(N)}_{s,y'-z'}
			\bm{p}_{s-r}(y-w)\bm{p}_{s-r}(y'-w+v)\\
		&=4\left[C_{t,4,\varepsilon,\sigma}A_t\lip(\sigma)\right]^2
			\int_0^s\d r\int_{\R^d}f(\d v)\int_{\R^d}\d y\int_{\R^d}f(\d z)
			\int_{\R^d}\d y'\int_{\R^d}f(\d z')\\
		&\quad\times
			\Pi^{(N)}_{s,y}\Pi^{(N)}_{s,y-z}\Pi^{(N)}_{s,y'}\Pi^{(N)}_{s,y'-z'}
			\bm{p}_{2(s-r)}(y'-y+v),
	\end{align*}
	owing to the symmetry and semigroup properties of the heat kernel.
	Since $\Pi^{(N)}\le N^{-d}$ [see \eqref{E:Pi}]
	we may replace three of the four terms that involve the $\Pi$'s
	with $N^{-d}$ in order to find that
	\begin{align*}
	\Var(F_s) &\le \frac{4\left[C_{t,4,\varepsilon,\sigma}A_t\lip(\sigma)\right]^2}{N^{4d}} \int_0^s\d r\int_{\R^d}f(\d v)\\
		  &\quad\times \int_{\R^d}\d y\int_{\R^d}f(\d z) \int_{\R^d}\d y'\int_{\R^d}f(\d z') \int_{[0,N]^d}\d a\ \bm{p}_s(a-y)\bm{p}_{2(s-r)}(y'-y+v)\\
		  &=\frac{4\left[C_{t,4,\varepsilon,\sigma}A_t\lip(\sigma)\right]^2f
		  	\left(\R^d\right)^2}{N^{4d}}\int_0^s\d r\int_{\R^d}f(\d v) \int_{\R^d}\d y' \int_{[0,N]^d}\d a\ \bm{p}_{3s-2r}(a+y'+v)\\
		  &=\frac{4\left[C_{t,4,\varepsilon,\sigma}A_t\lip(\sigma)\right]^2f\left(\R^d\right)^3 s}{N^{4d}}.
	\end{align*}
	Plug this bound into \eqref{E:Var(|v|)} to deduce the result.
\end{proof}

\begin{lemma}\label{L:Var(T,v):g(z)=z}
	There exists a real number $L_2=L_2 (f(\R^d)\,,\lip(\sigma))$
	such that for every $\varepsilon\in(0\,,1)$, $N,t>0$,
	\[
		\Var\left( \left\< \mathcal{T}\,,v_{N,t}(g)\right\>_{\mathcal{H}}\right)
		\le L_2t^3C_{t,4,\varepsilon,\sigma}^2 A_t^2 N^{-3d},
	\]
	where $C_{t,4,\varepsilon,\sigma}$ and $A_t$
	were respectively defined in Lemmas \ref{L:Du} and \ref{L:||v||:g(z)=z}.
\end{lemma}

\begin{proof}
	Choose and fix $N,t>0$ and $\varepsilon\in(0\,,1)$ throughout. It follows from
	\eqref{E:DvT}, the definition of the Hilbert space $\mathcal{H}$,
	and a standard appeal to a stochastic Fubini theorem that
	\begin{align*}
		&\left\< \mathcal{T}\,,v_{N,t}(g)\right\>_{\mathcal{H}}
			= \int_0^t\d s\int_{\R^d}\d y\int_{\R^d}f(\d w)\
			\mathcal{T}(s\,,y)v_{N,t}(g)(s\,,y-w) \\
		&= \int_{(0,t)\times\R^d}\left[\int_0^r\d s\int_{\R^d}\d y\int_{\R^d}f(\d w)\
			\Pi^{(N)}_{s,z} \sigma'(u(r\,,z))
			D_{s,y} u(r\,,z)\Pi^{(N)}_{s,y-w}\sigma(u(s\,,y-w))\right]\eta(\d r\,\d z)\\
		&=: \int_{(0,t)\times\R^d} \mathcal{A}(r\,,z)\,\eta(\d r\,\d z),
	\end{align*}
	notation being clear.
	Among other things, we see from the above representation that the random variable
	$\<\mathcal{T}\,,v_{N,t}(g)\>_{\mathcal{H}}$ has mean zero. Therefore,
	\begin{equation}\label{E:VA}
		\Var\left( \left\< \mathcal{T}\,,v_{N,t}(g)\right\>_{\mathcal{H}}\right)
		= \E\left(\left\< \mathcal{T}\,,v_{N,t}(g)\right\>_{\mathcal{H}}^2\right)\\
		=\int_0^t\d r\int_{\R^d}\d a\int_{\R^d} f(\d b)\
		\E\left[\mathcal{A}(r\,,a)\mathcal{A}(r\,,a-b)\right].
	\end{equation}
	The integrand is bounded as follows:
	\begin{align*}
		&\E\left| \mathcal{A}(r\,,a)\mathcal{A}(r\,,a-b) \right|\\
		&\le
			\int_0^r\d s\int_{\R^d}\d y\int_{\R^d}f(\d w)
			\int_0^r\d s'\int_{\R^d}\d y'\int_{\R^d}f(\d w')\
			\Pi^{(N)}_{s,a} \Pi^{(N)}_{s,y-w}\Pi^{(N)}_{s'\!,a-b}\Pi^{(N)}_{s',y'-w'}\\
		&\quad\times\E\bigg[ \left| \sigma'(u(r\,,a)) D_{s,y} u(r\,,a)\sigma(u(s\,,y-w)) \right|\\
		&\qquad \qquad \quad \times \left|  \sigma'(u(r\,,a-b))
			D_{s'\!,y'} u(r\,,a-b)\sigma(u(s',y'-w'))\right|\bigg]\\
		&\le[\lip(\sigma)]^2\int_0^r\d s\int_{\R^d}\d y\int_{\R^d}f(\d w)
			\int_0^r\d s'\int_{\R^d}\d y'\int_{\R^d}f(\d w')\
			\Pi^{(N)}_{s,a} \Pi^{(N)}_{s,y-w}\Pi^{(N)}_{s'\!,a-b}\Pi^{(N)}_{s',y'-w'}\\
		&\quad\times
			\left\| D_{s,y} u(r\,,a) \right\|_4 \left\| \sigma(u(s\,,y-w)) \right\|_4
			\left\| D_{s'\!,y'} u(r\,,a-b) \right\|_4 \left\| \sigma(u(s',y'-w'))\right\|_4,
	\end{align*}
	thanks to H\"older's inequality.
	Thus, it follows from Lemma \ref{L:Du} and the stationarity of
	$x\mapsto u(s\,,x)$ that
	\begin{align*}
		\E\left| \mathcal{A}(r\,,a)\mathcal{A}(r\,,a-b) \right|
			&\le\left[ A_tC_{t,4,\varepsilon,\sigma}\lip(\sigma)\right]^2
			\int_0^r\d s\int_{\R^d}\d y\int_{\R^d}f(\d w)
			\int_0^r\d s'\int_{\R^d}\d y'\int_{\R^d}f(\d w')\\
		&\quad\times
			\Pi^{(N)}_{s,a} \Pi^{(N)}_{s,y-w}\Pi^{(N)}_{s'\!,a-b}\Pi^{(N)}_{s',y'-w'}
			\bm{p}_{r-s}(a-y) \bm{p}_{r-s'}(a-b-y')\\
		&\le\left[ A_tC_{t,4,\varepsilon,\sigma}\lip(\sigma)\right]^2
			\int_0^r\d s\int_{\R^d}\d y\int_{\R^d}f(\d w)
			\int_0^r\d s'\int_{\R^d}\d y'\int_{\R^d}f(\d w')\\
		&\quad\times
			N^{-2d}\Pi^{(N)}_{s,y-w}\Pi^{(N)}_{s',y'-w'}
			\bm{p}_{r-s}(a-y) \bm{p}_{r-s'}(a-b-y').
	\end{align*}
	This and the semigroup property of the heat kernel yield
	\begin{align*}
		\int_{\R^d}\E\left| \mathcal{A}(r\,,a)\mathcal{A}(r\,,a-b) \right|\d a
			&\le\frac{\left[ A_tC_{t,4,\varepsilon,\sigma}\lip(\sigma)\right]^2}{N^{2d}}
			\int_0^r\d s\int_{\R^d}\d y\int_{\R^d}f(\d w)
			\int_0^r\d s'\int_{\R^d}\d y'\int_{\R^d}f(\d w')\\
		&\hskip.5in\times
			\Pi^{(N)}_{s,y-w}\Pi^{(N)}_{s',y'-w'}
			\bm{p}_{2r-s-s'}(y-b-y').
	\end{align*}
	Apply the simple bound $\Pi^{(N)}_{s',y'-w'}\le N^{-d}$ and first compute the $\d y'$-integral
	to remove the term $\bm{p}_{s-s'}(y-b-y')$, and then the $\d y$-integral to see that
	\[
		\int_{\R^d}\E\left| \mathcal{A}(r\,,a)\mathcal{A}(r\,,a-b) \right|\d a
		\le\frac{\left[ rf(\R^d)A_tC_{t,4,\varepsilon,\sigma}\lip(\sigma)\right]^2}{N^{3d}}.
	\]
	Plug this inequality into \eqref{E:VA} to complete the proof.
\end{proof}

\begin{proof}[Proof of Theorem \ref{T:TV_Bnd}]
	Since $\Var(X+Y)\le 2\Var(X)+2\Var(Y)$
	for all random variables $X$ and $Y$, \eqref{E:K:g(z)=z}, and
	Lemmas \ref{L:||v||:g(z)=z} and \ref{L:Var(T,v):g(z)=z} together yield
	\begin{equation}\label{E:Var(K_N)_g(z)=z}
		\Var(K_N) \le 2(L_1+L_2) t^3 C_{t,4,\varepsilon,\sigma}^2A_t^2 N^{-d}.
	\end{equation}
	Because $|\sigma(z)|\le |\sigma(0)|+\lip(\sigma)|z|$ for all $z\in\R^d$,
	the large-deviations method of Foondun and Khoshnevisan \cite{FoondunKhoshnevisan2013}
	implies that $A_t \le |\sigma(0)| + \lip(\sigma) \sup_{s\in(0,t)}\|u(s\,,0)\|_4
	\le a\exp(bt)$ for all $t>0$, and for strictly positive numbers $a,b$ that depend
	only on $(f\,,\sigma)$.
	And Lemma \ref{L:Du} shows that $C_{t,4,\varepsilon,\sigma}\le c\exp(dt)$
	for all $t>0$ and positive numbers $c$ and $d$ that depend only on $(f\,,\sigma)$.
	This proves immediately the existence of positive real numbers $\ell$ and $\lambda$
	--- depending only on $(f\,,\sigma)$ --- such that
	\[
		\Var(K_N) \le \frac{\ell\e^{2\lambda t}}{N^d}.
	\]
	It is easy to see that $d_{\rm TV}(\alpha X\,,\alpha Y)$ does not depend on
	the numerical value of $\alpha >0$ for all random variables $X$ and $Y$. Therefore,
	as long as $\mathbf{B}_{N,t}(g)>0$, \eqref{E:Var_N} implies that
	\begin{align*}
		d_{\rm TV} \left( \frac{\mathcal{S}_{N,t}(g)}{\sqrt{\Var(\mathcal{S}_{N,t}(g))}}
			~,~{\rm N}(0\,,1) \right)
			& = d_{\rm TV}\left( \frac{N^{d/2} \mathcal{S}_{N,t}(g)}{
		 	\sqrt{\mathbf{B}_{N,t}(g)}}\,,Z\right)\\
		&= d_{\rm TV}\left( N^{d/2}\mathcal{S}_{N,t}(g)\,,\sqrt{\mathbf{B}_{N,t}(g)}\, Z\right)
			\le \frac{2}{\mathbf{B}_{N,t}(g)}\E\left| K_N-\E(K_N)\right|;
	\end{align*}
	see \eqref{E:K} for the last inequality. Apply the Cauchy-Schwarz inequality
	[see also \eqref{E:SM3}] and the preceding
	upper bound for $\Var(K_N)$ to complete the proof with $L=2\sqrt{\ell}$,
	since $\lim_{N\to\infty}\mathbf{B}_{N,t}(g)=\mathbf{B}_t(g)$ implies that
	indeed $\mathbf{B}_{N,t}(g)>0$ for all sufficiently large $N>0$.
\end{proof}

\section{Parabolic Anderson Model: Proof of Theorem \ref{T:PAM}}\label{S:PAM}

We begin by trying to mimic the proof of Theorem \ref{T:TV_Bnd},
except we do not assume that $g(v)=v$ for all $v\in\R^d$.

Recall $K_N$ from \eqref{E:KN}  and note that
\[
	\Var( K_N ) = N^{2d} \Var\left(
	 \int_0^t \d s \int_{\R^d} \d y\
	 D_{s,y} \mathcal{S}_{N,t}(g)
		\left( \E[ D_{s,\bullet} \mathcal{S}_{N,t}(g)  \mid \mathcal{F}_s]*f \right)(y) \right).
\]
Therefore, we may take square roots and appeal to Minkowski's inequality [cf.\
\eqref{E:Var(|v|)}] in order to see that
\begin{equation}\label{E:ecu3}
	\sqrt{\Var(K_N)} \le N^d \int_0^t   \d s\,
	\sqrt{\Var\left(  \left\<   \E\left[ D_{s,\bullet } \mathcal{S}_{N,t} (g)
	\mid  \mathcal{F}_s\right] ~,~ D_{s,\bullet} \mathcal{S}_{N,t}(g) \right\>_{\mathcal{H}_0}
	\right)}.
\end{equation}
We have seen already [see \eqref{E:ec2} and \eqref{E:Du}] that
\begin{align} \label{E:DS2}
	D_{s,y} \mathcal{S}_{N,t}(g)= N^{-d} \int_{[0,N]^d}  g'(u(t\,,x))  D_{s,y} u(t\,,x)\, \d x,
.\end{align}
where $Du(t\,,x) $ satisfies the linear stochastic integral equation,
\[
	D_{s,y} u(t\,,x)  = \bm{p}_{t-s}(x-y) u(s\,,y) +
	\int_{(s,t)\times\R^d}\bm{p}_{t-r}(x-z)  D_{s,y} u(r\,,z)
	\,\eta (\d r\,\d z).
\]
See \eqref{E:mild}.
It follows form these remarks that
\begin{align*}
	D_{s,y} \mathcal{S}_{N,t}(g)&=  u(s\,,y) \,
		N^{-d}\int_{[0,N]^d}  g'(u(t\,,x))\bm{p}_{t-s}(x-y) \,\d x \\
	&\qquad + N^{-d} \int_{ [0,N]^d}  g'(u(t\,,x))
		\left[\int_{(s,t)\times\R^d}\bm{p}_{t-r}(x-z)   D_{s,y} u(r\,,z)
		\,\eta (\d r\,\d z)\right]\d x.
\end{align*}
This is more or less \eqref{E:DvT}, except that: (i) $g$ is not necessarily linear;
(ii) $\sigma(z)=z$ for all $z\in\R$ here; and most significantly (iii) because of
the presence of the non-anticipative term $g'(u(t\,,x))$, we
can no longer interchange the $\d x$-integral with the $\eta(\d r\,\d z)$-stochastic integral.
It is in fact because of item (iii) that we will need to adopt a different way to estimate
$\Var(K_N)$ in the proof of
Theorem \ref{T:TV_Bnd}. And at present the only other proof that we can construct
works only when $g\in C^2(0,\infty)$ satisfies \eqref{E:MmCond_PAM} and $\sigma(z)=z$ for all $z\in\R$,
whence follow the conditions of Theorem \ref{T:PAM}.

According to \eqref{E:ecu3} and the proof of Theorem \ref{T:TV_Bnd},
the key step of the argument is to find an efficient bound for
\[
	\mathcal{V}(s) =\mathcal{V}_{N,t,g}(s) :=
	{\rm Var} \left(  \left\< \E \left[ D_{s,\bullet } \mathcal{S}_{N,t} (g)
	\mid \mathcal{F}_s \right]
	~,~ D_{s,\bullet} \mathcal{S}_{N,t}(g) \right\>_{\mathcal{H}_0} \right)
	\qquad[0<s<t].
\]
However, as was alluded to earlier in this section, the above estimation will need
to involve other ideas than those that were used in the course of the proof of
Theorem \ref{T:TV_Bnd}. Indeed, we plan to estimate $\mathcal{V}(s)$
by appealing to the Poincar\'e inequality \eqref{E:Poincare}.
In order for this method to work, we will need to assume
that $g\in C^2(0,\infty)$ satisfies  \eqref{E:MmCond_PAM}.

In the following, we abbreviate $\mathcal{S}_{N,t} (g)$ by $\mathcal{S}_{N}$
in order to simplify the typesetting. With this convention in mind,
the Poincar\'e inequality yields
\begin{equation} \label{E:EQ7}
	\mathcal{V}(s)
	\le \E \left( \left\|
	D  \left\<\E \left[ D_{s,\bullet } \mathcal{S}_N
	\mid \mathcal{F}_s \right] ~,~
	D_{s,\bullet} \mathcal{S}_N \right\>_{\mathcal{H}_0}  \right\|^2_{\mathcal{H}} \right).
\end{equation}
Now,
\begin{align*}
	D_{r,z}\left\<\E\left[ D_{s,\bullet } \mathcal{S}_N
		\mid \mathcal{F}_s\right] ~,~
		D_{s,\bullet} \mathcal{S}_N \right\>_{\mathcal{H}_0}
	= &\quad \int_{\R^d}\d y\int_{\R^d} f(\d y_1)
		\E\left[ D_{s,y } \mathcal{S}_N \mid \mathcal{F}_s \right]
		D_{r,z} D_{s,y-y_1} \mathcal{S}_N \\
	  & + \int_{\R^d}\d y\int_{\R^d} f(\d y_1)\
		D_{s,y } \mathcal{S}_N\  D_{r,z}
		\E\left[ D_{s,y-y_1 } \mathcal{S}_N \mid \mathcal{F}_s\right]\\
	=: &\quad  \Psi^{N,1}_{r,z} + \Psi^{N,2} _{r,z}.
\end{align*}
By Proposition 1.2.8 of Nualart \cite{Nualart}, one can exchange the
Malliavin derivative and the conditional expectation so that
\[
	\Psi^{N,2} _{r,z} =  \int_{\R^d}\d y\int_{\R^d} f(\d y_1)\
	D_{s,y } \mathcal{S}_N\
	\E\left[ D_{r,z}D_{s,y-y_1 } \mathcal{S}_N \mid
	\mathcal{F}_s\right].
\]
Therefore, it follows  that
\begin{equation} \label{E:EQ8}
	\mathcal{V}(s) \le 2 \E\left(  \left\| \Psi^{N,1} \right\|^2_{\mathcal{H}} \right) +
	2\E \left( \left\| \Psi^{N,2} \right\|^2_{\mathcal{H}} \right).
\end{equation}
Both expectations can be estimated in the same way, and both estimates will yield
the same upper bound. Therefore,
we present only the details of the estimate for the first expectation.

We can write
\begin{align*}
	\E\left(\left\| \Psi^{N,1}\right\|^2_{\mathcal{H}} \right)
	& = \int_0^t\d r\int_{\R^d}\d z\int_{\R^d}\d y\int_{\R^d}\d y'
		\int_{\R^d}f(\d z')\int_{\R^d}f(\d y_1)\int_{\R^d}f(\d y'_1)\\
	& \quad \times
		\E \bigg[ \quad\E \left[ D_{s,y } \mathcal{S}_N \mid
		\mathcal{F}_s \right]  D_{r,z}
		D_{s,y-y_1}\mathcal{S}_N\\
	& \qquad\quad \times
		\E \left[ D_{s,y' } \mathcal{S}_N \mid
		\mathcal{F}_s \right]  D_{r,z-z'} D_{s,y'- y'_1}\mathcal{S}_N  \bigg].
\end{align*}
Then an application of H\"older inequality shows that
\begin{align}\begin{aligned} \label{E:PsiN}
	\E\left(\left\| \Psi^{N,1}\right\|^2_{\mathcal{H}} \right)
	&  \le\int_0^t\d r\int_{\R^d}\d z\int_{\R^d}\d y\int_{\R^d}\d y'
		\int_{\R^d}f(\d z')\int_{\R^d}f(\d y_1)\int_{\R^d}f(\d y'_1)\\
	& \quad \times
		\left\| D_{s,y } \mathcal{S}_N \right\|_4 \,
		\left\| D_{r,z} D_{s ,y-y_1}\mathcal{S}_N \right\|_4\,
		\left\| D_{s,y' } \mathcal{S}_N \right\|_4 \,
		\left\| D_{r,z-z'} D _{s,y'- y'_1}\mathcal{S}_N \right\|_4 .
\end{aligned}\end{align}
According to Lemma \ref{L:Du} and H\"older's inequality,
\[
	\|  D_{s,y } \mathcal{S}_N \|_4
	\le  \| g'(u(t,0)) \|_k \frac  { 1} {N^{d}} \int_{[0,N]^d}
	\left\| D_{s,y} u(t\,,x)\right\|_{\frac {4k}{k-4}} \d x \\
	\le C_{t,\frac {4k}{k-4},\varepsilon}  \| g'(u(t,0)) \|_k  \Pi^{(N)}_{s,y},
\]
where $\varepsilon\in(0\,,1)$ is arbitrary and $\Pi^{^(N)}$ was
defined in \eqref{E:Pi}. Next we apply \eqref{E:ec2}
in order to see that
\begin{align}\label{E:DDS}\begin{aligned}
	D_{r ,z}D_{s,y} \mathcal{S}_N  = & N^{-d} \int_{[0,N]^d} g'(u(t\,,x)) D_{r,z} D_{s,y} u(t\,,x)\,\d x \\
					 & + N^{-d} \int_{[0,N]^d} g''(u(t\,,x)) D_{r,z} u(t\,,x) \, D_{s,y} u(t\,,x)\, \d x.
\end{aligned}\end{align}
Therefore,  Proposition \ref{P:secondD}, Lemma \ref{L:Du}, and the notation for $\Pi^{(N)}$
in \eqref{E:Pi} together yield
\begin{align*}
	\| D_{r,z} D_{s ,y}\mathcal{S}_N \|_4
	&\le C^*_{t,\frac{4k}{k-4},\varepsilon}   \| g'(u(t,0)) \|_k
		\left[ \bm{p}_{s-r} (y-z)  \mathbf{1}_{\{r\le s\}}
		 \Pi^{(N)}_{s,y}
		 +\bm{p}_{r-s} (z-y) \mathbf{1}_{\{s\le r\}}
		\Pi^{(N)}_{r,z}\right]\\
	&\quad+  \frac {C_{t,\frac {8k}{k-4},\varepsilon, \sigma}^2  \| g''(u(t,0)) \|_k } {N^d}
		\int_{[0,N]^d}  \bm{p}_{t-r} (x-z) \bm{p}_{t-s} (x-y)\,\d x\\
	&\le\frac {C^*_{t,\frac {4k} {k-4},\varepsilon}  \| g'(u(t,0)) \|_k }{N^d}
		\bm{p}_{|s-r|} (y-z) +  \frac{C_{t,\frac {8k}{k-4},\varepsilon, \sigma}^2    \| g''(u(t,0)) \|_k}{N^d}
		\bm{p}_{2t-r-s} (y-z),
\end{align*}
owing to the semigroup property of the heat kernel and Rademacher's theorem
(Remark \ref{R:Lip}), where the constant $C^*_{t,4,\epsilon}$ is defined in \eqref{E:Cstar}.
Thus,
\begin{align*}
	\E\left(\left\|  \Psi^{N,1} \right\|^2_{\mathcal{H}} \right) \le
	& C_1 \| g'(u(t,0)) \|_k ^2\\
	&\times \int_0^t\d r\int_{\R^d}\d z\int_{\R^d}\d y\int_{\R^d}\d y'
		\int_{\R^d}f(\d z')\int_{\R^d}f(\d y_1)\int_{\R^d}f(\d y'_1)\
		\Pi^{(N)}_{s,y}\Pi^{(N)}_{s,y'}\\
	&\times \left(
		\frac {C_2  \| g'(u(t,0)) \|_k }{N^d}
		\bm{p}_{|s-r|} (y-y_1-z)    +  \frac {C_3  \| g''(u(t,0)) \|_k}{N^d}
		\bm{p}_{2t-r-s} (y-y_1-z)  \right) \\
	&\times \left(
		\frac {C_2\| g'(u(t,0)) \|_k}{N^d}
		\bm{p}_{|s-r|} (y'-y'_1-z+z')    +  \frac {C_3  \| g''(u(t,0)) \|_k}{N^d}
		\bm{p}_{2t-r-s} (y'-y'_1-z+z')  \right),
\end{align*}
where $C_1=C^2_{t,\frac {4k}{k-4},\varepsilon}  $,
$C_2=C^*_{t,\frac {4k}{k-4},\varepsilon} $ and $C_3=C_{t,\frac {8k}{k-4},\varepsilon, \sigma}^2$.
We apply the bound $\Pi^{(N)}_{s,y'}\le N^{-d}$ and  integrate [first
$\d y'$, then $\d z$, and then finally $\d y$, and in this order],
using also the inequality $\int_{\R^d}\Pi^{(N)}_{s,y}\,\d y\le1$, in order to obtain
\begin{align} \label{E_:BoundPhi}
	\E  \left(\| \Psi^{N,1}\|^2_{\mathcal{H}}\right) \le  C_4
	[f(\R^d)]^3t N^{-3d},
\end{align}
where
\begin{align} \label{E:C_1}
	C_4= C_1 \| g'(u(t,0)) \|_k^2
	\left[ C_2 \| g'(u(t,0)) \|_k+ C_3  \| g''(u(t,0)) \|_k\right]^2.
\end{align}
The very same method applies and yields the same upper bound for
$\E(\|  \Psi^{N,2} \|^2_{\mathcal{H}} )$.
Therefore, we may deduce from (\ref{E:EQ7}) and (\ref{E:EQ8}) that
$\mathcal{V}(s) \le 4C_4 t N^{-3d}[f(\R^d)]^3.$
We plug this estimate into \eqref{E:ecu3} to obtain the following:
\begin{align}\label{E:VarKn}
	\sqrt{\Var(K_N)} \le N^d \int_0^t
	\sqrt{\mathcal{V}(s)}\,\d s
	\le 2C_4^{1/2} t^{3/2} [f(\R^d)]^{3/2} N^{-d/2}.
\end{align}
This is the replacement of the inequality \eqref{E:Var(K_N)_g(z)=z}
in the present setting of the parabolic Anderson model; the remainder of the proof
of Theorem \ref{T:TV_Bnd} goes through unhindered to complete
the proof of  Theorem \ref{T:PAM}.

\appendix

\section{Association for parabolic SPDEs}

We conclude the paper by showing how we can adapt the ideas of this paper to also
document a result about the association of the solution to \eqref{E:SHE}. First,
let us recall the following.

\begin{definition}[Esary, Proschan, and Walkup \cite{EPW}]
	A random vector $X:=(X_1\,,\ldots,X_n)$ is said to be \emph{associated} if
	\begin{equation}\label{E:assoc}
		\Cov[h_1(X)\,,h_2(X)]\ge0,
	\end{equation}
	for every pair of functions $h_1,h_2:\R^n\to\R$
	that are nondecreasing in every coordinate and satisfy $h_1(X),h_2(X)\in L^2(\Omega)$.
\end{definition}

It is easy to see that $h_1$ and $h_2$ could instead be assumed to be
nonincreasing in every coordinate: One simply replaces $(h_1\,,h_2)$ by $(-h_1\,,-h_2)$.
Other simple variations abound. Esary, Proschan, and Walkup
\cite{EPW} prove the  more interesting result that $X$ is associated iff
\eqref{E:assoc} holds for all $h_1,h_2\in C_b(\R^n)$ that are nondecreasing coordinatewise.\footnote{%
	Recall that $C_b(\R^n)$ denotes the collection of all real functions on $\R^n$
	that are bounded and continuous.}
This and the dominated convergence
theorem together imply the following simple criterion for association; see also Pitt \cite{Pitt}.

\begin{lemma}
	$X=(X_1\,,\ldots,X_n)$ is associated iff \eqref{E:assoc} holds for all $h_1,h_2\in C^\infty_c(\R^n)$
	such that $\partial_i h_j\ge0$ for every $i=1,2$ and $j=1,\ldots,n$.\footnote{Recall that
	the subscript ``$c$'' in $C^\infty_c$ refers to functions of compact support.}
\end{lemma}

The notion of association has a natural extension to random processes. Here is one.

\begin{definition}
	A random field $\Phi=\{\Phi(t\,,x)\}_{t\ge0,x\in\R^d}$ is
	\emph{associated} if $(\Phi(t_1,x_1)\,,\ldots,\Phi(t_n\,,x_n))$ is associated
	for every $t_1,\ldots,t_n>0$ and $x_1,\ldots,x_n\in\R^d$.
\end{definition}

The following is the principal result of this appendix.

\begin{theorem}\label{T:assoc}
	Let $u$ be the solution to \eqref{E:SHE} with $u(0,\cdot)$ being a nonnegative measure such that $\left(u(0,\cdot)*\bm{p}_t\right)<\infty$ for all $t>0$.
	Under Dalang's condition \eqref{E:Dalang}, if $\sigma\circ u$ a.s.\ does not change signs; that is,
	\begin{equation}\label{E:sigma_sign}
		\P\left\{ \sigma(u(t\,,x))\sigma(u(s\,,y))\ge0\right\}=1
		\quad\text{for every $s,t>0$ and $x,y\in\R^d$,}
	\end{equation}
	then, $u$ is associated.
\end{theorem}

Theorem \ref{T:assoc} reduces to   simpler results in  two important
special cases that we present in the following two examples.

\begin{example}
	Condition \eqref{E:sigma_sign}  holds
	tautologically when $\sigma$ is a constant, say $\sigma\equiv\sigma_0\in\R$.
	In that case, it is not hard to see that the solution to \eqref{E:SHE}
	is a Gaussian random field with  covariance function,
	\begin{align*}
		\Cov[u(t\,,x)\,,u(t',x')] = \sigma_0^2\int_0^t\left( \bm{p}_{2s}*f\right)(x-x')\,\d s.
	\end{align*}
	Because this quantity is $\ge0$ for every $t,t'>0$ and $x,x'\in\R^d$, Pitt's characterization
	of Gaussian random vectors
	\cite{Pitt} immediately implies the association of $u$ when $\sigma$ is a constant.
\end{example}

\begin{example}
	If $\sigma(u)=u$,  $d=1$, and $f=\delta_0$, then the sign condition \eqref{E:sigma_sign}
	holds simply because $u(t\,,x)\ge0$ a.s. In this case, Corwin and
	Quastel \cite[Proposition 1]{CorwinQuastel13} have observed that $u$ is associated;
	see also Corwin and Ghosal \cite[Proposition 1.9]{CorwinGhosal18}. Strictly speaking,
	the statements of the latter two results are weaker that the association of $u$. But the proof
	of Proposition 1.9 of Corwin and Quastel ({\it ibid.}) goes through unhindered to imply
	association. Indeed, Corwin and Quastel note the said association from
	the well-known fact that all exclusion processes are associated
	(the FKG inequality for i.i.d.\ random variables).
\end{example}

The preceding examples are proved using two different ``direct hands-on'' methods. The general
case is quite a bit more involved. In order to avoid writing a lengthy proof, we merely outline
the proof, with detailed pointers to the requisite literature and preceding results of the present paper.

\begin{proof}[Outline of the proof of Theorem \ref{T:assoc}]
	By considering $-\sigma$ in place of $\sigma$ if need be, we may -- and will --
	assume without loss in generality that
	\[
		\sigma(u(t\,,x))\ge0\qquad\text{a.s.\ for every $(t\,,x)\in(0\,,\infty)\times\R^d$}.
	\]

	We begin by rehashing through the proof of Theorem \ref{T:Positive_D}, and recall that
	the Malliavin derivative of $u(t\,,x)$ a.s.\ solves
	\begin{equation}\label{D:2}
		D_{s,z}u(t\,,x) = \bm{p}_{t-s}(x-z) \sigma(u(s\,,z)) +
		\int_{\left(s,t\right)\times\R^d} \bm{p}_{t-r}(x-y)
		\sigma'(u(r\,,y)) D_{s,z} u(r\,,y)\,\eta(\d r\,\d y),
	\end{equation}
	for almost all $(s\,,z)\in(0\,,t)\times\R^d$. And, of course, $D_{s,z}u(t\,,x)=0$
	a.s.\ for a.e.\ $(s\,,z)\in(t\,,\infty)\times\R^d$.
	Note that the right-hand side of \eqref{D:2} is in fact a modification of $D_{s,z}u(t\,,x)$.
	Therefore, we can -- and will without loss of generality -- redefine $D_{s,z}u(t\,,x)$ so that
	\eqref{D:2} is an almost-sure identity for every $(s\,,z)\in(0\,,t)\times\R^d$, and that
	$D_{s,z}u(t\,,x):=0$ whenever $s\ge t$. It is not hard to deduce from
	\eqref{D:2}, and arguments involving continuity in probability,
	that the use of this particular modification does not change
	the law of the Malliavin derivative. However, it makes the following discussion simpler
	as we can avoid worrying about null sets.

	We aim to prove that,
	under Dalang's condition \eqref{E:Dalang} alone,
	\begin{equation}\label{D:ge:0}
		D_{s,z}u(t\,,x)\ge0\qquad\text{a.s.\ for every $(t\,,x)\in(0\,,\infty)\times\R^d$}.
	\end{equation}
	Once this is proved, one can appeal to the Clark-Ocone formula
	\eqref{E:Clark-Ocone} to conclude the theorem;
	see the proof of Theorem \ref{P:associate} for a similar argument.

	Choose and fix  $(s\,,z)\in(0\,,\infty)\times\R^d$, and consider a space-time random field
	$\Phi_{s,z}$ defined via
	\[
		\Phi_{s,z}(t\,,x) := D_{s,z}u(s+t\,,z+x)\qquad\text{for all $t\ge0$ and $x\in\R^d$}.
	\]
	Our claim \eqref{D:ge:0} is equivalent to the a.s.-nonnegativity of $\Phi_{s,z}(t\,,x)$.

	Observe (similarly to what we did in the proof of Theorem \ref{T:Positive_D}), that
	$\Phi_{s,z}$ solves
	\[
		\Phi_{s,z}(t\,,x) = \bm{p}_t(z+x)\sigma(u(s\,,z))+
		\int_{(0,t)\times\R^d} \bm{p}_{t-r}(x-y)
		H_{s,z}(t\,,y)\Phi_{s,z} (r\,,y)\,\eta_{s,z}(\d r\,\d y),
	\]
	for  $(t\,,x)\in(0\,,\infty)\times\R^d$, where $\eta_{s,z}$ is the same shifted
	Gaussian noise that arose in the course of the proof of Theorem \ref{T:Positive_D},
	and
	\[
		H_{s,z}(t\,,y) = \sigma'\left( u(s+t\,,z+y)\right),
	\]
	also as in the proof of Theorem \ref{T:Positive_D}.
	In other words, conditional on the sigma-algebra $\mathcal{F}_s$,
	the random field $\Phi_{s,z}$ solves the SPDE,
	\begin{equation}\label{E:Phi_SPDE}
		\partial_t\Phi_{s,z} = \tfrac12\Delta\Phi_{s,z} + \Phi_{s,z} H_{s,z}\eta_{s,z}
		\quad\text{subject to $\Phi_{s,z}(0\,,x) =\sigma(u(s\,,z))\delta_0(x)$.}
	\end{equation}
	This is a standard Walsh-type SPDE.
	In order to understand why this is the case, note first that $u(s\,,z)$ is $\mathcal{F}_s$-measurable
	and $\eta_{s,z}$ is independent of $\mathcal{F}_s$; see also the proof of Lemma
	\ref{T:Positive_D}. In order to complete the assertion about \eqref{E:Phi_SPDE} being a standard
	Walsh-type SPDE we note that
	$M_t(A) := \int_{(0,t)\times A}H_{s,z}(r\,,y)\,\eta_{s,z}(\d r\,\d y)$
	defines a martingale measure in the sense of Walsh. Hence, we view
	\eqref{E:Phi_SPDE} conditionally on $\mathcal{F}_s$ to see that it is a reformulation
	of the Walsh SPDE,
	\[
		\partial_t \Phi_{s,z} = \tfrac12\Delta\Phi_{s,z} + \Phi_{s,z}\dot{M}_t(x)
		\quad\text{subject to}\quad \Phi_{s,z}(0\,,x)=\sigma(u(s\,,z))\delta_0(x),
	\]
	also viewed conditionally on $\mathcal{F}_s$. In fact, we can interpret
	\eqref{E:Phi_SPDE} as a parabolic Anderson model with respect to the martingale measure $M$.
	Because of \eqref{E:f_finite}, and since $\sigma\in\lip$, $M$ is in fact a worthy
	martingale measure because $H_{s,z}$ is bounded; see Walsh \cite{Walsh}.

	For every $\varepsilon\in(0\,,1)$ let $\eta_{s,z}^\varepsilon$ denote the Gaussian noise
	\[
		\eta_{s,z}^\varepsilon(t\,,x)\,\d t := \int_{\R^d}\bm{p}_\varepsilon(x-y)\,\eta_{s,z}(\d t\,\d y).
	\]
	The spatial correlation measure $f_\varepsilon$ of $\eta_{s,z}^\varepsilon$ is
	in fact a function [not just a measure]
	and is given by $f_\varepsilon=\bm{p}_{2\varepsilon}*f.$ Because $f$ is a finite measure,
	it follows that $f_\varepsilon$ satisfies the extended Dalang condition \eqref{E:Dalang2} with
	$\alpha=1$. Let $\Phi_{s,z}^\varepsilon$ denote the unique solution to the SPDE,
\[
		\partial_t \Phi_{s,z}^\varepsilon = \tfrac12 \Delta\Phi_{s,z}^\varepsilon +
		(\sigma'\circ u)\Phi_{s,z}^\varepsilon\eta^{\varepsilon}_{s,z}
		\quad\text{subject to}\quad
		\Phi_{s,z}(0\,,x)=\sigma(u(s\,,z))\delta_0(x).
	\]
	This is the same SPDE as \eqref{E:Phi_SPDE}, but with respect to the mollified
	version $\eta_{s,z}^\varepsilon$ of the Gaussian noise $\eta_{s,z}$ in place
	of $\eta_{s,z}$.
	Because $f_\varepsilon$ satisfies \eqref{E:Dalang2} for some $\alpha>0$, Theorem 1.6
	of Chen and Huang \cite{CH19} [applied conditionally on $\mathcal{F}_s$]
	implies that, if $H_{s,z}$ were
	a constant [that is, $\sigma(u)\propto u$], then $\Phi_{s,z}^\varepsilon(t\,,x)>0$ a.s.\
	for every $t>0$ and $x\in\R^d$. Careful scrutiny of the proof of Theorem 1.6
	of \cite{CH19}, using the boundedness of $H_{s,z}$,
	yields that $\Phi_{s,z}^\varepsilon>0$ a.s.\ in all cases; see also Theorem 1.8
	of Chen and Huang \cite{CH19Density}.
	Finally, we let $\varepsilon\downarrow0$ and appeal to Theorem 1.9
	of Chen and Huang \cite{CH19} to see that $\lim_{\varepsilon\to0}\Phi_{s,z}^\varepsilon(t\,,x)=
	\Phi_{s,z}(t\,,x)$ in $L^2(\Omega)$ for all $(t\,,x)\in(0\,,\infty)\times\R^d$. These facts,
	together imply $\Phi_{s,z}(t\,,x)\ge0$ a.s.\ for all $(t\,,x)\in(0\,,\infty)\times\R^d$,
	which is another way to write \eqref{D:ge:0}. This concludes the proof.
\end{proof}

\end{document}